%% file: ktheory.tex
\documentclass[a4paper,twoside,11pt,english,intlimits]{article}

\usepackage{amsmath,amsthm,amsfonts,amssymb}
\usepackage{babel}
\usepackage[T1]{fontenc}
\usepackage{times,euler,euscript,stmaryrd}
\usepackage{fancyhdr,titlesec,url,enumerate}
\usepackage[all]{xy}
\usepackage{hyperref}
\usepackage{catex}

\CompileMatrices

\DeclareMathAlphabet{\mathscr}{T1}{pzc}{m}{it}

\titleformat{\section}[block]{\scshape\filcenter\Large}{\thesection.}{.5em}{}
\titleformat{\subsection}[block]{\bfseries\filcenter\large}{\thesubsection.}{.5em}{\medskip}
\titleformat{\subsubsection}[runin]{\bfseries}{\thesubsubsection.}{.5em}{}[.]

\titlespacing{\subsubsection}{0pt}{10pt}{.5em}

\newtheoremstyle{ntheorem}%
	{\topsep}{\topsep}{\itshape}{0pt}{\bfseries}{.}{.5em}%
	{\thmnumber{#2.\hspace{.5em}}\thmname{#1}\thmnote{ (#3)}}
	
\newtheoremstyle{ndefinition}%
	{\topsep}{\topsep}{\normalfont}{0pt}{\bfseries}{.}{.5em}%
	{\thmnumber{#2.\hspace{.5em}}\thmname{#1}\thmnote{ (#3)}}
	
\newtheoremstyle{nremark}%
	{\topsep}{\topsep}{\normalfont}{0pt}{\itshape}{.}{.5em}%
	{\thmnumber{}\thmname{#1}\thmnote{ (#3)}}

\theoremstyle{ntheorem}
  	\newtheorem{theorem}[subsubsection]{Theorem}
  	\newtheorem{proposition}[subsubsection]{Proposition}
	\newtheorem{lemma}[subsubsection]{Lemma}
  	\newtheorem{corollary}[subsubsection]{Corollary}

\theoremstyle{ndefinition}

	\newtheorem{example}[subsubsection]{Example}

\theoremstyle{nremark}
	\newtheorem{remark}{Remark}

\fancypagestyle{plain}{
\fancyhf{}\fancyfoot[LC,RC]{}
\fancyfoot[LE,RO]{$\thepage$}

}

\input{macroPY}

\newcommand{\pdf}[1]{\texorpdfstring{$#1$}{1}}

\def\catego#1{{\bf{\sf #1}}}

\renewcommand{\Ab}{\catego{Ab}}
\renewcommand{\Cat}{\catego{Cat}}
\renewcommand{\mod}{\catego{Mod}}
\renewcommand{\Mod}{\catego{Mod}}

\newcommand{\Nat}{\catego{Nat}}

\newcommand{\Epi}{\mathbf{Epi}}

\def\op#1{#1^{o}}
\def\iar#1{\lfloor #1 \rfloor}
\def\cl#1{\overline{#1}}

\renewcommand{\Nb}{\mathbb{N}}
\renewcommand{\Zb}{\mathbb{Z}}

\newcommand{\C}{\mathbf{C}}
\newcommand{\D}{\mathbf{D}}
\newcommand{\F}{\mathbf{F}}
\newcommand{\G}{\mathbf{G}}
\newcommand{\M}{\mathbf{M}}

\DeclareMathOperator{\N}{N}

\newcommand{\As}{\mathrm{As}}

\newcommand{\fnat}[2]{F_{#1}[#2]}

\newcommand{\tck}[1]{#1^{\top}}
\newcommand{\ab}[1]{#1_{\text{ab}}}
\newcommand{\abtck}[1]{\ab{\tck{#1}}}

\newcommand{\poldim}[1]{d_{\mathrm{pol}}(#1)}
\newcommand{\cohdim}[1]{\mathrm{cd}(#1)}

\renewcommand{\tilde}[1]{\widetilde{#1}}

\newcommand{\rep}[1]{\widehat{#1}}

\renewcommand{\Ct}[1]{\mathrm{Ct}(#1)}

\DeclareMathOperator{\h}{h}
\DeclareMathOperator{\crit}{c}
\DeclareMathOperator{\Aut}{Aut}
\DeclareMathOperator{\FDT}{FDT}
\DeclareMathOperator{\FDTAB}{FDT_{\mathrm{ab}}}
\DeclareMathOperator{\FP}{FP}
\DeclareMathOperator{\Lan}{Lan}

\newcommand{\sm}{\scriptstyle}
\newcommand{\ifl}{\rightarrowtail}
\newcommand{\pfl}{\twoheadrightarrow}

\newcommand{\defmap}[5]{
\begin{array}{c c c c c}
#1 & \::\: & #2 & \:\longrightarrow\: & #3
\\ && #4 &\longmapsto& #5
\end{array}
}

\newcommand{\qfl}{\xymatrix@1@C=10pt{\ar@4 [r] &}}

\deftwocell[black]{mu : 2 -> 1}
\deftwocell[circle,black]{eta : 0 -> 1}
\deftwocell[gray]{alpha : 3 -> 1}
\deftwocell[gray,lefthalfcircle]{lambda : 1 -> 1}
\deftwocell[gray,righthalfcircle]{rho : 1 -> 1}
\deftwocell[white]{aleph : 4 -> 1}
\deftwocell[white]{beth : 2 -> 1}
\deftwocell[crossing]{tau : 2 -> 2}
\deftwocell[crossing]{yb : 3 -> 3}
\deftwocell[crossing]{yb2 : 4 -> 4}

\begin{document}

\thispagestyle{empty}
\begin{center}

\begin{Large}
\begin{uppercase}
{Higher-dimensional normalisation strategies \\ \medskip for acyclicity}
\end{uppercase}
\end{Large}

\bigskip
\hrule height 1.5pt

\bigskip
\begin{large}
\begin{uppercase}
{Yves Guiraud -- Philippe Malbos}
\end{uppercase}
\end{large}

\end{center}

\vspace*{\stretch{2}}
\begin{center}
\begin{minipage}{12cm}
\begin{small}
\noindent \textbf{Abstract --}
We introduce \emph{acyclic polygraphs}, a notion of complete categorical cellular model for (small) categories, containing generators, relations and higher-dimensional globular syzygies. We give a rewriting method to construct explicit acyclic polygraphs from \emph{convergent} presentations. For that, we introduce \emph{higher-dimensional normalisation strategies}, defined as homotopically coherent ways to relate each cell of a polygraph to its normal form, then we prove that acyclicity is equivalent to the existence of a normalisation strategy. Using acyclic polygraphs, we define a higher-dimensional homotopical finiteness condition for higher categories which extends Squier's finite derivation type for monoids. We relate this homotopical property to a new homological finiteness condition that we introduce here.


\bigskip
\noindent \textbf{M.S.C. 2000 --} 18C10; 18D05; 18G10; 18G20; 68Q42.

\bigskip
\noindent \textbf{Affiliation --} INRIA and Université de Lyon, Institut Camille Jordan, CNRS UMR 5208, Université Claude Bernard Lyon~1, 43 boulevard du 11 novembre 1918, 69622 Villeurbanne cedex, France.
\end{small}
\end{minipage}
\end{center}

\vspace*{\stretch{2}}
%%%%%%%%%%%%%%%%%%%%%%%%%%%%%%%%%%%%%%%%%%%%%%%%%%%%%
%%%%%%%%%%%%%%%%%%%%%%%%%%%%%%%%%%%%%%%%%%%%%%%%%%%%%
\begin{center}
\begin{minipage}{12cm}
\begin{small}
\setcounter{tocdepth}{1}
\tableofcontents
\end{small}
\end{minipage}
\end{center}
%%%%%%%%%%%%%%%%%%%%%%%%%%%%%%%%%%%%%%%%%%%%%%%%%%%%%
%%%%%%%%%%%%%%%%%%%%%%%%%%%%%%%%%%%%%%%%%%%%%%%%%%%%%
\vspace*{\stretch{1}}

%%%%%%%%%%%%%%%%%%%%%%%%%%%%%%%%%%%%%%%%%%%%%%%%%%%%%
%%%%%%%%%%%%%%%%%%%%%%%%%%%%%%%%%%%%%%%%%%%%%%%%%%%%%
\newpage
\section{Introduction}
%%%%%%%%%%%%%%%%%%%%%%%%%%%%%%%%%%%%%%%%%%%%%%%%%%%%%
%%%%%%%%%%%%%%%%%%%%%%%%%%%%%%%%%%%%%%%%%%%%%%%%%%%%%

%%%%%%%%%%%%%%%%%%%%%%%%%%%%%%%%%%%%%%%%%%%%
\subsection{An overview of Squier's theory}
%%%%%%%%%%%%%%%%%%%%%%%%%%%%%%%%%%%%%%%%%%%%

In the eighties, Squier has established a link between some computational, homological and homotopical properties of monoids,~\cite{Squier87,Squier94}. This allowed him to answer an open question: does a finitely generated monoid with a decidable word problem always admit a finite convergent presentation? 

%%%%%%%%%%%%%%%%%%%%%%%%%%%%%%%%
\subsubsection{The word problem and rewriting theory}

Given a monoid $\M$, a generating set $\Sigma_1$ for $\M$ provides a way to represent the elements of $\M$ in the free monoid $\Sigma_1^*$, \ie, as finite words written with the elements of $\Sigma_1$. But, if $\M$ is not free, there is no reason for an element of $\M$ to have a single representative in~$\Sigma_1^*$. The \emph{word problem} for $\M$ consists in finding a generating set $\Sigma_1$ and an algorithm that can determine whether or not any two elements of $\Sigma_1^*$ represent the same element of $\M$.

One way to solve the word problem is to exhibit a finite presentation $(\Sigma_1,\Sigma_2)$ of $\M$ with a good computational property, called \emph{convergence} in rewriting theory. There, one studies presentations where the relations in $\Sigma_2$ are not seen as equalities between the words in $\Sigma_1^*$, such as $u=v$, but, instead, as rewriting rules that can only be applied in one direction, like $u\dfl v$, simulating a non-reversible computational process. Convergence is defined as the conjunction of the following two conditions:
\begin{itemize}
\item \emph{termination}, \ie, all the computations end eventually,
\item \emph{confluence}, \ie, different computations on the same input lead to the same result.
\end{itemize}
A finite convergent presentation $(\Sigma_1,\Sigma_2)$ of $\M$ gives a solution to the word problem: the \emph{normal form algorithm}. Given an element $u$ in $\Sigma_1^*$, convergence ensures that all the applications of directed relations to~$u$, in any possible manner, will eventually produce a unique result: an element $\rep{u}$ of $\Sigma_1^*$ where no directed relation applies anymore, called the \emph{normal form} of $u$. And, by construction, two elements $u$ and $v$ of $\Sigma_1^*$ represent the same element of $\M$ if, and only if, their normal forms are equal in $\Sigma_1^*$. Finally, finiteness ensures that one can determine if an element of $\Sigma_1^*$ is a normal form, by examining all the relations. (As far as rewriting is concerned, this article is self-contained, but this wider mathematical field is covered in more details by Book and Otto,~\cite{BookOtto93}, Baader and Nipkow,~\cite{BaaderNipkow98}, and the group Terese,~\cite{Terese03}.)

Thus, if a monoid admits a finite convergent presentation, it has a decidable word problem. In the middle of the eighties, it was still unknown if the converse implication held. In~\cite{KapurNarendran85}, Kapur and Narendran had exhibited a monoid that admits a finite generating set for which the word problem was solvable, but that do not admit a finite convergent presentation with the same generators. However, this did not answer the original question, the generating set having been fixed.

%%%%%%%%%%%%%%%%%%%%%%%%%%%%%%%%
\subsubsection{From computational to homological properties}

At that time, Squier linked the existence of a finite convergent presentation to a homological invariant of the monoid, the homological type left-$\FP_3$, that is independent of the choice of a presentation of $M$ and, in particular, of a generating set. A monoid~$\M$ has homological type left-$\FP_3$ when there exists an exact sequence
\[
\xymatrix{
P_3
	\ar [r]  
& P_2
	\ar [r] 
& P_1
	\ar [r] 
& P_0
	\ar [r] 
& \Zb
	\ar [r]
& 0
}
\]
of (left) modules over $\M$, where $\Zb$ denotes the trivial $\M$-module and each $P_i$ is projective and finitely generated.

From a presentation $(\Sigma_1,\Sigma_2)$ of $\M$, one can build an exact sequence of free $\M$-modules
\begin{equation}
\label{resolution2}
\xymatrix{
\Zb\M[\Sigma_2] 
	\ar [r] ^-{J} 
& \Zb\M[\Sigma_1]
	\ar [r] ^-{}
& \Zb\M
	\ar [r] ^-{}
& \Zb
	\ar [r]
& 0,
}
\end{equation}
where $\Zb\M[\Sigma_1]$ and $\Zb\M[\Sigma_2]$ are the free $\M$-modules over $\Sigma_1$ and $\Sigma_2$, respectively. The differential~$J$, called the \emph{Fox Jacobian} after~\cite{Fox53}, is defined on a directed relation $\alpha:u\dfl v$ by $J(\alpha)\,=\,[u]-[v]$, where~$[\,\cdot\,]$ is the unique derivation of~$\Sigma_1^*$ with values in $\Zb\M[\Sigma_1]$ that extends the canonical inclusion of~$\Sigma_1$ into $\Zb\M[\Sigma_1]$.

In~\cite{Squier87}, Squier proved that, when $(\Sigma_1,\Sigma_2)$ is a convergent presentation, its \emph{critical branchings} form a generating set of the kernel of the Fox Jacobian. A critical branching of $(\Sigma_1,\Sigma_2)$ is an overlapping application of two different directed relations on the same word $u$ of $\Sigma_1^*$, where $u$ has minimal size. For example, the relations $\alpha:xy\dfl v$ and $\beta:yz\dfl w$ generate a critical branching $(\alpha z,x\beta)$ on  $u=xyz$:
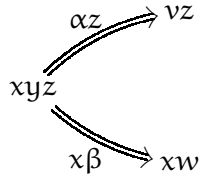
\[
\xymatrix @C=3em @R=1.5em {
& vz
\\
xyz 
	\ar@2@/^/ [ur] ^-{\alpha z}
	\ar@2@/_/ [dr] _-{x\beta}
\\
& xw
}	
\]
Convergence of the presentation $(\Sigma_1,\Sigma_2)$ ensures that any critical branching $(f,g)$ can be completed as in the following diagram:
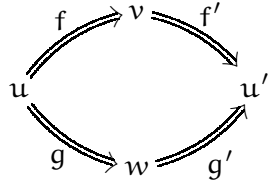
\[
\xymatrix @!C @C=2.5em @R=1.5em {
& v
	\ar@2@/^/ [dr] ^-{f'}
\\
u 
	\ar@2@/^/ [ur] ^-{f}
	\ar@2@/_/ [dr] _-{g}
&& u'
\\
& w
	\ar@2@/_/ [ur] _-{g'}
}	
\]
The boundary of such a branching is defined as the element $J'(f,g) \,=\, [f] - [g] + [f'] - [g']$, where~$[\,\cdot\,]$ extends the canonical inclusion of $\Sigma_2$ into $\Zb\M[\Sigma_2]$. 

Squier proves that the set $\Sigma_3$ of critical branchings of $(\Sigma_1,\Sigma_2)$ and the boundary~$J'$ extend the exact sequence~\eqref{resolution2} by one step:
\begin{equation}
\label{resolution3}
\xymatrix{
\Zb\M[\Sigma_3] 
	\ar [r] ^-{J'} 
& \Zb\M[\Sigma_2] 
	\ar [r] ^-{J} 
& \Zb\M[\Sigma_1]
	\ar [r] ^-{}
& \Zb\M
	\ar [r] ^-{}
& \Zb
	\ar [r]
& 0.
}
\end{equation}
Moreover, when $(\Sigma_1,\Sigma_2)$ is finite, then $\Sigma_3$ is finite, proving that, if a monoid has a finite convergent presentation, then it is of homological type left-$\FP_3$.

Finally, Squier exhibited a finitely generated monoid, with a decidable word problem, but that is not of homological type left-$\FP_3$. This gave a negative answer to the aforementioned open question: there exist finitely generated monoids with a decidable word problem that do not admit a finite convergent presentation (for any possible finite set of generators). 

%%%%%%%%%%%%%%%%%%%%%%%%%%%%%%%%
\subsubsection{From computational to homotopical properties}

In~\cite{Squier94}, Squier links the existence of a finite convergent presentation to a homotopical invariant of monoids, called \emph{finite derivation type} ($\FDT_3$) and that is a natural extension of the properties of being finitely generated ($\FDT_1$) and finitely presented ($\FDT_2$).

To define this invariant, for a monoid $\M$ with a presentation $(\Sigma_1,\Sigma_2)$, Squier constructs a cellular complex $S(\Sigma_1,\Sigma_2)$ with one $0$-cell, whose $1$-cells are the elements of the free monoid $\Sigma_1^*$ and whose $2$-cells are generated by the relations of $\Sigma_2$. More precisely, there is a $2$-cell between every pair of words with shape $wuw'$ and $wvw'$ such that $u=v$ is a relation in $\Sigma_2$. Then, to get $S(\Sigma_1,\Sigma_2)$, one fills with $3$-cells all the squares formed by independent applications of relations, such as the following one, where $(u_1,v_1)$ and $(u_2,v_2)$ are relations in $\Sigma_2$:
\[
\xymatrix @!C @C=3em @R=1.5em {
& {wv_1w'u_2w''}
	\ar@{=}@/^/ [dr] ^{u_2=v_2}
		\ar@3{-} []!<0pt,-25pt>;[dd]!<0pt,25pt> 
\\
{wu_1w'u_2w''} 
	\ar@{=}@/^/ [ur] ^{u_1=v_1}
	\ar@{=}@/_/ [dr] _{u_2=v_2}
&& {wv_1w'v_2w''}
\\
& {wu_1w'v_2w''}
	\ar@{=}@/_/ [ur] _{u_1=v_1}
}	
\]
If $\Sigma_3$ is a set of $3$-cells over $S(\Sigma_1,\Sigma_2)$, then the set $\Sigma_1^*\Sigma_3\Sigma_1^*$ is the set of $3$-cells $u\beta v$, with $\beta$ in~$\Sigma_3$ and~$u$ and~$v$ in $\Sigma_1^*$, and whose boundary is the one of $\beta$ multiplied by $u$ on the left and $v$ on the right. A \emph{homotopy basis} of $(\Sigma_1,\Sigma_2)$ is a set $\Sigma_3$ of $3$-cells such that $\Sigma_1^*\Sigma_3\Sigma_1^*$ makes the complex $S(\Sigma_1,\Sigma_2)$ contractible. A monoid is of \emph{finite derivation type} ($\FDT_3$) if it admits a finite presentation whose associated complex admits a finite homotopy basis or, in other words, whose ``relations among the relations'' are finitely generated.

Squier proves that, given a convergent presentation $(\Sigma_1,\Sigma_2)$, it is sufficient to attach one $3$-cell to each $3$-dimensional sphere corresponding to a critical branching to get a homotopy basis of $(\Sigma_1,\Sigma_2)$. Moreover, if $\Sigma_2$ is finite, the presentation $(\Sigma_1,\Sigma_2)$ has finitely many critical branchings proving that, if a monoid admits a finite convergent presentation, then it is $\FDT_3$. Squier used this result to give another proof that there exist finitely generated monoids with a decidable word problem that do not admit a finite convergent presentation. 

%%%%%%%%%%%%%%%%%%%%%%%%%%%%%%%%
\subsubsection{Refinements of Squier's conditions}

By his results, Squier has opened two different directions, one homological and one homotopical, to explore in the quest for a complete characterisation of the existence of finite convergent presentations in the case of monoids. The corresponding invariants are related: $\FDT_3$ implies left-$\FP_3$, as proved by several authors,~\cite{CremannsOtto94,Pride95,Lafont95jpaa}. The converse implication is false in general, as already noted by Squier in~\cite{Squier94}, yet it is true in the special case of groups,~\cite{CremannsOtto96}, the latter result being based on the Brown-Huebschmann isomorphism between homotopical and homological syzygies,~\cite{BrownHuebschmann82}. 

However, the invariants left-$\FP_3$ and $\FDT_3$ are not complete characterisations of the property to admit a finite convergent presentation: they are necessary, but not sufficient conditions, as already proved by Squier in~\cite{Squier94}. Following this observation, various refinements of both invariants have been explored.

In the homological direction, thanks to the notion of Abelian resolution, one defines the more restrictive conditions left-$\FP_n$, for every natural number $n>3$, and left-$\FP_{\infty}$: a monoid $\M$ has homological type left-$\FP_{\infty}$ when there exists a resolution of the trivial $\M$-module by finitely generated and projective $\M$-modules. In~\cite{Kobayashi91}, a notion of $n$-fold critical branching is used to complete the exact sequence~\eqref{resolution3} into a resolution, obtaining the following implication: if a monoid admits a finite convergent presentation, then it is of homological type left-$\FP_{\infty}$, the converse implication still being false in general. The same results are also known for associative algebras,~\cite{Anick86}, and for groups,~\cite{Cohen92,Brown92,Groves90}. One can obtain similar implications with the properties right-$\FP_{\infty}$ and bi-$\FP_{\infty}$, defined with resolutions by right modules and bimodules, respectively.

In the homotopical direction, the condition $\FDT_3$ has been refined into $\FDT_4$, a property about the existence of a finite presentation with a finite homotopy basis, itself satisfying a homotopical finiteness property,~\cite{Pride05}. The condition $\FDT_4$ is also necessary for a monoid to admit a finite convergent presentation and it is sufficient, but not necessary, for the conditions left/right/bi-$\FP_4$. 

%%%%%%%%%%%%%%%%%%%%%%%%%%%%%%%%%%%%%%%%%%%%
\subsection{Organisation and main results of the article}
%%%%%%%%%%%%%%%%%%%%%%%%%%%%%%%%%%%%%%%%%%%%

%%%%%%%%%%%%%%%%%%%%%%%%
\subsubsection{Polygraphic resolutions}

In Section~\ref{Section:PolygraphicResolutions}, we introduce a notion of homotopical resolution that generalises Squier's complex, in order to define the homotopical finiteness conditions $\FDT_n$ and $\FDT_{\infty}$. Squier's complex appears as the first two dimensions of a free $(\infty,1)$-category, \ie, a free $\infty$-category whose cells of dimension $2$ and higher are invertible. Then, homotopy bases and higher homotopy bases generate the higher dimensions of this $(\infty,1)$-category, in such a way that the latter is homotopically equivalent to the starting monoid. Moreover, these resolutions further generalise from monoids to $p$-categories, yielding free $(\infty,p)$-categories.

More explicitly, let $(\Sigma_1,\Sigma_2)$ be a presentation of a monoid $\M$. Such a presentation of a monoid, or more generally of a (small) category, is called a \emph{$(2,1)$-polygraph} in this article. The terminology comes from Burroni's polygraphs,~\cite{Burroni93}, also known as Street's computads,~\cite{Street76,Street87}. From the $(2,1)$-polygraph $(\Sigma_1,\Sigma_2)$, we generate a free $(2,1)$-category $\tck{\Sigma}_2$ by taking all the formal composites of $2$-cells, modulo exchange relations that correspond exactly to the $2$-cells of Squier's complex associated to $(\Sigma_1,\Sigma_2)$. Informally, the $(2,1)$-category $\tck{\Sigma}_2$ is homotopically equivalent to Squier's complex: this allows us to see a homotopy basis as a set $\Sigma_3$ of $3$-cells, attached to $\tck{\Sigma}_2$, such that the quotient $(2,1)$-category $\tck{\Sigma}_2/\Sigma_3$ is homotopically equivalent to the original monoid or, equivalently, such that any parallel $2$-cells of $\tck{\Sigma}_2$ are identified in the quotient by $\Sigma_3$. 

Then, one considers the free $(3,1)$-category $\tck{\Sigma}_3$ generated by the \emph{$(3,1)$-polygraph $(\Sigma_1,\Sigma_2,\Sigma_3)$}. One defines a homotopy basis of $\tck{\Sigma}_3$ as a set of $4$-cells over $\tck{\Sigma}_3$ that relate every parallel $3$-cells. The same idea is used to define homotopy bases in every dimension, yielding our notion of \emph{polygraphic resolution} of the monoid $\M$: this is an \emph{acyclic $(\infty,1)$-polygraph $\Sigma=(\Sigma_n)_{n\geq 1}$} such that $(\Sigma_1,\Sigma_2)$ is a presentation of $\M$, where acyclic means that each $\Sigma_{n+1}$ is a homotopy basis of the free $(n,1)$-category $\tck{\Sigma}_n$, for $n\geq 3$. 

The notion we get is close to Métayer's polygraphic resolutions, introduced in~\cite{Metayer03}: these are $\infty$-polygraphs that produce cofibrant approximations (free objects that are homotopically equivalent to the original one) in the canonical model structure on $\infty$-categories, described in~\cite{LafontMetayerWorytkiewicz10}. Our resolutions, called \emph{$(\infty,p)$-polygraphs}, have the same good homotopical properties, with respect to the canonical model structure on $(\infty,p)$-categories, obtained by Ara and Métayer,~\cite{AraMetayer11}.

\smallskip
\begin{quote}
\textbf{Theorem~\ref{TheoremCofibrantApproximation}.}
\begin{em}
Let $\Sigma$ be a polygraphic resolution of a $p$-category $\Cr$. The canonical projection $\tck{\Sigma}\pfl \Cr$ is a cofibrant approximation of $\Cr$ in the canonical model structure on $(\infty,p)$-categories.
\end{em}
\end{quote}

\smallskip
\noindent
We say that a monoid and, more generally, a $p$-category is of \emph{finite $\infty$-derivation type} ($\FDT_{\infty}$) when it admits a polygraphic resolution with finitely many cells in every dimension. This generalises to higher categories and in every dimension the two previously known homotopical finiteness conditions, $\FDT_3$ introduced by Squier for monoids and its refinement $\FDT_4$.

%%%%%%%%%%%%%%%%%%%%%%%%%%%%%%%%%%%%%%%%%%%%
\subsubsection{Normalisation strategies for polygraphs}

In Section~\ref{sectionStrategies}, we give a constructive characterisation of the acyclicity of an  $(\infty,1)$-polygraph. Let $\Sigma$ be a polygraphic resolution of a monoid $\M$. In particular, if $\pi:\Sigma_1^*\pfl\M$ denotes the canonical projection, one can choose a (non-functorial) section $\iota$ of $\pi$. Moreover, the $2$-cells of~$\Sigma_2$ are generating relations for $\M$: for every element $u$ of $\Sigma_1^*$, one can choose a $2$-cell $\sigma_u:u\dfl\iota\pi(u)$ in $\tck{\Sigma}_2$. Then, we use that $\Sigma_3$ is a homotopy basis: for every $2$-cell $f:u\dfl v$ of~$\tck{\Sigma}_2$, one can choose a $3$-cell $\sigma_f$ in $\tck{\Sigma}_3$ with shape 
\[
\xymatrix @!C @C=1em {
u 
	\ar@2 @/^3ex/ [rr] ^{f} _{}="src"
	\ar@2 @/_/ [dr] _-{\sigma_u}
&& v
	\ar@2 @/^/ [dl] ^-{\sigma_v}
\\
& {\iota\pi(u)} = {\iota\pi(v)}
		\ar@3 "src"!<0pt,-10pt>;"2,2"!<0pt,15pt> ^{\sigma_f}
}
\]
From the acyclicity of $\Sigma$, we deduce that similar choices can be made in every dimension, in a coherent way: if $\pi:\tck{\Sigma}\pfl\M$ and $\iota:\M\ifl\tck{\Sigma}$ are now seen as (weak) $\infty$-functors, then $\sigma$ is a (weak) natural isomorphism from the identity of~$\tck{\Sigma}$ to the composite (weak) $\infty$-functor $\iota\pi$.

We call $\sigma$ a \emph{normalisation strategy}. It generalises to every dimension the notion of strategy appearing in rewriting theory: a canonical way, among all the computations generated by the directed relations, to reduce a word into a normal form. An $(\infty,1)$-polygraph with a normalisation strategy is \emph{normalising}.

\smallskip
\begin{quote}
\textbf{Theorem~\ref{MainTheorem1}.}
\begin{em}
An $(n,1)$-polygraph is acyclic if, and only if, it is normalising.
\end{em}
\end{quote}

\smallskip
\noindent
Moreover, a normalisation strategy can always be assumed to commute with the monoid product in a sensible way: it can always ``reduce'' a word by starting on the left or on the right. This leads to the \emph{left-normalising} and \emph{right-normalising} properties for $(\infty,1)$-polygraphs, which are also equivalent to acyclicity.

%%%%%%%%%%%%%%%%%%%%%%%%%%%%%%%%%%%%%%%%%%%%
\subsubsection{Polygraphic resolutions from convergent presentations}

In Section~\ref{Section:Convergent2Polygraphs}, we use normalisation strategies to build, by induction on the dimension, an explicit polygraphic resolution from a convergent presentation. Given a convergent $(2,1)$-polygraph $\Sigma$, the first dimensions of the polygraphic resolution $\crit_{\infty}(\Sigma)$ we get are similar to the ones of Squier's complex: generators in dimension $1$, generating relations in dimension $2$, critical branchings in dimension $3$. Then, we build the $4$-cells from the critical triple branchings, \ie, the minimal overlappings of three $2$-cells on the same $1$-cell: for such a $(f,g,h)$, we use a normalisation strategy $\sigma$ to build the corresponding $4$-cell
\[
\xymatrix @C=2em @!C {
& {v} 
	\ar@2 @/^3ex/ [drr] ^-{\sigma_v}
		\ar@3 []!<7.5pt,-10pt>;[d]!<7.5pt,10pt> ^{A}
&&&&&& {v}
	\ar@2 @/^3ex/ [drr] ^-{\sigma_v}
		\ar@3 []!<10pt,-20pt>;[dd]!<10pt,20pt> ^{C}
\\
{u}
	\ar@2 @/^/ [ur] ^-{f}
	\ar@2 [r] |-{g}
	\ar@2 @/_/ [dr] _-{h}
& {w} 
		\ar@3 []!<7.5pt,-10pt>;[d]!<7.5pt,10pt> ^{B}
	\ar@2 [rr] |-{\sigma_w} 
&& {\rep{u}}
& \strut
	\ar@4 [r] ^-{}
&& {u}
	\ar@2 @/^/ [ur] ^-{f}
	\ar@2 @/_/ [dr] _-{h}
&&& {\rep{u}}
\\
& {x}
	\ar@2 @/_3ex/ [urr] _-{\sigma_x}
&&&&&& {x}
	\ar@2 @/_3ex/ [urr] _-{\sigma_x}
}
\]
where $A$, $B$ and $C$ are $3$-cells built by using the critical branchings and the normalisation strategy $\sigma$. In higher dimensions, we proceed similarly to build the $(n+1)$-cells of the resolution from the critical $n$-fold branchings.

\smallskip
\begin{quote}
\textbf{Theorem~\ref{MainTheorem2.0}.} 
\begin{em}
If $\Sigma$ is a convergent presentation of a category $\C$, then the $(\infty,1)$-polygraph $\crit_{\infty}(\Sigma)$ is a polygraphic resolution of $\C$.
\end{em}
\end{quote}

\smallskip
\noindent
Since a finite convergent $(2,1)$-polygraph has finitely many $n$-fold critical branchings, a category with a finite convergent presentation is $\FDT_{\infty}$.

%%%%%%%%%%%%%%%%%%%%%%%%%%%%%%%%%%%%%%%%%%%%
\subsubsection{Abelianisation of polygraphic resolutions}

In Section~\ref{SectionAbelianisation}, we relate the homotopical finiteness condition $\FDT_{\infty}$ to a new homological finiteness condition, called $\FP_{\infty}$. For that, from a polygraphic resolution $\Sigma$ of a category $\C$, we deduce a free Abelian resolution
\[
\xymatrix{
\,\cdots\,
	\ar [r] ^-{\delta_{n+1}}
& \fnat{\C}{\Sigma_n}
	\ar [r] ^-{\delta_n}
& \fnat{\C}{\Sigma_{n-1}}
	\ar [r] ^-{\delta_{n-1}}
& \,\cdots\,
	\ar [r] ^-{\delta_2}
& \fnat{\C}{\Sigma_1}
	\ar [r] ^-{\delta_1}
& \fnat{\C}{\Sigma_0}
	\ar [r] ^-{\epsilon}
& \Zb
	\ar [r]
& 0
}
\]
in the category of \emph{natural systems on $\C$}, a generalisation of bimodules due to Baues, see~\cite{Baues91}. This complex, denoted by $\fnat{\C}{\Sigma}$, is called the \emph{Reidemeister-Fox-Squier complex of $\Sigma$} and extends Squier's exact sequence~\eqref{resolution3}. The acyclicity of $\fnat{\C}{\Sigma}$ is proved by using a contracting homotopy induced by a normalisation strategy. 

\smallskip
\begin{quote}
\textbf{Theorem~\ref{Theorem:AbelianResolution}.}
\begin{em}
If $\Sigma$ is a polygraphic resolution of a category $\C$, then the Reidemeister-Fox-Squier complex $\fnat{\C}{\Sigma}$ is a free resolution of the constant natural system $\Zb$ on $\C$.
\end{em}
\end{quote}

\smallskip
\noindent
We define the homological properties $\FP_n$, for any $n$, and $\FP_{\infty}$ as the refined versions of left/right/bi-$\FP_n$ and left/right/bi-$\FP_{\infty}$ with natural systems instead of left/right/bi-modules. We get that, for categories, the property $\FDT_n$ implies $\FP_n$ and, as a consequence, that a category with a finite convergent presentation is of homological type $\FP_{\infty}$.

Finally, we relate the homological $2$-syzygies of a presentation $\Sigma$ to its \emph{identities among relations}, defined by the authors in~\cite{GuiraudMalbos10smf}.

\smallskip
\begin{quote}
\textbf{Theorem~\ref{Theorem:IsomorphismPiH2}.}
\begin{em}
For every $2$-polygraph $\Sigma$, the natural systems of homological $2$-syzygies and of identities among relations of $\Sigma$ are isomorphic.
\end{em}
\end{quote}

\smallskip
\noindent
As a consequence, for finitely presented categories, the homological finiteness condition $\FP_3$ is equivalent to the homotopical finiteness condition $\FDTAB$, characterising the existence of a finite homotopy basis of an Abelianised version of a presentation of the category, see Theorem~\ref{Theorem:TDFAB<=>FP_3}.

%%%%%%%%%%%%%%%%%%%%%%%%%%%%%%%%%%%%%%%%%%%%
\subsubsection{Examples}

Throughout this article, we apply our constructions to the example of the \emph{reduced standard presentation} of a category, yielding, at the end, an Abelian resolution that is similar to the bar construction. In Section~\ref{SectionExamples}, we give two more examples: the monoid with one non-unit and idempotent element and the subcategory of the simplicial category whose morphisms are the monotone surjections only. They give rise to resolutions where the higher-dimensional cells have the shapes of associahedra and permutohedra, respectively.

%%%%%%%%%%%%%%%%%%%%%%%%%%%%%%%%%%%%%%%%%%%%
\subsection{Acknowledgements}
%%%%%%%%%%%%%%%%%%%%%%%%%%%%%%%%%%%%%%%%%%%%

The authors wish to thank François Métayer, Timothy Porter and the anonymous referee for their comments and suggestions that helped to produce an improved version of this work.

\newpage
%%%%%%%%%%%%%%%%%%%%%%%%%%%%%%%%%%%%%%%%%%%%
%%%%%%%%%%%%%%%%%%%%%%%%%%%%%%%%%%%%%%%%%%%%
\section{Polygraphic resolutions}
\label{Section:PolygraphicResolutions}
%%%%%%%%%%%%%%%%%%%%%%%%%%%%%%%%%%%%%%%%%%%%
%%%%%%%%%%%%%%%%%%%%%%%%%%%%%%%%%%%%%%%%%%%%

Throughout this section, we denote by $n$ either a natural number or $\infty$.

%%%%%%%%%%%%%%%%%%%%%%%%%%%%%%%%%%%%%%%%%%%
\subsection{Higher-dimensional categories}
%%%%%%%%%%%%%%%%%%%%%%%%%%%%%%%%%%%%%%%%%%%

If $\Cr$ is an $n$-category (we always consider strict, globular $n$-categories), we denote by $\Cr_k$ the set (and the $k$-category) of $k$-cells of~$\Cr$. If $f$ is a $k$-cell of $\Cr$, then $s_i(f)$ and $t_i(f)$ respectively denote the $i$-source and $i$-target of $f$; we drop the suffix~$i$ when $i=k-1$. The source and target maps satisfy the \emph{globular relations}: 
\[
s_i\circ s_{i+1} \:=\: s_i\circ t_{i+1} 
\qquad\text{and}\qquad
t_i\circ s_{i+1} \:=\: t_i\circ t_{i+1}.
\]
We respectively denote by $f:u\fl v$, $\;f:u\dfl v\;$ or $\;f:u\tfl v\;$ a $1$-cell, a $2$-cell or a $3$-cell $f$ with source~$u$ and target~$v$. 

If $f$ and $g$ are $i$-composable $k$-cells, that is when $t_i(f)=s_i(g)$, we denote by $f\star_i g$ their $i$-composite; we simply use $fg$ when $i=0$. The compositions satisfy the \emph{exchange relations} given, for every $i\neq j$ and every possible cells $f$, $g$, $h$ and $k$, by: 
\[
(f \star_i g) \star_j (h \star_i k) \:=\: (f \star_j h) \star_i (g \star_j k).
\]
If $f$ is a $k$-cell, we denote by $1_f$ its identity $(k+1)$-cell. When $1_f$ is composed with cells of dimension $k+1$ or higher, we simply denote it by~$f$. 

%%%%%%%%%%%%%%%%%%%%%%
\subsubsection{\pdf{(n,p)}-categories}

A $k$-cell $f$ of an $n$-category~$\Cr$, with $i$-source $u$ and $i$-target $v$, is \emph{$i$-invertible} when there exists a (necessarily unique) $k$-cell $g$ in $\Cr$, with $i$-source $v$ and $i$-target $u$ in $\Cr$, called the \emph{$i$-inverse of $f$}, that satisfies 
\[
f\star_i g \:=\: 1_u 
\qquad\text{and}\qquad
g\star_i f \:=\: 1_v.
\]
When $i=k-1$, we just say that $f$ is \emph{invertible} and we denote by $f^-$ its \emph{inverse}. As in higher-dimensional groupoids, if a $k$-cell $f$ is invertible and if its $i$-source $u$ and $i$-target $v$ are invertible, then $f$ is $(i-1)$-invertible, with $(i-1)$-inverse given by 
\[
v^- \star_{i-1} f^- \star_{i-1} u^-.
\]
For a natural number $p\leq n$, or for $p=n=\infty$, an \emph{$(n,p)$-category} is an $n$-category whose $k$-cells are invertible for every $k>p$. When $n<\infty$, this is a $p$-category enriched in $(n-p)$-groupoids and, when $n=\infty$, a $p$-category enriched in $\infty$-groupoids. In particular, an $(n,n)$-category is an $n$-category, an $(n,0)$-category is an $n$-groupoid and, when $n<\infty$, an $(n,n-1)$-category is a track $(n-1)$-category, as defined in~\cite{GuiraudMalbos09} after Baues,~\cite{Baues91,BauesMinian01}. If $n<\infty$, any $(n,p)$-category can be seen as an $(\infty,p)$-category with only identity $k$-cells for $k>n$.

%%%%%%%%%%%%%%%%%%%%%%%%
\subsubsection{Spheres and asphericity}

Let $\Cr$ be an $n$-category. A \emph{$0$-sphere of $\Cr$} is a pair $\gamma=(f,g)$ of $0$-cells of $\Cr$ and, for $1\leq k\leq n$, a \emph{$k$-sphere of $\Cr$} is a pair $\gamma=(f,g)$ of parallel $k$-cells of $\Cr$, \ie, with $s(f)=s(g)$ and $t(f)=t(g)$; we call $f$ the \emph{source} of $\gamma$ and~$g$ its \emph{target}. If $f$ is a $k$-cell of $\Cr$, for $1\leq k\leq n$, the \emph{boundary of $f$} is the $(k-1)$-sphere $(s(f),t(f))$. If $n<\infty$, the $n$-category $\Cr$ is \emph{aspherical} when the source and the target of each $n$-sphere of $\Cr$ coincide, \ie, when every $n$-sphere of $\Cr$ has shape $(f,f)$ for some $(n-1)$-cell $f$ of $\Cr$. 

%%%%%%%%%%%%%%%%%%%%
\subsubsection{The canonical model structure on $(\infty,p)$-categories}

In~\cite{AraMetayer11}, Ara and Métayer have proved that the canonical model structure on $\infty$-categories from~\cite{LafontMetayerWorytkiewicz10} transfers to $(\infty,p)$-categories through the adjunction
\[
\xymatrix @C=5em @!C{
(\infty,p)\Cat
	\ar@/^3ex/ [r] ^{U}
	\ar@{} [r] |{\top}
& \infty\Cat
	\ar@/^3ex/ [l] ^{\tck{(\;\cdot\;)}}
}
\]
where $U$ is the forgetful functor and its left adjoint adds to an $\infty$-category all the missing inverses for cells of dimension $p+1$ and above. The proof in~\cite{AraMetayer11} is detailed for the specific case $p=0$, \ie, for $\infty$-groupoids, but it works equally well in the general case. Here we are interested in cofibrant replacements in the model structure on $(\infty,p)$-categories, so let us examine the classes of weak equivalences and cofibrations. 

From~\cite{LafontMetayerWorytkiewicz10}, we recall that an $\infty$-functor $F:\Cr\fl\Dr$ is a \emph{weak equivalence} in the model structure on $\infty$-categories if, and only if, it satisfies the following properties:
\begin{itemize}
\item for every $0$-cell $x$ of $\Dr$, there exists a $0$-cell $\rep{x}$ of $\Cr$ such that $F(\rep{x})$ is $\omega$-equivalent to $x$,

\item for every $0$-cells $x$ and $y$ of $\Cr$ and every $1$-cell $u:F(x)\fl F(y)$ of $\Dr$, there exists a $1$-cell $\rep{u}:x\fl y$ in $\Cr$ such that $F(\rep{u})$ is $\omega$-equivalent to $u$,

\item for every parallel $n$-cells $u$ and $v$ of $\Cr$, with $n\geq 1$, and every $(n+1)$-cell $f:u\fl v$ of $\Dr$, there exists an $(n+1)$-cell $\rep{f}:u\fl v$ of $\Cr$ such that $F(\rep{f})$ is $\omega$-equivalent to $f$.
\end{itemize}
The $\omega$-equivalence relation is defined together with the notion of reversible cells, by mutual coinduction:
\begin{itemize}
\item two $n$-cells $u$ and $v$ of an $\omega$-category $\Cr$ are \emph{$\omega$-equivalent} when there exists a reversible $(n+1)$-cell $f:u\fl v$ in $\Cr$,
\item an $(n+1)$-cell $f:u\fl v$ of an $\omega$-category $\Cr$ is \emph{reversible} when there exists an $(n+1)$-cell $g:v\fl u$ in $\Cr$ such that $g\star_n f$ and $f\star_n g$ are $\omega$-equivalent to $1_x$ and $1_y$, respectively.
\end{itemize}
From the result of~\cite{AraMetayer11}, the weak equivalences for $(\infty,p)$-categories are the images through the forgetful functor $U$ of the weak equivalences for $\infty$-categories, \ie, the $\infty$-functors between $(\infty,p)$-categories that are weak equivalences for $\infty$-categories.

In the canonical model structure on $\infty$-categories, the cofibrations are the retracts of transfinite compositions of pushouts of the $\infty$-functors
\[
i_n \::\: \dr \Er_n \:\fl\: \Er_n,
\]
for $n\geq 0$, where $\Er_n$ is the $n$-globe and $\dr\Er_n$ its boundary: 
\begin{itemize}
\item the $n$-globe $\Er_n$ is the $n$-category with exactly one $n$-cell together with its distinct $k$-source and $k$-target for every $0\leq k<n$,
\item the boundary $\dr\Er_n$ of the $n$-globe is the same $n$-category as $\Er_n$ but with the $n$-cell removed.
\end{itemize}
By the result of~\cite{AraMetayer11}, we get that the cofibrations for $(\infty,p)$-categories are the retracts of transfinite compositions of pushouts of the $\infty$-functors
\[
\tck{i}_n \::\: \dr\tck{\Er_n} \:\fl\: \tck{\Er_n},
\]
for $n\geq 0$, where $\tck{\Er_n}$ and $\dr\tck{\Er_n}$ are obtained from $\Er_n$ and $\dr\Er_n$ by formal adjunction of inverses for every $k$-cell, with $1<k\leq n$.

%%%%%%%%%%%%%%%%%%%%%%%%%%%%%%%%%%%%%%%%%%%
\subsection{Polygraphs}
%%%%%%%%%%%%%%%%%%%%%%%%%%%%%%%%%%%%%%%%%%%

%%%%%%%%%%%%%%%%%%%%
\subsubsection{Cellular extensions}

Let us assume that $n<\infty$ and let $\Cr$ be an $n$-category. A \emph{cellular extension of $\Cr$} is a set $\Gamma$ equipped with a map $\dr$ from $\Gamma$ to the set of $n$-spheres of $\Cr$. By considering all the formal compositions of elements of~$\Gamma$, seen as $(n+1)$-cells with source and target in~$\Cr$, one builds the \emph{free $(n+1)$-category generated by~$\Gamma$ over $\Cr$}, denoted by $\Cr[\Gamma]$. The \emph{size} of an $(n+1)$-cell $f$ of $\Cr[\Gamma]$ is the number of $(n+1)$-cells of $\Gamma$ it contains. 

The \emph{quotient of $\Cr$ by $\Gamma$}, denoted by $\Cr/\Gamma$, is the $n$-category one gets from $\Cr$ by identification of the $n$-cells $s(\gamma)$ and $t(\gamma)$, for every $n$-sphere $\gamma$ of $\Gamma$. 
% We write $f\equiv_{\Gamma}g$ when two parallel $n$-cells $f$ and $g$ of~$\Cr$ are identified in~$\Cr/\Gamma$.
If $\Cr$ is an $(n,p)$-category and $\Gamma$ is a cellular extension of $\Cr$, then the \emph{free $(n+1,p)$-category generated by~$\Gamma$ over $\Cr$} is denoted by~$\Cr(\Gamma)$ and defined as follows: 
\[
\Cr(\Gamma) \:=\: \Cr \left[ \Gamma,\; \Gamma^- \right] \big/ \; \text{Inv}(\Gamma) 
\]
where $\Gamma^-$ contains the same $(n+1)$-cells as $\Gamma$, with source and target reversed, and $\text{Inv}(\Gamma)$ is the cellular extension made of two $(n+2)$-cells 
\[
\gamma\star_{n+1}\gamma^- \:\fl\: 1_{f}
\qquad\text{and}\qquad
\gamma^-\star_{n+1}\gamma \:\fl\: 1_{g}
\]
for each $(n+1)$-cell $\gamma$ from $f$ to $g$ in $\Gamma$. 

%%%%%%%%%%%%%%%%%%%%%%%%%%%%%%%%%%%%%%%%%%%
\subsubsection{Homotopy bases}

Let $\Cr$ be an $(n,p)$-category, for $p<n<\infty$. A \emph{homotopy basis of $\Cr$} is a cellular extension $\Gamma$ of $\Cr$ such that the $(n,p)$-category $\Cr/\Gamma$ is aspherical, \ie, such that, for every $n$-sphere $\gamma$ of $\Cr$, there exists an $(n+1)$-cell with boundary~$\gamma$ in the $(n+1,p)$-category~$\Cr(\Gamma)$. For example, the $n$-spheres of $\Cr$ form a cellular extension which is a homotopy basis of $\Cr$.

% In particular, the $(n,p)$-category $\Cr$ is acyclic if, and only if $\Cr$ is aspherical and, for every $p<k<n$, the cellular extension $\Cr_{k+1}$ of $(k+1)$-cells of the $(k,p)$-category $\Cr_{k}$ is a homotopy basis.

%%%%%%%%%%%%%%%%%%%%%%%%%%%%%%%%%%%%%%%%%%%
\subsubsection{$(n,p)$-polygraphs}
% \label{Definition:(n,p)-polygraphs}

An \emph{$n$-polygraph} is a data $\Sigma$ made of a set $\Sigma_0$ and, for every $0\leq k<n$, a cellular extension $\Sigma_{k+1}$ of the free $k$-category
\[
\Sigma_k^* \:=\: \Sigma_0[\Sigma_1]\cdots[\Sigma_k].
\]
For $p\leq n$, an \emph{$(n,p)$-polygraph} is a data $\Sigma$ made of:
\begin{itemize}
\item a $p$-polygraph $(\Sigma_0,\dots,\Sigma_p)$,
\item for every $p \leq k < n$, a cellular extension $\Sigma_{k+1}$ of the free $(k,p)$-category 
\[
\tck{\Sigma}_k \:=\: \Sigma^*_p(\Sigma_{p+1})\cdots(\Sigma_{k}).
\]
\end{itemize}
Note that $(n,n)$-polygraphs coincide with $n$-polygraphs, so that any notion defined on $(n,p)$-polygraphs also covers the case of $n$-polygraphs. 

For an $(n,p)$-polygraph $\Sigma$, an element of $\Sigma_k$ is a called \emph{$k$-cell of $\Sigma$} and $\Sigma$ is \emph{finite} when it has finitely many cells in every dimension. An $(n,p)$-polygraph $\Sigma$ is \emph{aspherical} when the free $(n,p)$-category $\tck{\Sigma}$ is aspherical. An $(n,p)$-polygraph $\Sigma$ is \emph{acyclic} when, for every $p<k<n$, the cellular extension $\Sigma_{k+1}$ is a homotopy basis of the $(k,p)$-category $\tck{\Sigma}_k$.

%%%%%%%%
\begin{remark}
An $(n,p)$-polygraph yields a diagram which is similar to the one given in the original definition of $n$-polygraphs by Burroni,~\cite{Burroni93}, drawn for the case $n<\infty$ as follows:
\[
\xymatrix@W=3em{
\Sigma_0^*
& (\cdots)
	\ar@<.25ex> [l]
	\ar@<-.25ex> [l]
& \Sigma_p^*
	\ar@<.25ex> [l]
	\ar@<-.25ex> [l]
& \tck{\Sigma}_{p+1}
	\ar@<.25ex> [l]
	\ar@<-.25ex> [l]
& (\cdots)
	\ar@<.25ex> [l]
	\ar@<-.25ex> [l]
& \tck{\Sigma}_{n-1}
	\ar@<.25ex> [l]
	\ar@<-.25ex> [l]
\\
\Sigma_0
	\ar@{=} [u]
& (\cdots)
	\ar@<.25ex> [ul]
	\ar@<-.25ex> [ul]
& \Sigma_p
	\ar@<.25ex> [ul]
	\ar@<-.25ex> [ul]
	\ar@{ >->} [u]
& \Sigma_{p+1}
	\ar@<.25ex> [ul]
	\ar@<-.25ex> [ul]
	\ar@{ >->} [u]
& (\cdots)
	\ar@<.25ex> [ul]
	\ar@<-.25ex> [ul]
& \Sigma_{n-1}
	\ar@<.25ex> [ul]
	\ar@<-.25ex> [ul]
	\ar@{ >->} [u]
& \Sigma_{n}
	\ar@<.25ex> [ul]
	\ar@<-.25ex> [ul]
}
\]
This diagram contains the source and target attachment maps of generating $(k+1)$-cells on composite $k$-cells, their extension to composite $(k+1)$-cells and the inclusion of generating $k$-cells into composite $k$-cells.
\end{remark}

\begin{proposition}
\label{PropositionPolygraphesCofibrants}
Every free $(\infty,p)$-category on an $(\infty,p)$-polygraph is a cofibrant object for the canonical model structure on $(\infty,p)$-categories.
\end{proposition}

\begin{proof}
Let $\Sigma$ be an $(\infty,p)$-polygraph. The unique $\infty$-functor from the initial $(\infty,p)$-category $\emptyset$ to~$\tck{\Sigma}$ is obtained as the following countable composition of inclusions:
\[
\emptyset
	\:\ifl\: \Sigma_0
	\:\ifl\: \Sigma_1^*
	\:\ifl\: \cdots
	\:\ifl\: \Sigma_p^*
	\:\ifl\: \tck{\Sigma}_{p+1}
	\:\ifl\: \cdots
	\:\ifl\: \tck{\Sigma}_n
	\:\ifl\: \cdots
\]
The generating cofibration $\tck{i}_0$ is the inclusion of $\dr\tck{\Er}_0=\emptyset$ into the singleton $\tck{\Er}_0$. Thus, for any set $X$, seen as a $0$-polygraph, the inclusion $\emptyset\ifl X$ is equal to 
\[
\xymatrix @C=5em {
\emptyset \:\simeq\: \bigsqcup_{x\in X} \dr\tck{\Er}_0
	\ar [r] ^-{\bigsqcup_{x\in X}\tck{i}_0} 
& {\bigsqcup_{x\in X} \tck{\Er}_0 \:\simeq\: X.}
}
\]
Then, for $0<k\leq p$, the inclusion of $\Sigma^*_{k-1}$ into $\Sigma^*_k$ is a particular case of an inclusion $\iota:\Cr\ifl\Cr[\Gamma]$ for $\Cr$ a $(k-1)$-category and $\Gamma$ a cellular extension of $\Cr$. By seeing each $(k-1)$-sphere $\gamma$ of $\Gamma$ as an $\infty$-functor from $\dr\tck{\Er_k}$ to $\Cr$, the inclusion $\iota$ is given by the following pushout:
\[
\xymatrix @C=4em @R=2.5em @!R {
{\bigsqcup_{\gamma\in\Gamma} \dr\tck{\Er_k}}
	\ar [r] ^-{\bigsqcup_{\gamma\in\Gamma} \tck{i_k}}
	\ar [d] _-{\Gamma}	
& {\bigsqcup_{\gamma\in\Gamma} \tck{\Er_k}}
	\ar [d]
\\
\Cr
	\ar@{ >->} [r] _-{\iota}
& \Cr[\Gamma].
}
\]
Finally, the inclusion of $\Sigma^*_p$ into $\tck{\Sigma}_{p+1}$ and, for $n>p$, the inclusion of $\tck{\Sigma}_n$ into $\tck{\Sigma}_{n+1}$ are particular cases of an inclusion $\iota:\Cr\ifl\Cr(\Gamma)$, for $\Cr$ an $(n,p)$-category and $\Gamma$ a cellular extension of $\Cr$. By seeing each $n$-sphere $\gamma$ of $\Gamma$ as an $\infty$-functor from $\dr\tck{\Er_{n+1}}$ to $\Cr$, the inclusion $\iota$ is given by the following pushout:
\[
\xymatrix @C=4em @R=2.5em @!R {
{\bigsqcup_{\gamma\in\Gamma} \dr\tck{\Er_{n+1}}}
	\ar [r] ^-{\bigsqcup_{\gamma\in\Gamma} \tck{i_{n+1}}}
	\ar [d] _-{\Gamma}	
& {\bigsqcup_{\gamma\in\Gamma} \tck{\Er_{n+1}}}
	\ar [d]
\\
\Cr
	\ar@{ >->} [r] _-{\iota}
& \Cr(\Gamma).
}
\]
As a conclusion, we get that the inclusion $\emptyset\ifl\tck{\Sigma}$ is a countable composition of pushouts of the generating cofibrations $(\tck{i}_n)_{n\geq 0}$ and, as such, it is a cofibration.
\end{proof}

%%%%%%%%%%%%%%%%%%%%%%%%%%%%%%%%%%%%%%%%%%%%
\subsection{Resolutions by \pdf{(n,p)}-polygraphs}
\label{subsectionTDFn}
%%%%%%%%%%%%%%%%%%%%%%%%%%%%%%%%%%%%%%%%%%%%

%%%%%%%%%%%%%%%%%%%%%%%%%%%%%%%%%%%%%%%%%%%%
\subsubsection{Polygraphic presentations}

If $p<n$, given an $(n,p)$-polygraph $\Sigma$, the \emph{$p$-category $\cl{\Sigma}$ presented by $\Sigma$} is defined by
\[
\cl{\Sigma} \:=\: \Sigma_p^*/\Sigma_{p+1}.
\]
We usually denote by $\cl{f}$ the image of a $p$-cell of $\Sigma^*_p$ through the canonical projection $\Sigma_p^*\twoheadrightarrow\cl{\Sigma}$. If $f$ is a $k$-cell of $\tck{\Sigma}$, with $p< k\leq n$, we also denote by $\cl{f}$ the common image in $\cl{\Sigma}$ of the $p$-cells $s_p(f)$ and $t_p(f)$ by the canonical projection. An $(m,p)$-polygraph $\Sigma$ and an $(n,p)$-polygraph $\Upsilon$ are \emph{Tietze-equivalent} when the $p$-categories $\cl{\Sigma}$ and $\cl{\Upsilon}$ they present are isomorphic.

\begin{example}
Every category $\C$ admits a presentation, called the \emph{standard presentation of $\C$}, defined as the $2$-polygraph whose cells are the following ones: 
\begin{itemize}
\item one $0$-cell for each $0$-cell of $\C$,

\item one $1$-cell $\rep{u}:x\fl y$ for every $1$-cell $u:x\fl y$ of $\C$,

\item one $2$-cell $\mu_{u,v}:\rep{u}\rep{v}\dfl\rep{uv}$ for every $1$-cells $u:x\fl y$ and $v:y\fl z$ of $\C$ and one $2$-cell $\eta_x:1_x\dfl\rep{1}_x$ for every $0$-cell $x$ of $\C$:
\[
\xymatrix@!C{
& y
	\ar @/^/ [dr] ^{\rep{v}}
\\
x 
	\ar @/^/ [ur] ^{\rep{u}}
	\ar @/_/ [rr] _{\rep{uv}} ^{}="tgt"
&& z
	\ar@2 "1,2"!<0pt,-15pt>;"tgt"!<0pt,10pt> ^{\mu_{u,v}}
}
\qquad\qquad\qquad
\raisebox{-1.75em}{
\xymatrix@C=3em @!C{
x 
	\ar@/^3ex/ [r] ^-{1_x} ^-{}="src"
	\ar@/_3ex/ [r] _-{\rep{1}_x} _-{}="tgt"
& x.
\ar@2 "src"!<0pt,-10pt>;"tgt"!<0pt,10pt> ^-{\eta_x}	
}
}
\]
\end{itemize}
In the free category generated by the $1$-cells of the standard presentation of $\C$, we get, for every $0$-cell~$x$, the identity $1_x$ of $x$ and the generating $1$-cell $\rep{1}_x$ associated to the identity of $x$ in $\C$. By removing this last superfluous generator, together with the corresponding $2$-cell $\eta_x$, we get another presentation of $\C$, namely the $2$-polygraph called the \emph{reduced standard presentation of $\C$}, with the following cells:
\begin{itemize}
\item one $0$-cell for each $0$-cell of $\C$,

\item one $1$-cell $\rep{u}:x\fl y$ for every non-identity $1$-cell $u:x\fl y$ of $\C$,

\item one $2$-cell 
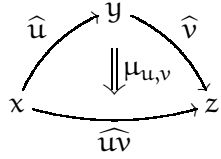
\[
\xymatrix@!C{
& y
	\ar @/^/ [dr] ^{\rep{v}}
\\
x 
	\ar @/^/ [ur] ^{\rep{u}}
	\ar @/_/ [rr] _{\rep{uv}} ^{}="tgt"
&& z
	\ar@2 "1,2"!<0pt,-15pt>;"tgt"!<0pt,10pt> ^{\mu_{u,v}}
}
\]
for every non-identity $1$-cells $u:x\fl y$ and $v:y\fl z$ of $\C$ such that $uv$ is not an identity,

\item one $2$-cell 
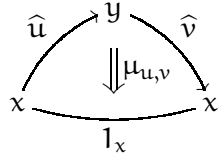
\[
\xymatrix{
& y
	\ar @/^/ [dr] ^{\rep{v}}
\\
x 
	\ar @/^/ [ur] ^{\rep{u}}
	\ar@{-} @/_/ [rr] _{1_x} ^{}="tgt"
&& x
	\ar@2 "1,2"!<0pt,-15pt>;"tgt"!<0pt,10pt> ^{\mu_{u,v}}
}
\]
for every non-identity $1$-cells $u:x\fl y$ and $v:y\fl x$ of $\C$ such that $uv=1_x$.
\end{itemize}
\end{example}

%%%%%%%%%%%%%%%%%%%%%%%%%%%%%%%%%%%%%%%%%%%
\subsubsection{Polygraphic resolutions}
\label{TrackPresentation}

Let $\Cr$ be a $p$-category. A \emph{polygraphic resolution of $\Cr$} is an acyclic $(\infty,p)$-polygraph $\Sigma$ such that the $p$-category $\cl{\Sigma}$ is isomorphic to $\Cr$. If $p<n<\infty$, a \emph{partial polygraphic resolution of length $n$ of $\Cr$} is an acyclic $(n,p)$-polygraph $\Sigma$ such that $\cl{\Sigma}$ is isomorphic to $\Cr$. Explicitly, the first dimensions of a polygraphic resolution $\Sigma$ of $\Cr$ are given as follows:
\begin{itemize}
\item For $k<p$, the $k$-cells of $\Sigma$ are the ones of $\Cr$. In particular, polygraphic resolutions concern the cofibrant $p$-categories only, \ie, the $p$-categories that are free up to dimension $p-1$, which is always the case for $p=1$.
\item The $p$-cells of $\Sigma$ are generators for the ones of $\Cr$, \ie, the $p$-category $\Cr$ is a quotient of the free $p$-category $\Sigma_p^*$.
\item The $(p+1)$-cells of $\Sigma$ are relations, \ie, the $(p+1)$-polygraph $\Sigma_{p+1}$ is a presentation of $\Cr$.
\item The $(p+2)$-cells of $\Sigma$ form a homotopy basis of $\tck{\Sigma}_{p+1}$, \ie, they are generators of the relations between relations of the presentation $\Sigma_{p+1}$ of $\Cr$.
\end{itemize}

\noindent
As previously mentioned, Métayer introduced a notion of polygraphic resolution of a $p$-category~$\Cr$, with $0\leq p\leq\infty$, in~\cite{Metayer03}: this is an $\infty$-polygraph $\Sigma$ such that the free $\infty$-category $\Sigma^*$ is a cofibrant replacement of $\Cr$ in the canonical model structure on $\infty$-categories. The notion we use here is similar, using $(\infty,p)$-polygraphs to produce cofibrant approximations of $p$-categories in the canonical model structure on $(\infty,p)$-categories:

\begin{theorem}
\label{TheoremCofibrantApproximation} 
Let $\Sigma$ be a polygraphic resolution of a $p$-category $\Cr$. The canonical projection $\tck{\Sigma}\!\!\pfl\!\!\Cr$ is a cofibrant approximation of $\Cr$ in the canonical model structure on $(\infty,p)$-categories.
\end{theorem}

\begin{proof}
From Proposition~\ref{PropositionPolygraphesCofibrants}, we already know that $\tck{\Sigma}$ is cofibrant. There remains to check that the canonical projection $\tck{\Sigma}\pfl \Cr$ is a weak equivalence. Since, by hypothesis, the $p$-categories $\Cr$ and $\cl{\Sigma}$ are isomorphic, it is sufficient to prove that the canonical projection $\tck{\Sigma}\pfl\cl{\Sigma}$ is a weak equivalence. First, we note that the $\omega$-equivalence relation is reflexive: hence, proving that two $k$-cells of $\cl{\Sigma}$ are equal implies that they are $\omega$-equivalent.

By definition, the $(\infty,p)$-categories $\tck{\Sigma}$ and $\cl{\Sigma}$ have the same cells up to dimension $p-1$. Thus, if $x$ is a $0$-cell of~$\cl{\Sigma}$, we take $\rep{x}=x$. Moreover, if $x$ and $y$ are parallel $k$-cells of $\tck{\Sigma}$, for $0<k<p-1$, and if $u:x\fl y$ is a $(k+1)$-cell of $\cl{\Sigma}$, then we take $\rep{u}=u$.

Now, let $x$ and $y$ be parallel $(p-1)$-cells of $\tck{\Sigma}$ and let $u:x\fl y$ be a $p$-cell of $\cl{\Sigma}$. By definition, the $p$-category $\cl{\Sigma}$ is a quotient of the $p$-category underlying $\tck{\Sigma}$. Hence, there exists a $p$-cell $\rep{u}$ in $\tck{\Sigma}$ sent to~$u$ by the canonical projection.

Then, let $u$ and $v$ be parallel $p$-cells of $\tck{\Sigma}$ and let $f:\cl{u}\fl\cl{v}$ be a $(p+1)$-cell of $\cl{\Sigma}$. Since $\cl{\Sigma}$ is a $p$-category, we must have $\cl{u}=\cl{v}$ and $f=1_{\cl{u}}$. By definition of $\cl{\Sigma}$, there exists a $(p+1)$-cell $\rep{f}$ in $\tck{\Sigma}$ from~$u$ to $v$, which is sent to $1_{\cl{u}}$ by the canonical projection.

Finally, let $f$ and $g$ be parallel $n$-cells of $\tck{\Sigma}$, for $n>p$. Both $f$ and $g$ must be sent to the same $n$-cell of $\cl{\Sigma}$, so that the only possible $(n+1)$-cell of $\cl{\Sigma}$ between their images is an identity. Since $\Sigma$ is acyclic, there exists an $(n+1)$-cell from $f$ to $g$ in $\tck{\Sigma}$, and this $(n+1)$-cell must be sent to this same identity $(n+1)$-cell of $\cl{\Sigma}$.
\end{proof}

%%%%%%%%%%%%%%%%%%%%%%%%%%
\subsubsection{Polygraphic dimension}

Let us note that every cofibrant $p$-category $\Cr$ admits a presentation, \ie, a partial polygraphic resolution of length $p+1$. Indeed, we can take:
\begin{itemize}
\item for $k<p$, any choice of generating $k$-cells of $\Cr$,
\item the same $p$-cells as $\Cr$,
\item one $(p+1)$-cell from $u$ to $v$, when $u$ and $v$ are $p$-cells of the free $p$-category $\Cr^*$ that are identified by the projection $\Cr^*\twoheadrightarrow\Cr$, \ie, that yield the same $p$-cell of $\Cr$ by composition.
\end{itemize}
Furthermore, every partial polygraphic resolution $\Sigma$ of length $n$ of a $p$-category $\Cr$ can be extended into a partial polygraphic resolution of length $n+1$ of $\Cr$ by adjunction of the homotopy basis of the $n$-spheres of $\tck{\Sigma}$. By iterating this process, we can extend any partial polygraphic resolution into a polygraphic resolution of the same $p$-category. Applied to the generic presentation of a cofibrant $p$-category, we get that any cofibrant $p$-category admits a polygraphic resolution.

If $\Cr$ is a cofibrant $p$-category, the \emph{polygraphic dimension of $\Cr$} is the element $\poldim{\Cr}$ of $\Nb\amalg\ens{\infty}$ defined as follows: if there exists a natural number $n$ such that $\Cr$ admits an aspherical partial polygraphic resolution of length $n$, then $\poldim{\Cr}$ is the smallest of those natural numbers; otherwise $\poldim{\Cr}=\infty$.

%%%%%%%%%%%%%%%%%%%%%%%%%%%%%%%%%%%%%%%%%%%%
\subsubsection{Higher-dimensional finite derivation type}

For $n\geq p$, a $p$-category is of \emph{finite $n$-derivation type} ($\FDT_n$) when it admits a finite partial polygraphic resolution of length $n$. A $p$-category is of \emph{finite $\infty$-derivation type} ($\FDT_{\infty}$) when it admits a finite polygraphic resolution, \ie, when it is~$\FDT_n$ for every $n\geq p$. By extension, for $n<p$, a $p$-category is of \emph{finite $n$-derivation type} when it admits finitely many generating $n$-cells.

In particular, a $p$-category is $\FDT_p$ when it is finitely generated, it is $\FDT_{p+1}$ when it is finitely presented and it is $\FDT_{p+2}$ when it has finite derivation type, a condition introduced by the authors in~\cite{GuiraudMalbos09}. When $p=1$ and for monoids, seen as categories with one $0$-cell, the property $\FDT_3$ corresponds to the finite derivation type condition originally defined by Squier,~\cite{Squier87}, while the property $\FDT_4$ was introduced in~\cite{Pride05}. 

Let us note that $FDT_{n+1}$ is harder to fulfil than $FDT_n$, because of the finiteness condition on $(n+1)$-cells, leading to the following chain of implications:
\[
\FDT_{\infty} 
	\:\Rightarrow\: (\cdots) 
	\:\Rightarrow\: \FDT_{p+2} \:\Rightarrow\: \FDT_{p+1} \:\Rightarrow\: \FDT_p 
	\:\Rightarrow\: (\cdots) 
	\:\Rightarrow\: \FDT_0.
\]   

%%%%%%%%%%%%%%%%%%%%%%%%%%%%%%%%%%%%%%%%%%%%
%%%%%%%%%%%%%%%%%%%%%%%%%%%%%%%%%%%%%%%%%%%%
\section{Normalisation strategies for polygraphs}
\label{sectionStrategies}
%%%%%%%%%%%%%%%%%%%%%%%%%%%%%%%%%%%%%%%%%%%%
%%%%%%%%%%%%%%%%%%%%%%%%%%%%%%%%%%%%%%%%%%%%

%%%%%%%%%%%%%%%%%%%%%%%%%%%%%%%%%%%%%%%%%%%%
\subsection{Strategies in rewriting theory}
\label{Subsection:StrategiesInRewriting}
%%%%%%%%%%%%%%%%%%%%%%%%%%%%%%%%%%%%%%%%%%%%

A rewriting system specifies a set of rules that describe valid replacements of subformulas by other ones,~\cite{Thue14,Newman42}. On some formulas, the rewriting rules may produce conflicts, when two or more rules can be applied. For this reason, to transform a rewriting system into a genuine computation algorithm, one specifies a way to apply the rules in a deterministic way by a \emph{strategy}.

For example, in a word rewriting system, formulas are elements of a free monoid. There are two canonical strategies to reduce words where several rewriting rules apply: the leftmost one and the rightmost one, using first the rewriting rule that can be applied on the leftmost or the rightmost subformula:
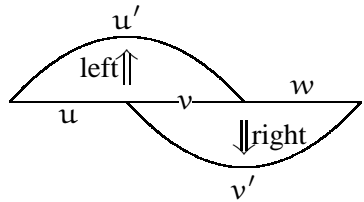
\[
\xymatrix@C=4em@W=0pt@M=0pt{
\strut 
	\ar@{-} [r] _-{u}
	\ar@{-}@/^5ex/ [rr] ^-{u'} _-{}="tgt1"
& \strut
	\ar@{-} [r] |-{v}
	\ar@{-}@/_5ex/ [rr] _-{v'} ^-{}="tgt2"
& \strut 
	\ar@{-} [r] ^-{w}
& \strut
\ar@2 "1,2";"tgt1"!<0pt,-5pt> ^-{\text{left}} 
\ar@2 "1,3";"tgt2"!<0pt,5pt> ^-{\text{right}} 
}
\]
In term rewriting, formulas are morphisms with target the terminal object in a free Lawvere algebraic theory,~\cite{Lawvere63}. Formulas can be represented by trees and rewriting rules replace subtrees by other subtrees. There exist many possible strategies for term rewriting systems. Among them, outermost and innermost strategies are families of strategies that first use rules that apply closer to the root or closer to the leaves of the term, respectively:
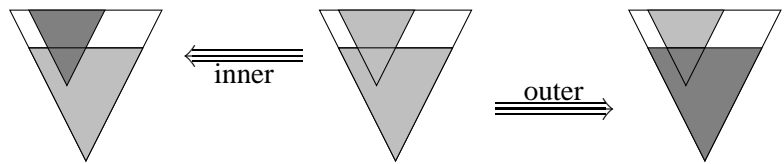
\[
\begin{tikzpicture}
\draw (1,0) -- (2,2) -- (0,2) -- cycle ;;
\draw[fill = lightgray] (1,0) -- (1.75,1.5) -- (0.25,1.5) -- cycle ;;
\draw[fill = gray] (0.75,1) -- (1.25,2) -- (0.25,2) -- cycle ;;
\draw (0.25,1.5) -- (1.75,1.5) ;;
\end{tikzpicture}
\raisebox{40pt}{
\xymatrix@C=4em{
\strut 
& \strut
	\ar@3 [l] ^-{\text{inner}}
}
}
\begin{tikzpicture}
\draw (1,0) -- (2,2) -- (0,2) -- cycle ;;
\draw[fill = lightgray] (1,0) -- (1.75,1.5) -- (0.25,1.5) -- cycle ;;
\draw[fill = lightgray] (0.75,1) -- (1.25,2) -- (0.25,2) -- cycle ;;
\draw (0.25,1.5) -- (1.75,1.5) ;;
\end{tikzpicture}
\raisebox{20pt}{
\xymatrix@C=4em{
\strut
	\ar@3 [r] ^-{\text{outer}}
& \strut
}
}
\begin{tikzpicture}
\draw (1,0) -- (2,2) -- (0,2) -- cycle ;;
\draw[fill = lightgray] (0.75,1) -- (1.25,2) -- (0.25,2) -- cycle ;;
\draw[fill = gray] (1,0) -- (1.75,1.5) -- (0.25,1.5) -- cycle ;;
\draw (0.5,1.5) -- (0.75,1) -- (1,1.5) -- cycle ;;
\end{tikzpicture}
\]
In programming languages based on rewriting mechanisms, such as Caml,~\cite{Caml}, and Haskell,~\cite{Haskell}, strategies are implicitly used by compilers to transform rewriting systems into deterministic algorithms. In that setting, innermost strategies include the call-by-value evaluation, while outermost strategies contain the call-by-need evaluation. Some programming languages, like Tom,~\cite{Tom}, include a dedicated grammar to explicitly construct user-defined strategies.

Several models have been introduced to study the computational properties of strategies. In abstract rewriting, a strategy is defined as a subgraph of the ambient abstract rewriting system. This definition allows the introduction of some properties: for example, a normalisation strategy is a strategy that reaches normal forms,~\cite{Terese03}. Strategies in functional programming languages are usually classified by corresponding notions of strategies in the $\lambda$-calculus,~\cite{Lévy78}. This has led to the axiomatic setting of standardisation theory, where strategies are seen as standardisation systems of rewriting paths,~\cite{Melliès02}.

In this work, we introduce a notion of normalisation strategy for higher-dimensional rewriting systems that, in turn, induces a notion of normal forms in every dimension, together with a homotopically coherent reduction of every cell to its normal form.

%%%%%%%%%%%%%%%%%%%%%%%%%%%%%%%%%%%%%%%%%%%%
\subsection{Normalisation strategies}
%%%%%%%%%%%%%%%%%%%%%%%%%%%%%%%%%%%%%%%%%%%%

Before a formal definition of normalisation strategy, let us give the idea underlying this notion. If $\Sigma$ is an $(\infty,p)$-polygraph, the $p$-category $\cl{\Sigma}$ it presents can be seen as an $(\infty,p)$-category with identity cells only in dimensions $p+1$ and higher. This way, the canonical projection $\pi:\Sigma_p^*\pfl\cl{\Sigma}$ can be extended into an $(\infty,p)$-functor $\pi:\tck{\Sigma}\pfl\cl{\Sigma}$. Given a (non-functorial) section $\iota:\cl{\Sigma}\ifl\Sigma_p^*$ of the canonical projection, a normalisation strategy corresponds to an extension of this section into a $(\infty,p)$-functor $\iota:\cl{\Sigma}\ifl\tck{\Sigma}$, in a suitably weak sense, that satisfies $\pi\iota=\id_{\cl{\Sigma}}$ and $\iota\pi\simeq\id_{\tck{\Sigma}}$, with an explicitly chosen natural isomorphism witnessing this last fact: it follows that $\Sigma$ is a polygraphic resolution of $\cl{\Sigma}$. Let us fix $n$ and $p$ with $0\leq p\leq n\leq\infty$.

%%%
\subsubsection{Sections} 

Let $\Sigma$ be an $(n,p)$-polygraph. A \emph{section of $\Sigma$} is a choice of a representative $p$-cell $\rep{u}:x\fl y$ in $\tck{\Sigma}$, for every  $p$-cell $u:x\fl y$ of $\cl{\Sigma}$, such that 
\[
\rep{1_x} \:=\: 1_x 
\]
holds for every $(p-1)$-cell $x$ of $\cl{\Sigma}$. Such an assignment $u\mapsto\rep{u}$ is not assumed to be functorial with respect to the compositions: in general, such a property can only be required for a $(p,p)$-polygraph, \ie, when $\cl{\Sigma}$ is a free $p$-category.

Since, by hypothesis, the assignment $u\mapsto\rep{u}$ is compatible with the quotient map, it extends to a mapping of each $p$-cell $u$ in $\Sigma^*$ to a parallel $p$-cell in $\Sigma^*$, still denoted by $\rep{u}$, in such a way that the equality $\cl{u}=\cl{v}$ holds in $\cl{\Sigma}$ if, and only if, we have $\rep{u}=\rep{v}$ in $\Sigma^*$. Thereafter, we assume that every $(n,p)$-polygraph comes with an implicitly chosen section.

%%%%%%%%%%%%%%%%%%%%%
\subsubsection{Normalisation strategies}
\label{Subsubsection:NormalisationStrategies}

Let $\Sigma$ be an $(n,p)$-polygraph. A \emph{normalisation strategy for $\Sigma$} is a mapping $\sigma$ of every $k$-cell $f$ of $\tck{\Sigma}$, with $p\leq k<n$, to a $(k+1)$-cell 
\[
\xymatrix{
f
	\ar[r] ^-{\sigma_f}
& {\rep{f}}
}
\]
where, for $k>p$, the notation $\rep{f}$ stands for the $k$-cell $\rep{f} = \sigma_{s(f)} \star_{k-1} \sigma_{t(f)}^-$, such that the following properties are satisfied:
\begin{itemize}
\item for every $k$-cell $f$, with $p\leq k<n$, 
\[
\sigma_{\rep{f}} \:=\: 1_{\rep{f}}
\]
\item for every pair $(f,g)$ of $i$-composable $k$-cells, with $p\leq i<k<n$, 
\[
\sigma_{f\star_i g} \:=\: \sigma_f \star_i\sigma_g \;.
\]
\end{itemize}
An $(n,p)$-polygraph is \emph{normalising} when it admits a normalisation strategy. This property is independent of the chosen section. Indeed, let us consider an $(n,p)$-polygraph $\Sigma$ with two sections $f\mapsto\rep{f}$ and $f\mapsto\tilde{f}$ of $\Sigma$ and let us assume that $\sigma$ is a normalisation strategy for $\Sigma$, equipped with the section $f\mapsto\rep{f}$. Then, one checks that we get a normalisation strategy $\tau$ for the other section by defining~$\tau_f$ as the following composite:
\[
\xymatrix@C=3em{
f
	\ar [r] ^-{\sigma_f}
& {\rep{f}}
	\ar [r] ^-{(\sigma_{\tilde{f}})^-}
& {\tilde{f}} \,.
}
\] 

\begin{lemma}
Let $\Sigma$ be an $(n,p)$-polygraph and let $\sigma$ be a normalisation strategy for $\Sigma$. 
\begin{enumerate}[\bf i)]
\item For every $k$-cell $f$, with $p-1\leq k<n-1$, we have
\[
\sigma_{1_f} \:=\: 1_{1_f}.
\]
\item For every $k$-cell $f$, with $p\leq k<n-1$, we have
\[
\sigma_{\sigma_f} \:=\: 1_{\sigma_f}.
\]
\item For every $k$-cell $f$, with $p<k<n$, we have
\[
\sigma_{f^-} \:=\: f^- \star_{k-1} \sigma_f^- \star_{k-1} {\rep{f}}^-.
\]
\end{enumerate}
\end{lemma}

\begin{proof}
For {\bf i)}, if $x$ is a $(p-1)$-cell, we have $\rep{1_x}=1_x$ by definition. If $f$ is a $k$-cell, with $p\leq k<n-1$, then we have, by definition of $\rep{1_f}$:
\[
\rep{1_f} \:=\: \sigma_{s(1_f)} \star_k \sigma_{t(1_f)}^- \:=\: \sigma_f \star_k \sigma_f^- \:=\: 1_f.
\]
In either case, if $f$ is a $k$-cell, with $p-1\leq k\leq n$, we get $\sigma_{1_f}=1_{1_f}$ by definition of $\sigma$. For {\bf ii)}, if $f$ is a $k$-cell, with $p\leq k<n-1$, then the definition of $\rep{\sigma_f}$ gives:
\[
\rep{\sigma_f} \:=\: \sigma_{s(\sigma_f)} \star_k \sigma_{t(\sigma_f)}^-
	\:=\: \sigma_f \star_k \sigma_{\rep{f}}^-
	\:=\: \sigma_f \star_k 1_{f}^-
	\:=\: \sigma_f.
\]
As a consequence, we get $\sigma_{\sigma_f}=1_{\sigma_f}$. Finally, for {\bf iii)}, if $f$ is a $k$-cell, with $p<k<n$, we have:
\[
\sigma_{f} \star_{k-1} \sigma_{f^-}
	\:=\: \sigma_{f\star_{k-1} f^-}
	\:=\: \sigma_{1_{s(f)}} 
	\:=\: 1_{1_{s(f)}}.
\]
Thus, $\sigma_{f^-}$ is the $(k-1)$-inverse of $\sigma_f$, yielding:
\[
\sigma_{f^-}
\:=\:
	s(\sigma_f) ^- \star_{k-1} \sigma_f^- \star_{k-1} t(\sigma_f) ^- \\
\:=\:
	f^- \star_{k-1} \sigma_f^- \star_{k-1} \rep{f}^-.
	\qedhere
\]
\end{proof}

%%%%%%%%%%%%%%%%%%%%%%%%%%%%%%%%%%%%
\subsection{The case of \pdf{(n,1)}-polygraphs}
\label{Subsection:StrategiesForLow-dimensional}
%%%%%%%%%%%%%%%%%%%%%%%%%%%%%%%%%%%%

Let $\Sigma$ be an $(n,1)$-polygraph. In the lower dimensions, a normalisation strategy $\sigma$ for $\Sigma$ specifies the following assignments:
\begin{itemize}
\item For every $1$-cell $u$ of $\tck{\Sigma}$, a $2$-cell 
\[
\xymatrix@C=3em{
u 
	\ar@2 [r] ^-{\sigma_u} 
& {\rep{u}}
}
\]
of $\tck{\Sigma}$ that satisfies $\sigma_{\rep{u}}=1_{\rep{u}}$ and thus, in particular, $\sigma_{1_{x}}=1_{1_{x}}$ for every $0$-cell $x$ of $\Sigma$.
\item For every $2$-cell $f:u\dfl v$ of $\tck{\Sigma}$, a $3$-cell 
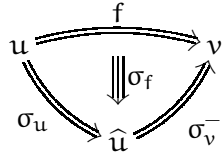
\[
\xymatrix{
u
	\ar@2 @/^/ [rr] ^-{f} _-{}="src"
	\ar@2 @/_/ [dr] _-{\sigma_u} 
&& v
\\
& {\rep{u}}
	\ar@2 @/_/ [ur] _-{\sigma_v^-}	
	\ar@3 "src"!<0pt,-10pt>;[]!<0pt,15pt> ^-{\sigma_f}	
}
\]
of $\tck{\Sigma}$ that satisfies $\sigma_{\rep{f}}=1_{\rep{f}}$ and the following relations:
\begin{itemize}
\item If $u$ is a $1$-cell of $\tck{\Sigma}$, then $\sigma_{1_u}=1_{1_u}$: 
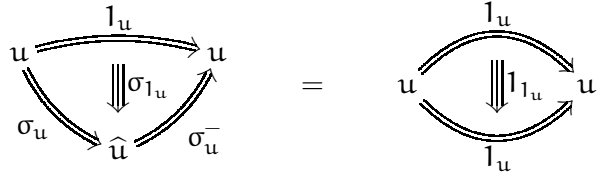
\[
\raisebox{3ex}{
\xymatrix{
u
	\ar@2 @/^/ [rr] ^-{1_u} _-{}="src"
	\ar@2 @/_/ [dr] _-{\sigma_u} 
&& u
\\
& {\rep{u}}
	\ar@2 @/_/ [ur] _{\sigma_u^-}	
	\ar@3 "src"!<0pt,-10pt>;[]!<0pt,15pt> ^-{\sigma_{1_u}}	
}
}
\qquad=\qquad
\xymatrix{
u 
	\ar@2 @/^4ex/ [rr] ^-{1_u} _-{}="s"
	\ar@2 @/_4ex/ [rr] _-{1_u} ^-{}="t"
	\ar@3 "s"!<0pt,-10pt>;"t"!<0pt,10pt> ^-{1_{1_u}}
&& u
}
\] 
\item If $f:u\dfl v$ and $g:v\dfl w$ are $2$-cells in $\tck{\Sigma}$, then $\sigma_{f\star_1 g} = \sigma_f \star_1 \sigma_g$: 
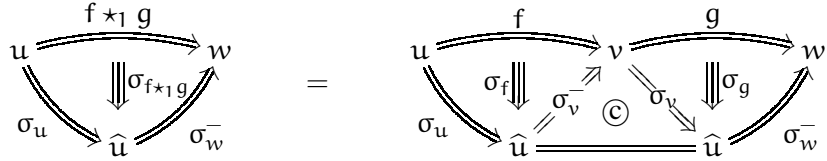
\[
\raisebox{3ex}{
\xymatrix{
u
	\ar@2 @/^/ [rr] ^-{f\star_1 g} _-{}="src"
	\ar@2 @/_/ [dr] _-{\sigma_u} 
&& w
\\
& {\rep{u}}
	\ar@2 @/_/ [ur] _-{\sigma_w^-}	
	\ar@3 "src"!<0pt,-10pt>;[]!<0pt,15pt> ^-{\sigma_{f\star_1 g}}	
}
}
\qquad=\qquad
\raisebox{3ex}{
\xymatrix{
u
	\ar@2 @/^/ [rr] ^-{f} _-{}="src1"
	\ar@2 @/_/ [dr] _-{\sigma_u} 
&& v
	\ar@2 @/^/ [rr] ^-{g} _-{}="src2"
	\ar@2 [dr] |-{\sigma_v}
&& w
\\
& {\rep{u}}
	\ar@2 [ur] |-{\sigma_v^-}	
	\ar@3 "src1"!<0pt,-10pt>;[]!<0pt,15pt> _-{\sigma_{f}}
	\ar@{=} [rr] ^-{}="tgt"
	\ar@{} "1,3";"tgt" |(0.66){\copyright}
&& {\rep{u}}
	\ar@2 @/_/ [ur] _-{\sigma_w^-}	
	\ar@3 "src2"!<0pt,-10pt>;[]!<0pt,15pt> ^-{\sigma_{g}}		
}
}
\]
\item If $f:u\dfl v$ is a $2$-cell in $\tck{\Sigma}$, then $\rep{f}^-=\sigma_v\star_1 \sigma_u^-$ and $\sigma_{f^-} = f^-\star_1 \sigma_f^- \star_1 \rep{f}^-$: 
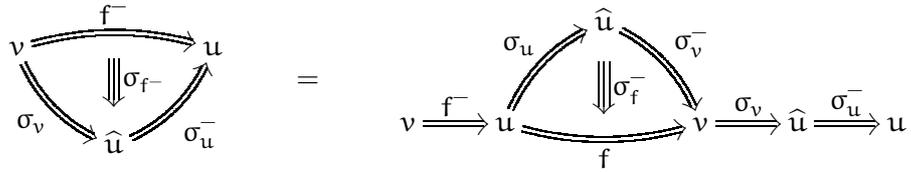
\[
\raisebox{3ex}{
\xymatrix{
v
	\ar@2 @/^/ [rr] ^-{f^-} _-{}="src"
	\ar@2 @/_/ [dr] _-{\sigma_v} 
&& u
\\
& {\rep{u}}
	\ar@2 @/_/ [ur] _{\sigma_u^-}	
	\ar@3 "src"!<0pt,-10pt>;[]!<0pt,15pt> ^-{\sigma_{f^-}}	
}
}
\qquad=\qquad
\raisebox{5ex}{
\xymatrix{
&& {\rep{u}}
	\ar@2 @/^/ [dr] ^-{\sigma_v^-}	
\\
v 
	\ar@2 [r] ^-{f^-} 
& u
	\ar@2 @/^/ [ur] ^{\sigma_u}	
	\ar@2 @/_/ [rr] _-{f} ^-{}="tgt"
&& v
	\ar@3 "1,3"!<0pt,-15pt>;"tgt"!<0pt,10pt> ^-{\sigma_f^-}
	\ar@2 [r] ^-{\sigma_v}
& {\rep{u}}
	\ar@2 [r] ^-{\sigma_u^-}
& u
}
}
\]
\end{itemize}
\item For every $3$-cell $A:f\tfl g:u\dfl v$ of $\tck{\Sigma}$, a $4$-cell
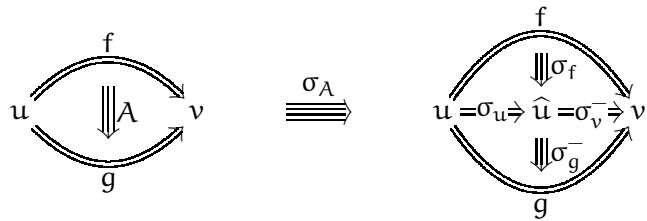
\[
\xymatrix{
u 
	\ar@2 @/^4ex/ [rr] ^-{f} _{}="src"
	\ar@2 @/_4ex/ [rr] _-{g} ^{}="tgt"
&& v
\ar@3 "src"!<0pt,-10pt>;"tgt"!<0pt,10pt> ^-{A}
}
\qquad
\xymatrix{ \strut \ar@4 [r] ^*+\txt{$\sigma_A$} & }
\qquad
\xymatrix{
u
	\ar@2 @/^6ex/ [rr] ^-{f} _(0.505){}="src"
	\ar@2 @/_6ex/ [rr] _-{g} ^(0.505){}="tgt"
	\ar@2 [r] |-{\sigma_u}
& {\rep{u}}
	\ar@2 [r] |-{\sigma_v^-}
	\ar@3 "src"!<0pt,-7.5pt>;[]!<0pt,10pt> ^-{\sigma_f}
	\ar@3 []!<0pt,-10pt>;"tgt"!<0pt,7.5pt> ^-{\sigma_g^-}
& v
}
\]
of $\tck{\Sigma}$ with $\sigma_{\rep{A}}=1_{\rep{A}}$ and such that the following relations hold:
\begin{itemize}
\item If $f$ is a $2$-cell of $\tck{\Sigma}$, then $\sigma_{1_f}=1_{1_f}$: 
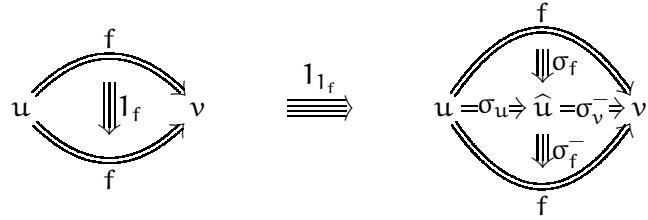
\[
\xymatrix{
u 
	\ar@2 @/^4ex/ [rr] ^-{f} _{}="src"
	\ar@2 @/_4ex/ [rr] _-{f} ^{}="tgt"
&& v
\ar@3 "src"!<0pt,-10pt>;"tgt"!<0pt,10pt> ^-{1_f}
}
\qquad
\xymatrix{ \strut \ar@4 [r] ^*+\txt{$1_{1_f }$} & }
\qquad
\xymatrix{
u
	\ar@2 @/^6ex/ [rr] ^-{f} _(0.505){}="src"
	\ar@2 @/_6ex/ [rr] _-{f} ^(0.505){}="tgt"
	\ar@2 [r] |-{\sigma_u}
& {\rep{u}}
	\ar@2 [r] |-{\sigma_v^-}
	\ar@3 "src"!<0pt,-7.5pt>;[]!<0pt,10pt> ^-{\sigma_f}
	\ar@3 []!<0pt,-10pt>;"tgt"!<0pt,7.5pt> ^-{\sigma_f^-}
& v
}
\]

\item If $A:f\tfl f':u\dfl v$ and $B:g\tfl g':v\dfl w$ are $3$-cells of $\tck{\Sigma}$, then $\sigma_{A\star_1 B} = \sigma_A\star_1\sigma_B$: 
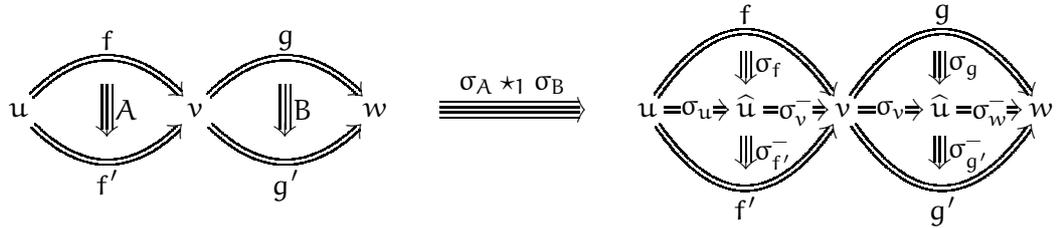
\[
\xymatrix{
u 
	\ar@2 @/^4ex/ [rr] ^-{f} _{}="src1"
	\ar@2 @/_4ex/ [rr] _-{f'} ^{}="tgt1"
	\ar@3 "src1"!<0pt,-10pt>;"tgt1"!<0pt,10pt> ^-{A}
&& v
	\ar@2 @/^4ex/ [rr] ^-{g} _{}="src2"
	\ar@2 @/_4ex/ [rr] _-{g'} ^{}="tgt2"
	\ar@3 "src2"!<0pt,-10pt>;"tgt2"!<0pt,10pt> ^-{B}
&& w
}
\quad
\xymatrix{ \strut \ar@4 [rr] ^*+\txt{$\sigma_A \star_1 \sigma_B$} && }
\quad
\xymatrix{
u
	\ar@2 @/^6ex/ [rr] ^-{f} _(0.505){}="src1"
	\ar@2 @/_6ex/ [rr] _-{f'} ^(0.505){}="tgt1"
	\ar@2 [r] |-{\sigma_u}
& {\rep{u}}
	\ar@2 [r] |-{\sigma_v^-}
	\ar@3 "src1"!<0pt,-7.5pt>;[]!<0pt,10pt> ^-{\sigma_f}
	\ar@3 []!<0pt,-10pt>;"tgt1"!<0pt,7.5pt> ^-{\sigma_{f'}^-}
& v
	\ar@2 @/^6ex/ [rr] ^-{g} _(0.488){}="src2"
	\ar@2 @/_6ex/ [rr] _-{g'} ^(0.488){}="tgt2"
	\ar@2 [r] |-{\sigma_v}
& {\rep{u}}
	\ar@2 [r] |-{\sigma_w^-}
	\ar@3 "src2"!<0pt,-7.5pt>;[]!<0pt,10pt> ^-{\sigma_g}
	\ar@3 []!<0pt,-10pt>;"tgt2"!<0pt,7.5pt> ^-{\sigma_{g'}^-}
& w
}
\]

\item If $A:f\tfl g:u\dfl v$ and $B:g\tfl h:u\dfl v$ are $3$-cells of $\tck{\Sigma}$, then $\sigma_{A\star_2 B} = \sigma_A \star_2 \sigma_B$: 
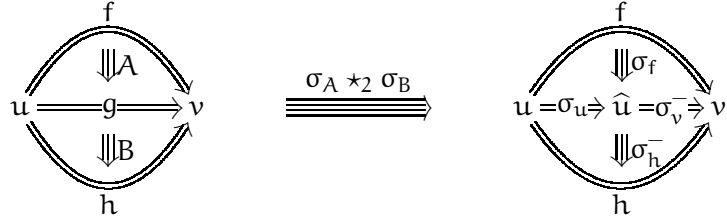
\[
\xymatrix{
u 
	\ar@2 @/^6ex/ [rr] ^-{f} _{}="src1"
	\ar@2 [rr] |-{g} ^{}="tgt1" _{}="src2"
	\ar@2 @/_6ex/ [rr] _-{h} ^{}="tgt2"
	\ar@3 "src1"!<0pt,-7.5pt>;"tgt1"!<0pt,10pt> ^-{A}
	\ar@3 "src2"!<0pt,-10pt>;"tgt2"!<0pt,7.5pt> ^-{B}
&& v
}
\qquad
\xymatrix{ \strut \ar@4 [rr] ^*+\txt{$\sigma_A \star_2 \sigma_B$} && }
\qquad
\xymatrix{
u
	\ar@2 @/^6ex/ [rr] ^-{f} _(0.505){}="src"
	\ar@2 @/_6ex/ [rr] _-{h} ^(0.505){}="tgt"
	\ar@2 [r] |-{\sigma_u}
& {\rep{u}}
	\ar@2 [r] |-{\sigma_v^-}
	\ar@3 "src"!<0pt,-7.5pt>;[]!<0pt,10pt> ^-{\sigma_f}
	\ar@3 []!<0pt,-10pt>;"tgt"!<0pt,7.5pt> ^-{\sigma_h^-}
& v
}
\]

\item If $A:f\tfl g:u\dfl v$ is a $3$-cell of $\tck{\Sigma}$, then $\rep{A}=\sigma_f\star_2\sigma_g^-$ and $\sigma_{A^-}=A^-\star_2\sigma_A^-\star_2\rep{A}^-$:
\[
\xymatrix{
u 
	\ar@2 @/^9ex/ [rr] ^-{g} _{}="src1"
	\ar@2 @/^3ex/ [rr] |(0.4){f} ^{}="tgt1" _{}="src2"
	\ar@2 @/_3ex/ [rr] |(0.4){g} ^{}="tgt2" _{}="src3"
	\ar@2 @/_9ex/ [rr] _-{f} ^{}="tgt3"
&& v
\ar@3 "src1"!<0pt,-7.5pt>;"tgt1"!<0pt,7.5pt> ^-{A^-}
\ar@3 "src2"!<0pt,-7.5pt>;"tgt2"!<0pt,7.5pt> ^-{\rep{A}}
\ar@3 "src3"!<0pt,-7.5pt>;"tgt3"!<0pt,7.5pt> ^-{\rep{A}^-}
}
\qquad
\xymatrix{ 
\strut 
	\ar@4 [rrr] ^*+\txt{$A^-\star_2 \sigma_A^- \star_2 \rep{A}^-$} 
&&& \strut
}
\qquad
\xymatrix{
u 
	\ar@2 @/^9ex/ [rr] ^-{g} _{}="src1"
	\ar@2 @/^3ex/ [rr] |(0.4){f} ^{}="tgt1" _{}="src2"
	\ar@2 @/_3ex/ [rr] |(0.4){g} ^{}="tgt2" _{}="src3"
	\ar@2 @/_9ex/ [rr] _-{f} ^{}="tgt3"
&& v
\ar@3 "src1"!<0pt,-7.5pt>;"tgt1"!<0pt,7.5pt> ^-{A^-}
\ar@3 "src2"!<0pt,-7.5pt>;"tgt2"!<0pt,7.5pt> ^-{A}
\ar@3 "src3"!<0pt,-7.5pt>;"tgt3"!<0pt,7.5pt> ^-{\rep{A}^-}
}
\]
\end{itemize}
\end{itemize}

%%%%%%
\begin{lemma}
\label{lemma : decomposition strategies}
Let $\Sigma$ be an $(n,1)$-polygraph. Normalisation strategies for $\Sigma$ are in bijective correspondence with data made of:
\begin{itemize}
\item a family with one $2$-cell
\[
\sigma_u \::\: u \:\dfl\: \rep{u}
\]
for every $1$-cell $u$ of $\tck{\Sigma}$ such that $\rep{u}\neq u$,
\item a family with one $(k+1)$-cell
\[
\sigma_{u\phi v} \::\: u\phi v \:\fl\: \rep{u\phi v}
\]
for every $1<k<n$, every $k$-cell $\phi$ of $\Sigma$ and every pair $(u,v)$ of $1$-cells of $\tck{\Sigma}$ such that the composite $k$-cell $u\phi v$ is defined.
\end{itemize}
\end{lemma}

\begin{proof}
We proceed by induction on the size of cells of $\tck{\Sigma}$. We already know that a normalisation strategy~$\sigma$ has fixed values on normal forms, identities, inverses and $i$-composites for $i\geq 1$. As a consequence, using the exchange relations, we get that the values of $\sigma$ are entirely and uniquely determined by its values on $1$-cells that are not normal forms and, for every $k\geq 2$, on $k$-cells with shape $u\phi v$, where~$\phi$ is a $k$-cell of $\Sigma$ and $u$ and $v$ are $1$-cells of~$\tck{\Sigma}$.
\end{proof}

%%%%%%
\subsubsection{From normalisation strategies to natural transformations} 

Let $\sigma$ be a normalisation strategy for an $(n,1)$-polygraph $\Sigma$. We define, for every $1$-cell $u$ of $\tck{\Sigma}$, the $1$-cell $u^*$ as $u$ and, by induction on the dimension, for every $k$-cell $f$ in $\tck{\Sigma}$, with $1< k\leq n$, the $k$-cell $f^*$ in~$\tck{\Sigma}$ is given by
\[
f^* \:=\: ( (f\star_1\sigma_{{t_1(f)}^*}) \star_2 \cdots ) \star_{k-1} \sigma^*_{t_{k-1}(f)^*} \;.
\]
For example, for a $2$-cell $f:u\dfl v$, the $2$-cell $f^*$ is
\[
\xymatrix{
u
	\ar@2 [r] ^-{f}
& v
	\ar@2 [r] ^-{\sigma_v}
& {\rep{u}}
}
\]
and, for a $3$-cell $A:f\tfl g:u\dfl v$, the $3$-cell $A^*$ is
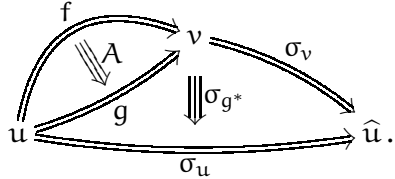
\[
\xymatrix{
&& v
	\ar@2 @/^/ [drr] ^-{\sigma_v}
\\
u
	\ar@2 @/^5ex/ [urr] ^-{f} _-{}="s"
	\ar@2 @/_1ex/ [urr] _-{g} ^-{}="t"
	\ar@2 @/_/ [rrrr] _-{\sigma_u} ^-{}="tgt"
&&&& {\rep{u}} \,.
	\ar@3 "s"!<2.5pt,-5pt>;"t"!<-2.5pt,5pt> ^-{A}
	\ar@3 "1,3"!<0pt,-15pt>;"tgt"!<0pt,10pt> ^-{\sigma_{g^*}}
}
\]
One checks that, for any $k$-cell $f$, with $k>1$, the $k$-cell $f^*$ has source $s(f)^*$ and target $\rep{t(f)}^*$. Moreover, we have $(\rep{f})^*=\rep{f^*}$, which implies $\sigma_{f^*}=\sigma_f^*$. 

Since every $k$-cell of $\tck{\Sigma}$ is invertible for $k\geq 2$, one can recover $\sigma$ from $\sigma^*$, in a unique way, so that the normalisation strategy $\sigma$ is uniquely and entirely determined by the values 
\[
\sigma^*_u \:=\: \sigma_u \::\: u \:\dfl\: \rep{u}
\]
for every $1$-cell $u$ with $u\neq \rep{u}$ and 
\[
\sigma^*_{u\phi v} \::\: (u\phi v)^* \:\fl\: \rep{u\phi v}^*
\]
for every $1<k<n$, every $k$-cell $\phi$ of $\Sigma$ and every $1$-cells $u$ and $v$ of $\tck{\Sigma}$ such that the $k$-cell $u\phi v$ is defined. In the lowest dimensions, the natural transformation form $\sigma^*$ of the strategy $\sigma$ consists of the following data:
\begin{itemize}
\item For every $1$-cell $u$ of $\tck{\Sigma}$, a $2$-cell $\sigma^*_u=\sigma_u$ from $u$ to $\rep{u}$.
\item For every $2$-cell $f:u\dfl v$ of $\tck{\Sigma}$, a $3$-cell $\sigma^*_f$ of $\tck{\Sigma}$ corresponding to a (weak) naturality condition: 
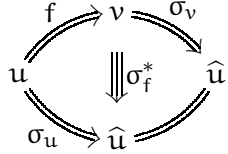
\[
\xymatrix@R=1em{
& v
	\ar@2 @/^/ [dr] ^-{\sigma_v}	
\\
u
	\ar@2 @/^/ [ur] ^{f}
	\ar@2 @/_/ [dr] _-{\sigma_u} 
&& {\rep{u}}
\\
& {\rep{u}}
	\ar@{=} @/_/ [ur] 
\ar@3 "1,2"!<0pt,-15pt>;"3,2"!<0pt,15pt> ^-{\sigma^*_f}	
}
\]
To simplify subsequent diagrams, we draw the $3$-cell $\sigma^*_f$ in a more compact shape, as follows:
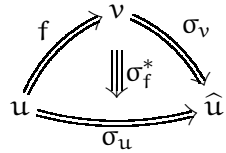
\[
\xymatrix{
& v
	\ar@2 @/^/ [dr] ^-{\sigma_v}	
\\
u
	\ar@2 @/^/ [ur] ^{f}
	\ar@2 @/_/ [rr] _-{\sigma_u} ^-{}="tgt"
&& {\rep{u}}
	\ar@3 "1,2"!<0pt,-15pt>;"tgt"!<0pt,10pt> ^-{\sigma^*_f}	
}
\]
This $3$-cell must satisfy the following relations:
\begin{itemize}
\item if $u$ is a $1$-cell of $\tck{\Sigma}$, then $\sigma^*_{1_u}=1_{\sigma_u}$ holds:
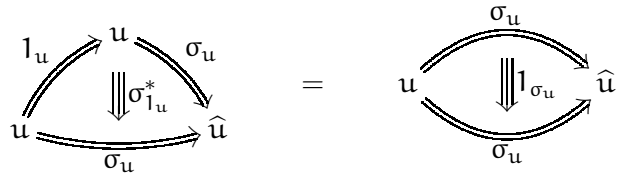
\[
\raisebox{4.5ex}{
\xymatrix{
& u
	\ar@2 @/^/ [dr] ^-{\sigma_u}
\\
u
	\ar@2 @/^/ [ur] ^{1_u}
	\ar@2 @/_/ [rr] _-{\sigma_u} ^-{}="tgt"
&& {\rep{u}}
	\ar@3 "1,2"!<0pt,-15pt>;"tgt"!<0pt,10pt> ^-{\sigma^*_{1_u}} 
}
}
\qquad=\qquad
\xymatrix@!C{
u 
	\ar@2 @/^4ex/ [rr] ^-{\sigma_u} _{}="s"
	\ar@2 @/_4ex/ [rr] _-{\sigma_u} ^{}="t"
	\ar@3 "s"!<0pt,-10pt>;"t"!<0pt,10pt> ^-{1_{\sigma_u}}
&& {\rep{u}}
}
\] 
\item if $f:u\dfl v$ and $g:v\dfl w$ are $2$-cells of $\tck{\Sigma}$, then $\sigma^*_{f\star_1 g} = (f\star_1\sigma^*_g) \star_2 \sigma^*_f$ holds: 
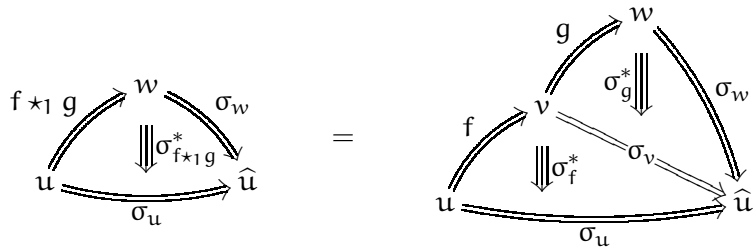
\[
\raisebox{4.5ex}{
\xymatrix{
& w
	\ar@2 @/^/ [dr] ^-{\sigma_w}	
\\
u
	\ar@2 @/^/ [ur] ^{f\star_1 g}
	\ar@2 @/_/ [rr] _-{\sigma_u} ^-{}="tgt"
&& {\rep{u}}
	\ar@3 "1,2"!<0pt,-15pt>;"tgt"!<0pt,10pt> ^-{\sigma^*_{f\star_1 g}} 
}
}
\qquad=\qquad
\raisebox{10ex}{
\xymatrix{
&& w
	\ar@2 @/^/ [ddr] ^-{\sigma_w}
\\
& v 
	\ar@2 @/^/ [ur] ^-{g}
	\ar@2 [drr] |-{\sigma_v} ^-{}="tgt1" 
\\
u 
	\ar@2 @/^/ [ur] ^{f} 
	\ar@2 @/_/ [rrr] _-{\sigma_u} ^(0.325){}="tgt2"
&&& {\rep{u}}
	\ar@3 "1,3"!<0pt,-15pt>;"tgt1"!<0pt,15pt> _-{\sigma^*_g}
	\ar@3 "2,2"!<0pt,-15pt>;"tgt2"!<0pt,10pt> ^-{\sigma^*_f}
}
}
\]
\item if $f:u\dfl v$ is a $2$-cell of $\tck{\Sigma}$, then $\sigma^*_{f^-} = f^-\star_1 (\sigma^*_f)^-$ holds: 
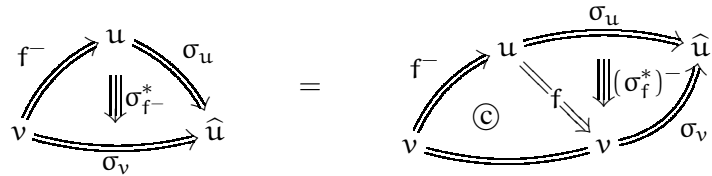
\[
\raisebox{4.5ex}{
\xymatrix{
& u
	\ar@2 @/^/ [dr] ^-{\sigma_u}	
\\
v
	\ar@2 @/^/ [ur] ^-{f^-}
	\ar@2 @/_/ [rr] _-{\sigma_v} ^-{}="tgt"
&& {\rep{u}}
	\ar@3 "1,2"!<0pt,-15pt>;"tgt"!<0pt,10pt> ^-{\sigma^*_{f^-}} 
}
}
\qquad=\qquad
\raisebox{3.5ex}{
\xymatrix{
& u 
	\ar@2 @/^/ [rr] ^-{\sigma_u} _-{}="src"
	\ar@2 [dr] |-{f} _{}="tgt"
&& {\rep{u}}
\\
v
	\ar@2 @/^/ [ur] ^{f^-}
	\ar@{=} @/_/ [rr]
	\ar@{} ;"tgt" |-{\copyright}
&& v
	\ar@2 @/_2ex/ [ur] _{\sigma_v}
	\ar@3 "src"!<0pt,-10pt>;[]!<0pt,15pt> ^-{(\sigma^*_f)^-}
}
}
\]
\end{itemize}
\item For every $3$-cell $A:f\tfl g:u\dfl v$ of $\tck{\Sigma}$, a $4$-cell
\[
\xymatrix{
&& v
	\ar@2 @/^/ [drr] ^-{\sigma_v}
\\
u
	\ar@2 @/^5ex/ [urr] ^-{f} _-{}="s"
	\ar@2 @/_1ex/ [urr] _-{g} ^-{}="t"
	\ar@2 @/_/ [rrrr] _-{\sigma_u} ^-{}="tgt"
&&&& {\rep{u}}
	\ar@3 "s"!<2.5pt,-5pt>;"t"!<-2.5pt,5pt> ^-{A}
	\ar@3 "1,3"!<0pt,-15pt>;"tgt"!<0pt,10pt> ^-{\sigma^*_g}
}
\qquad
\raisebox{-1.9em}{\xymatrix{ \strut \ar@4 [r] ^*+\txt{$\sigma^*_A$} & }}
\qquad
\xymatrix{
& v
	\ar@2 @/^/ [dr] ^-{\sigma_v}	
\\
u
	\ar@2 @/^/ [ur] ^{f}
	\ar@2 @/_/ [rr] _-{\sigma_u} ^-{}="tgt"
&& {\rep{u}}
	\ar@3 "1,2"!<0pt,-15pt>;"tgt"!<0pt,10pt> ^-{\sigma^*_f}
}
\]
of $\tck{\Sigma}$ such that the following relations are satisfied:
\begin{itemize}
\item if $f$ is a $2$-cell of $\tck{\Sigma}$, then $\sigma^*_{1_f}=1_{\sigma^*_f}$ holds: 
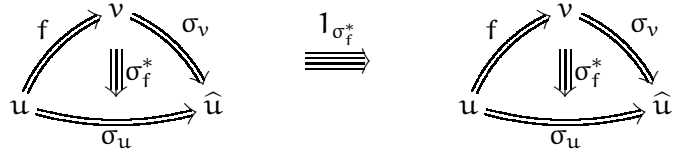
\[
\raisebox{4.5ex}{
\xymatrix{
& v
	\ar@2 @/^/ [dr] ^-{\sigma_v}	
\\
u
	\ar@2 @/^/ [ur] ^{f}
	\ar@2  @/_/ [rr] _-{\sigma_u} ^-{}="tgt"
&& {\rep{u}}
	\ar@3 "1,2"!<0pt,-15pt>;"tgt"!<0pt,10pt> ^-{\sigma^*_f}	
}
}
\qquad
%\raisebox{-1.9em}{
\xymatrix{ \strut \ar@4 [r] ^*+\txt{$1_{\sigma^*_f }$} & }
%}
\qquad
\raisebox{4.5ex}{
\xymatrix{
& v
	\ar@2 @/^/ [dr] ^-{\sigma_v}	
\\
u
	\ar@2 @/^/ [ur] ^{f}
	\ar@2 @/_/ [rr] _-{\sigma_u} ^-{}="tgt"
&& {\rep{u}}
	\ar@3 "1,2"!<0pt,-15pt>;"tgt"!<0pt,10pt> ^-{\sigma^*_f}	
}
}
\]

\item if $A:f\tfl f':u\dfl v$ and $B:g\tfl g':v\dfl w$ are $3$-cells of $\tck{\Sigma}$, then  $\sigma^*_{A\star_1 B} = (f\star_1\sigma^*_B) \star_2 \sigma^*_A$ holds: 
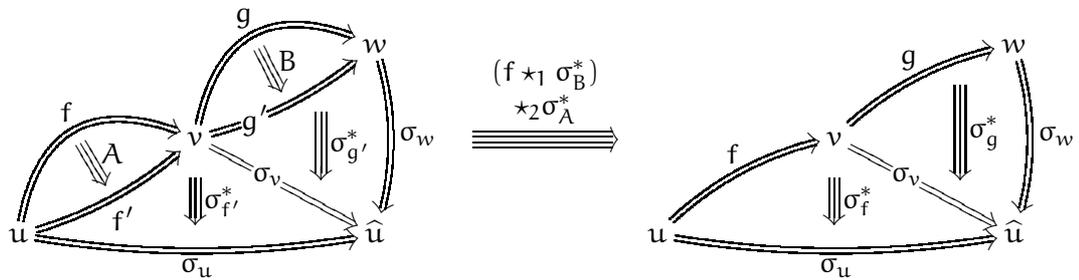
\[
\!\!\!\!\!\!\!\!
\xymatrix{
&&&& w 
	\ar@2 @/^/ [dd] ^-{\sigma_w}
\\
&& v
	\ar@2[drr] |(0.4){\sigma_v} ^(0.7){}="tgt2"
	\ar@2@/^5ex/ [urr] ^-{g} _-{}="s2"
	\ar@2@/_1ex/ [urr] |(0.3){g'} ^-{}="t2" _(0.67){}="src2"
\\
u
	\ar@2@/^5ex/ [urr] ^-{f} _-{}="s"
	\ar@2@/_1ex/ [urr] _-{f'} ^-{}="t"
	\ar@2 @/_/ [rrrr] _-{\sigma_u} ^(0.4965){}="tgt"
&&&& {\rep{u}}
	\ar@3 "s"!<2.5pt,-5pt>;"t"!<-2.5pt,5pt> ^-{A}
	\ar@3 "s2"!<2.5pt,-5pt>;"t2"!<-2.5pt,5pt> ^-{B}
	\ar@3 "2,3"!<0pt,-15pt>;"tgt"!<0pt,10pt> ^-{\sigma^*_{f'}}
	\ar@3 "src2"!<0pt,-10pt>;"tgt2"!<0pt,10pt> ^-{\sigma^*_{g'}}	
}
\:
\raisebox{-3em}{
\xymatrix{ \strut \ar@4 [rr] ^-*+\txt{$(f\star_1 \sigma^*_B)$\\ $\star_2 \sigma^*_A$} && }
}
\:
\xymatrix{
&&&& w 
	\ar@2 @/^/ [dd] ^-{\sigma_w}
\\
&& v
	\ar@2 [drr] |(0.4){\sigma_v} ^(0.7){}="tgt2"
	\ar@2 @/^/ [urr] ^-{g} ^-{}="t2" _(0.733){}="src2"
\\
u
	\ar@2 @/^/ [urr] ^-{f} 
	\ar@2 @/_/ [rrrr] _-{\sigma_u} ^(0.4965){}="tgt"
&&&& {\rep{u}}
	\ar@3 "2,3"!<0pt,-15pt>;"tgt"!<0pt,10pt> ^-{\sigma^*_{f}}
	\ar@3 "src2"!<0pt,-10pt>;"tgt2"!<0pt,10pt> ^-{\sigma^*_{g}}	
}
\]

\item if $A:f\tfl g:u\dfl v$ and $B:g\tfl h:u\dfl v$ are $3$-cells of $\tck{\Sigma}$, \\
then $\sigma^*_{A\star_2 B} = ((A \star_1 \sigma_v) \star_2 \sigma^*_B) \star_3 \sigma^*_A$ holds: 
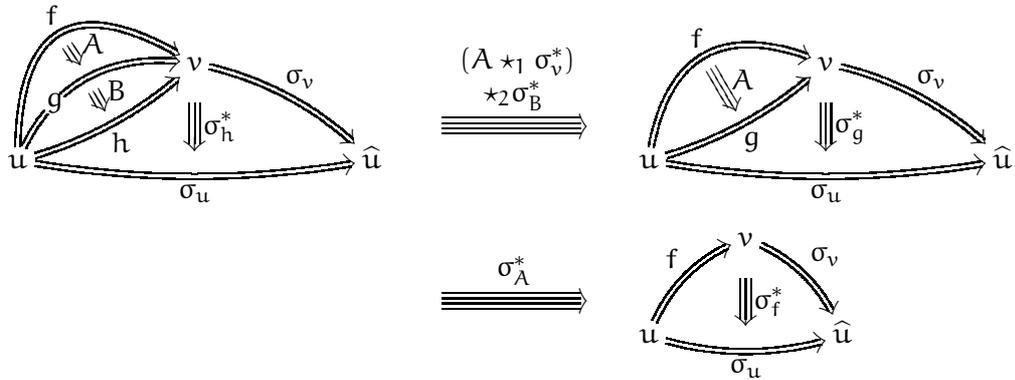
\begin{align*}
\xymatrix{
&& v
	\ar@2 @/^/ [drr] ^-{\sigma_v}
\\
u
	\ar@2@/^7ex/ [urr] ^-{f} _{}="s1"
	\ar@2@/^3ex/ [urr] |(0.3){g} ^{}="t1" _{}="s2"
	\ar@2@/_1ex/ [urr] _-{h} ^-{}="t2"
	\ar@2 @/_/ [rrrr] _-{\sigma_u} ^-{}="tgt"
&&&& {\rep{u}}
	\ar@3 "s1"!<2.5pt,-5pt>;"t1"!<-2.5pt,5pt> ^-{A}
	\ar@3 "s2"!<2.5pt,-5pt>;"t2"!<-2.5pt,5pt> ^-{B}
	\ar@3 "1,3"!<0pt,-15pt>;"tgt"!<0pt,10pt> ^-{\sigma^*_h}
}
\quad
&
\raisebox{-1.9em}{
\xymatrix{ \strut \ar@4 [rr] ^-*+\txt{ $(A \star_1 \sigma^*_v)$ \\ $\star_2 \sigma^*_B$ } && }
}
\quad
\xymatrix{
&& v
	\ar@2 @/^/ [drr] ^-{\sigma_v}
\\
u
	\ar@2@/^5ex/ [urr] ^-{f} _-{}="s"
	\ar@2@/_1ex/ [urr] _-{g} ^-{}="t"
	\ar@2 @/_/ [rrrr] _-{\sigma_u} ^-{}="tgt"
&&&& {\rep{u}}
	\ar@3 "s"!<2.5pt,-5pt>;"t"!<-2.5pt,5pt> ^-{A}
	\ar@3 "1,3"!<0pt,-15pt>;"tgt"!<0pt,10pt> ^-{\sigma^*_g}
}
\\ 
&
\raisebox{-1.9em}{
\xymatrix{ \strut \ar@4 [rr] ^*+\txt{$\sigma^*_A$} && }
}
\quad
\xymatrix{
& v
	\ar@2 @/^/ [dr] ^-{\sigma_v}	
\\
u
	\ar@2 @/^/ [ur] ^{f}
	\ar@2 @/_/ [rr] _-{\sigma_u} ^-{}="tgt"
&& {\rep{u}}
	\ar@3 "1,2"!<0pt,-15pt>;"tgt"!<0pt,10pt> ^-{\sigma^*_f}
}
\end{align*}

\item if $A:f\tfl g:u\dfl v$ is a $3$-cell of $\tck{\Sigma}$, then $\sigma^*_{A^-}=(A^-\star_1\sigma^*_v) \star_2 (\sigma^*_A)^-$ holds:
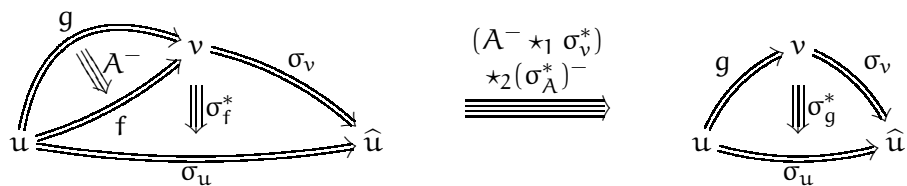
\[
\xymatrix{
&& v
	\ar@2 @/^/ [drr] ^-{\sigma_v}
\\
u
	\ar@2@/^5ex/ [urr] ^-{g} _-{}="s"
	\ar@2@/_1ex/ [urr] _-{f} ^-{}="t"
	\ar@2 @/_/ [rrrr] _-{\sigma_u} ^-{}="tgt"
&&&& {\rep{u}}
	\ar@3 "s"!<2.5pt,-5pt>;"t"!<-2.5pt,5pt> ^-{A^-}
	\ar@3 "1,3"!<0pt,-15pt>;"tgt"!<0pt,10pt> ^-{\sigma^*_f}
}
\qquad
\raisebox{-1.9em}{\xymatrix{ \strut \ar@4 [rr] ^-*+\txt{$(A^-\star_1\sigma^*_v)$ \\ $\star_2 (\sigma^*_A)^-$} && }}
\qquad
\xymatrix{
& v
	\ar@2 @/^/ [dr] ^-{\sigma_v}	
\\
u
	\ar@2 @/^/ [ur] ^{g}
	\ar@2 @/_/ [rr] _-{\sigma_u} ^-{}="tgt"
&& {\rep{u}}
	\ar@3 "1,2"!<0pt,-15pt>;"tgt"!<0pt,10pt> ^-{\sigma^*_g}
}
\]
\end{itemize}
\end{itemize}

%%%
\subsubsection{Left and right normalisation strategies} 

Let $\Sigma$ be an $(n,1)$-polygraph. A normalisation strategy~$\sigma$ for $\Sigma$ is a \emph{left} one when it satisfies the following properties:
\begin{itemize}
\item For every pair $(u,v)$ of $0$-composable $1$-cells of $\tck{\Sigma}$, we have $\sigma_{uv}=\sigma_u v\star_1 \sigma_{\rep{u}v}$, \ie,
\[
\xymatrix{
uv
	\ar@2 @/^/ [rr] ^-{\sigma_{uv}} _(.35){}="src"
	\ar@2 @/_/ [dr] _-{\sigma_u v} 
&& {\rep{uv}}
\\
& {\rep{u} v}
	\ar@2 @/_/ [ur] _-{\sigma_{\rep{u}v}}	
	\ar@{} "src"; ^{\copyright}	
}
\]
\item For every pair $(f,g)$ of $0$-composable $k$-cells of $\tck{\Sigma}$, with $2\leq k\leq p$, $t_1(f)=u'$ and $s_1(g)=v$, we have
\[
\sigma_{fg} \:=\: \sigma_f v \star_1 \sigma_{u' g}.
\]
In particular, when $f:u\dfl u'$ and $g:v\dfl v'$ are $0$-composable $2$-cells of $\tck{\Sigma}$:
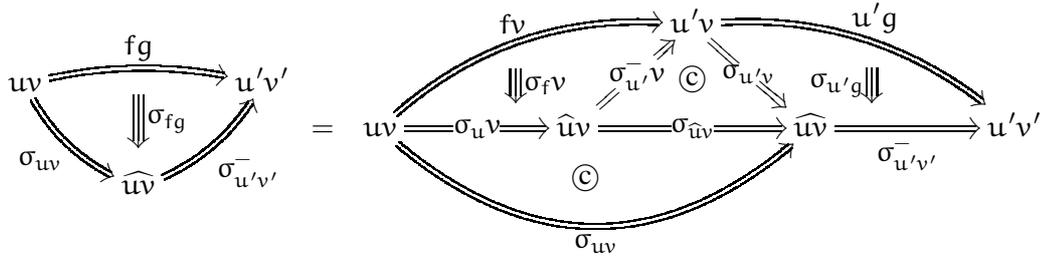
\[
\raisebox{4ex}{
\xymatrix{
uv 
	\ar@2 @/^/ [rr] ^-{fg} _-{}="src"
	\ar@2 @/_/ [dr] _-{\sigma_{uv}}
&& u'v'
\\
& {\rep{uv}}
	\ar@2 @/_/ [ur] _-{\sigma_{u'v'}^-}
	\ar@3 "src"!<0pt,-10pt>;[]!<0pt,15pt> ^-{\sigma_{fg}}
}
}
\: = \:
\raisebox{8.5ex}{
\xymatrix{
&&& u'v
	\ar@2 @/^3ex/ [drrr] ^{u'g} _(0.521){}="src2"
	\ar@2  [dr] |-{\sigma_{u'v}}
\\
uv 
	\ar@2 @/^3ex/ [urrr] ^-{fv} _(0.475){}="src1"
	\ar@2 [rr] |-{\sigma_u v} ^(0.7){}="tgt1"
	\ar@2 @/_7.5ex/ [rrrr] _-{\sigma_{uv}} ^-{}="tgt4"
&& {\rep{u}v} 
	\ar@2 [ur] |-{\sigma_{u'}^- v}
	\ar@2 [rr] |-{\sigma_{\rep{u}v}} ^-{}="tgt3"
&& {\rep{uv}}
	\ar@2 [rr] _-{\sigma_{u'v'}^-}	^(0.3){}="tgt2"
&& u'v'
	\ar@3 "src1"!<0pt,-10pt>;"tgt1"!<0pt,10pt> ^-{\sigma_f v}
	\ar@3 "src2"!<0pt,-10pt>;"tgt2"!<0pt,10pt> _-{\sigma_{u'g}}
	\ar@{} "1,4";"tgt3" |(0.55){\copyright}
	\ar@{} "2,3";"tgt4" |-{\copyright}
}
}
\]
\end{itemize}
In a symmetric way, a normalisation strategy $\sigma$ is a \emph{right} one when it satisfies:
\[
\sigma_{uv} \:=\: u\sigma_v\star_1 \sigma_{u\rep{v}}
\qquad\text{and}\qquad
\sigma_{fg} \:=\: u\sigma_g \star_1 \sigma_{f v'}.
\]
An $(n,1)$-polygraph is \emph{left} (resp. \emph{right}) \emph{normalising} when it admits a left (resp. right) normalisation strategy.

\begin{lemma}
\label{Lemma : relations leftmost}
Let $\Sigma$ be an $(n,1)$-polygraph, let $f$ be a $k$-cell of $\tck{\Sigma}$, for $1<k<n$, with $1$-source~$u$ and $1$-target~$v$, and let $w$ and $w'$ be $1$-cells of $\tck{\Sigma}$ such that $wfw'$ is defined. Then, if $\sigma$ is a left normalisation strategy for $\Sigma$, we have:
\[
\sigma_{wfw'} \:=\: \sigma_{w}uw' \star_1 \sigma_{\rep{w}f} w' \star_1 \sigma_{w}^- vw'
\qquad
\text{and}
\qquad
\sigma_{wfw'}^* \:=\: \sigma_w uw' \star_1 \sigma^*_{\rep{w}f} w' \star_1 \sigma_{\rep{wu}w'}.
\]
Symmetrically, if $\sigma$ is a right normalisation strategy, then we have:
\[
\sigma_{wfw'} \:=\: wu\sigma_{w'} \star_1 w\sigma_{f\rep{w}'} \star_1 wv\sigma_{w'}^-
\qquad
\text{and}
\qquad
\sigma^*_{wfw'} \:=\: wu\sigma_{w'} \star_1 w\sigma^*_{f\rep{w}'} \star_1 \sigma_{w\rep{uw}'}.
\]
\end{lemma}

\begin{proof}
In the case of a left normalisation strategy, the proof for right normalisation strategies being symmetric, we have:
\[
\sigma_{fw'} 
	\:=\: \sigma_{f}w'\star_1 \sigma_{1_{wvw'}} 
	\:=\: \sigma_{f}w' \star_1 1_{1_{wvw'}} 
	\:=\: \sigma_{f}w'.
\]
Then, we use the exchange relation to get:
\[
\sigma_{\sigma_w f} 
	\:=\: \sigma_{wf \star_1 \sigma_w v} 
	\:=\: \sigma_{wf} \star_1 \sigma_{\sigma_w v} 
	\:=\: \sigma_{wf} \star_1 1_{\sigma_w} v 
	\:=\: \sigma_{wf} \star_1 \sigma_w v.
\]
Moreover, the definition of left normalisation strategy implies:
\[
\sigma_{\sigma_w f}
	\:=\: \sigma_{\sigma_w} u \star_1 \sigma_{\rep{w}f} 
	\:=\: \sigma_{w} u \star_1 \sigma_{\rep{w}f}.
\]
From the last two computations, we deduce:
\[
\sigma_{w f} \:=\: \sigma_{w} u \star_1 \sigma_{\rep{w}f} \star_1 \sigma^-_w v.
\]
Combining all the results, we get the required equality:
\[
\sigma_{w f w'} 
	\:=\: \sigma_{w f} w' 
	\:=\: \sigma_{w} u w' \star_1 \sigma_{\rep{w}f} w' \star_1 \sigma^-_w v w'.
\]
For $\sigma^*_{wfw'}$, we proceed as follows:
\begin{align*}
\sigma_{wfw'}^* 
	\:&=\: \sigma_{(wfw')^*} \\
	\:&=\: \sigma_{wf^* w'} \star_1 \sigma_{wvw'} \\
	\:&=\: \sigma_{w}uw' \star_1 \sigma_{\rep{w}f^*} w' \star_1 \sigma_{w}^- vw' \star_1 \sigma_{wvw'} \\
	\:&=\: \sigma_{w}uw' \star_1 \sigma_{\rep{w}f^*} w' \star_1 \sigma_{\rep{w}v} w' \star_1 \sigma_{\rep{wu}w'} \\
	\:&=\: \sigma^*_{w} uw' \star_1 \sigma^*_{\rep{w}f w'} \star_1 \sigma^*_{\rep{wu} w'}.
	\qedhere
\end{align*}
\end{proof}

\begin{corollary}
\label{lemma : decomposition leftmost strategy}
Let $\Sigma$ be an $(n,1)$-polygraph. Left (resp. right) normalisation strategies on $\Sigma$ are in bijective correspondence with the families
\[
\sigma_{\rep{u}\phi} \::\: \rep{u}\phi \:\fl\: \rep{u\phi} 
\qquad\text{(} \:\text{resp. } \:
\sigma_{\phi\rep{u}} \::\: \phi\rep{u} \:\fl\: \rep{\phi u} 
\: \text{)}
\]
and with the families
\[
\sigma^*_{\rep{u}\phi} \::\: (\rep{u}\phi)^* \:\fl\: \rep{\phi u}^*
\qquad\text{(} \:\text{resp. } \:
\sigma^*_{\phi\rep{u}} \::\: (\phi\rep{u})^* \:\fl\: \rep{\phi u}^*
\: \text{)}
\]
of $(k+1)$-cells, indexed by $k$-cells $\phi$ of $\Sigma$, for $1\leq k<n$, and by $1$-cells $u$ of $\cl{\Sigma}$ such that the composite $k$-cell $\rep{u}\phi$ (resp. $\phi\rep{u}$) exists.
\end{corollary}

\begin{proof}
Let us assume, for example, that $\sigma$ is a left normalisation strategy. The property satisfied by $\sigma$ on $1$-cells of $\tck{\Sigma}$ gives, by induction on the size of $1$-cells, that the values of $\sigma$ on $1$-cells of $\tck{\Sigma}$ are determined by the $2$-cells $\sigma_{\rep{u}x}$, for $x$ a $1$-cell of $\Sigma$ and $u$ a $1$-cell of $\cl{\Sigma}$ such that $\rep{u}x$ is defined. Then, Lemma~\ref{lemma : decomposition strategies} tells us that the values of $\sigma$ on higher-dimensional cells of $\tck{\Sigma}$ are determined by the values of $\sigma$  or of $\sigma^*$ on $k$-cells $u\phi v$ of $\tck{\Sigma}$, where $\phi$ is a $k$-cell of $\Sigma$ and $u$ and $v$ are $1$-cells of $\tck{\Sigma}$. We use Lemma~\ref{Lemma : relations leftmost} to conclude.
\end{proof}

\begin{theorem}
\label{MainTheorem1}
Let $\Sigma$ be an $(n,1)$-polygraph. The following assertions are equivalent:
\begin{enumerate}[\bf i)]
\item $\Sigma$ is acyclic,
\item $\Sigma$ is normalising,
\item $\Sigma$ is left normalising,
\item $\Sigma$ is right normalising.
\end{enumerate}
\end{theorem}

\begin{proof}
Let us assume that there exists a normalisation strategy $\sigma$ for $\Sigma$. We consider a $k$-cell $f$ in $\tck{\Sigma}$, for some $1< k <n$. By definition of a normalisation strategy, the $(k+1)$-cell $\sigma_f$ has source $f$ and target~$\rep{f}$. Thus, if $g$ is a $k$-cell which is parallel to $f$, the $(k+1)$-cell $\sigma_f\star_k\sigma_g^-$ of $\tck{\Sigma}$ has source $f$ and target $g$, proving that $\Sigma_{k+1}$ forms a homotopy basis of $\tck{\Sigma}_k$. Hence $\Sigma$ is acyclic.

Conversely, let us assume that $\Sigma$ is acyclic and let us define a right normalisation strategy $\sigma$, the case of a left one being symmetric. By definition of the category $\cl{\Sigma}$, we can choose a $2$-cell 
\[
\sigma_{x\rep{u}} : x\rep{u} \dfl \rep{xu} 
\]
for every $1$-cell $x$ in $\Sigma$ and every $1$-cell $u$ in $\cl{\Sigma}$ such that $x\rep{u}$ is defined. Then, for $1< k<n$, we use the fact that $\Sigma_{k+1}$ is a homotopy basis of $\tck{\Sigma}_k$ to choose an arbitrary $(k+1)$-cell
\[
\sigma_{\phi\rep{u}} \::\: \phi\rep{u} \:\longrightarrow\: \rep{\phi u}
\]
for every $k$-cell $\phi$ in $\Sigma$ and every $1$-cell $u$ in $\cl{\Sigma}$ with $\phi\rep{u}$ is defined. We use Corollary~\ref{lemma : decomposition leftmost strategy} to conclude.
\end{proof}

\begin{corollary}
Let $\C$ be a category and let $n$ be a non-zero natural number. Then $\C$ is $\FDT_n$ if, and only if, there exists a finite, (left, right) normalising $(n,1)$-polygraph presenting~$\C$.
\end{corollary}

%%%%%%%%%%%%%%%%%%%%%%%%%%%%%%%%%%%%%%%%%%%
%%%%%%%%%%%%%%%%%%%%%%%%%%%%%%%%%%%%%%%%%%%
\section{Polygraphic resolutions from convergent presentations}
\label{Section:Convergent2Polygraphs}
%%%%%%%%%%%%%%%%%%%%%%%%%%%%%%%%%%%%%%%%%%%
%%%%%%%%%%%%%%%%%%%%%%%%%%%%%%%%%%%%%%%%%%%
 
%%%%%%%%%%%%%%%%%%%%%%%%%%%%%%%%%%%%%%%%%%%
\subsection{Convergent \pdf{2}-polygraphs}
\label{subsectionRewriting}
%%%%%%%%%%%%%%%%%%%%%%%%%%%%%%%%%%%%%%%%%%%

Let us recall notions and results from rewriting theory for $2$-polygraphs,~\cite{Guiraud06jpaa,GuiraudMalbos09,Mimram10}. Let~$\Sigma$ be a fixed $2$-polygraph.

%%%%%%%%%%%%%%%%%%%%%%%%%%%%%%%%%%%%%%%%%%%
\subsubsection{Rewriting and normal forms}

A \emph{rewriting step of $\Sigma$} is a $2$-cell of the free $2$-category $\Sigma^*$ with shape
\[
\xymatrix@C=4em{
{y}
	\ar [r] ^-{w} 
& {x}
	\ar@/^3ex/ [r] ^-{u} ^{}="src"
	\ar@/_3ex/ [r] _-{v} ^{}="tgt"
	\ar@2 "src"!<0pt,-10pt>;"tgt"!<0pt,10pt> ^-{\phi}
& {x'}
	\ar [r] ^-{w'}
& {y'}
}
\]
where $\phi:u\dfl v$ is a $2$-cell of $\Sigma$ and $w$ and $w'$ are $1$-cells of $\Sigma^*$. A \emph{rewriting sequence of $\Sigma$} is a finite or infinite sequence 
\[
\xymatrix@C=2.5em{
{u_1}
	\ar@2 [r] ^-{f_1}
& {u_2}
	\ar@2 [r] ^-{f_2}
& (\cdots)
	\ar@2 [r] ^-{f_{n-1}}
& {u_n}
	\ar@2 [r] ^-{f_n}
& (\cdots)
}
\]
of rewriting steps. If $\Sigma$ has a non-empty rewriting sequence from $u$ to $v$, we say that \emph{$u$ rewrites into~$v$}. Let us note that every $2$-cell $f$ of $\Sigma^*$ decomposes into a finite rewriting sequence of $\Sigma$, this decomposition being unique up to exchange relations. 

A $1$-cell $u$ of $\Sigma^*$ is a \emph{normal form} when $\Sigma$ has no rewriting step with source $u$. A \emph{normal form of $u$} is a $1$-cell $v$ that is a normal form and such that $u$ rewrites into $v$. A $1$-cell is \emph{reducible} if it is not a normal form.

%%%%%%%%%%%%%%%%%%%%%%%%%%%%%%%%%%%%%%%%%%%
\subsubsection{Branchings}

A \emph{branching of $\Sigma$} is a pair $(f,g)$ of $2$-cells of $\Sigma^*$ with a common source, as in the diagram
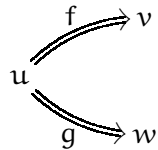
\[
\xymatrix @R=1em @C=3em {
& {v}
\\
{u}
	\ar@2@/^/ [ur] ^-{f}
	\ar@2@/_/ [dr] _-{g}
\\
& {w}
}
\]
The $1$-cell $u$ is the \emph{source} of this branching and the pair $(v,w)$ is its \emph{target}, written $(f,g):u\dfl(v,w)$. We do not distinguish the branchings $(f,g)$ and $(g,f)$. 

A branching $(f,g)$ is \emph{local} when $f$ and $g$ are rewriting steps. Local branchings belong to one of the three following families:
\begin{itemize}
\item \emph{aspherical} branchings have shape
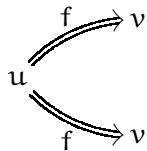
\[
\xymatrix @R=1em @C=3em {
& {v}
\\
{u}
	\ar@2@/^/ [ur] ^-{f}
	\ar@2@/_/ [dr] _-{f}
\\
& {v}
}
\]
with $f:u\dfl v$ a rewriting step of $\Sigma$,

\item \emph{Peiffer} branchings have shape
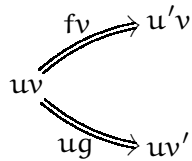
\[
\xymatrix @R=1em @C=3em {
& {u'v}
\\
{uv}
	\ar@2@/^/ [ur] ^-{fv}
	\ar@2@/_/ [dr] _-{ug}
\\
& {uv'}
}
\]
with $f:u\dfl v$ and $g:u'\dfl v'$ rewriting steps of $\Sigma$,

\item \emph{overlapping} branchings are the remaining local branchings.
\end{itemize}
The terms ``aspherical'' and ``Peiffer'' come from the corresponding notions for spherical diagrams in Cayley complexes associated to presentations of groups,~\cite{LyndonSchupp77}, while the term ``critical'' comes from rewriting theory,~\cite{BookOtto93,BaaderNipkow98}. 

Local branchings are compared by the order $\preccurlyeq$ generated by the relations
\[
(f,g) \:\preccurlyeq\: \big( u f v, u g v)
\]
given for any local branching $(f,g)$ and any $1$-cells $u$ and $v$ of $\Sigma^*$ such that $ufv$ exists (and, thus, so does $ugv$). An overlapping local branching that is minimal for the order $\preccurlyeq$ is called a \emph{critical branching}. 

A branching $(f,g)$ is \emph{confluent} when there exists a pair $(f',g')$ of $2$-cells of~$\Sigma^*$ with the following shape: 
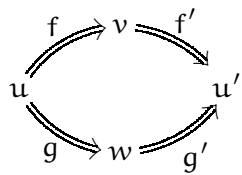
\[
\xymatrix @R=1em {
& v
	\ar@2@/^/ [dr] ^-{f'}
\\
u
	\ar@2@/^/ [ur] ^-{f}
	\ar@2@/_/ [dr] _-{g}
&& u'
\\
& w
	\ar@2@/_/ [ur] _-{g'}
}
\]

%%%%%%%%%%%%%%%%%%%%%%%%%%%%%%%%%%%%%%%%%%%
\subsubsection{Termination, confluence and convergence}

We say that $\Sigma$ \emph{terminates} when it has no infinite rewriting sequence. In that case, every $1$-cell has at least one normal form. Moreover, \emph{Noetherian induction} allows definitions and proofs of properties of $1$-cells by induction on the maximum size of the $2$-cells leading to normal forms. 

We say that $\Sigma$ is \emph{confluent} (resp. \emph{locally confluent}) when all of its branchings (resp. local branching) are confluent. In a confluent $2$-polygraph, every $1$-cell has at most one normal form. A fundamental result of rewriting theory is that local confluence is equivalent to confluence of critical branchings. For terminating $2$-polygraphs, Newman's Lemma ensures that local confluence and confluence are equivalent properties,~\cite{Newman42}.  

We say that $\Sigma$ is \emph{convergent} when it terminates and it is confluent. In that case, every $1$-cell $u$ has a unique normal form. Such a $\Sigma$ is called a \emph{convergent presentation of $\cl{\Sigma}$} and has a canonical section sending $u$ to its normal form $\rep{u}$, so that $\rep{u}=\rep{v}$ holds in $\Sigma^*$ if, and only if, we have $\cl{u}=\cl{v}$ in $\cl{\Sigma}$. As a consequence, a finite and convergent $2$-polygraph $\Sigma$ yields generators for the $1$-cells of the category~$\cl{\Sigma}$, together with a decision procedure for the corresponding word problem (finiteness is used to effectively check that a given $1$-cell is a normal form).

%%%%%%%%%%%%%%%%%%%%%%%%%%%%%%%%%%%%%%%%%%%
\subsubsection{Reduced \pdf{2}-polygraphs}
\label{DefinitionReducedPolygraph}
A $2$-polygraph $\Sigma$ is \emph{reduced} when, for every $2$-cell $\phi:u\dfl v$ in $\Sigma$, the $1$-cell $u$ is a normal form for $\Sigma_2\setminus\ens{\phi}$ and $v$ is a normal form for $\Sigma_2$. Let us note that, in that case, for every $1$-cell $u$ of $\Sigma^*$, there exist finitely many rewriting steps with source~$u$ in $\Sigma^*$: indeed, we have exactly one such $2$-cell for every decomposition $u=vwv'$ such that $w$ is the source of a $2$-cell of $\Sigma$ and the number of decompositions $u=vwv'$ is finite in a free category.

\begin{lemma}
For every (finite) convergent $2$-polygraph, there exists a (finite) Tietze-equivalent, reduced and  convergent $2$-polygraph.
\end{lemma}

\begin{proof}
Let $\Sigma$ be a (finite) convergent $2$-polygraph $\Sigma$. We successively transform $\Sigma$ as follows. First, we replace every $2$-cell $\phi:u\dfl v$ in $\Sigma$ with $\phi':u\dfl\rep{u}$. Then, if there exist several $2$-cells in $\Sigma$ with the same source, we drop all of them but one. Finally, we drop all the remaining $2$-cells whose source is reducible by another $2$-cell. After each step, we check that the (finite) $2$-polygraph we get is convergent and that it is Tietze-equivalent to the former one. Moreover, by construction, the result is a reduced $2$-polygraph.
\end{proof}

\begin{remark}
This result was proved by Métivier for term rewriting systems,~\cite{Metivier83}, and by Squier for word rewriting systems,~\cite{Squier87}. The proof works for any type of rewriting systems, including $n$-polygraphs for any $n$. 
\end{remark} 

\begin{example}
Let $\C$ be a category and let $\N\C$ be the reduced standard polygraphic presentation of~$\C$, \ie, the $2$-polygraph with the following cells: 
\begin{itemize}
\item one $0$-cell for each $0$-cell of $\C$,

\item one $1$-cell $\rep{u}:x\fl y$ for every non-identity $1$-cell $u:x\fl y$ of $\C$,

\item one $2$-cell $\mu_{u,v}:\rep{u}\rep{v}\dfl \rep{uv}$ for every non-identity $1$-cells $u:x\fl y$ and $v:y\fl z$ of $\C$ such that $uv$ is not an identity,

\item one $2$-cell $\mu_{u,v}:\rep{u}\rep{v}\dfl 1_x$ for every non-identity $1$-cells $u:x\fl y$ and $v:y\fl x$ of $\C$ such that $uv=1_x$.
\end{itemize}

\noindent
The $2$-polygraph $\N\C$ is reduced. Let us prove that it is convergent. For termination, one checks that each $2$-cell $\mu_{u,v}$ of $\N\C$ has source of size~$2$ and target of size~$1$ or~$0$. As a consequence, for every non-identity $2$-cell $f:u\dfl v$ of the free $2$-category $\N\C^*$, the size of~$u$ is strictly greater than the size of~$v$.

For confluence, we check that $\N\C$ has one critical branching for every triple $(u,v,w)$ of non-identity composable $1$-cells in $\C$:
\[
\big(\: \mu_{u,v} \rep{w} \:,\: \rep{u}\mu_{v,w} \:\big).
\]
Each of these critical branchings is confluent, with four possible cases, depending on whether $uv$ or $vw$ is an identity or not:
\begin{itemize}
\item if neither $uv$ nor $vw$ is an identity:
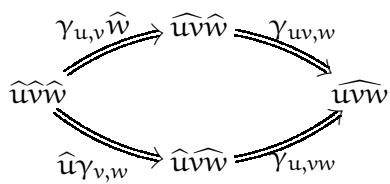
\[
\xymatrix @!C @C=3em @R=1em {
& {\rep{uv}\rep{w}}
	\ar@2 @/^/ [dr] ^{\gamma_{uv,w}}
%        \ar@3 []!<0pt,-15pt>;[dd]!<0pt,15pt> ^-{\gamma_{u,v,w}}
\\
{\rep{u}\rep{v}\rep{w}}
	\ar@2 @/^/ [ur] ^{\gamma_{u,v}\rep{w}}
	\ar@2 @/_/ [dr] _{\rep{u}\gamma_{v,w}}
&& {\rep{uvw}}
\\
& {\rep{u}\rep{vw}}
	\ar@2 @/_/ [ur] _{\gamma_{u,vw}}
}
\]
\item if $uv$ is an identity, but not $vw$:
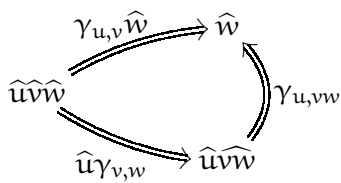
\[
\xymatrix @!C @C=4em @R=1em {
& {\rep{w}}
%        \ar@3 []!<-5pt,-15pt>;[dd]!<-5pt,15pt> _-{\gamma_{u,v,w}}
\\
{\rep{u}\rep{v}\rep{w}}
	\ar@2 @/^/ [ur] ^{\gamma_{u,v}\rep{w}}
	\ar@2 @/_/ [dr] _{\rep{u}\gamma_{v,w}}
\\
& {\rep{u}\rep{vw}}
	\ar@2 @/_3ex/ [uu] _{\gamma_{u,vw}}
}
\]
\item if $uv$ is not an identity, but $vw$ is:
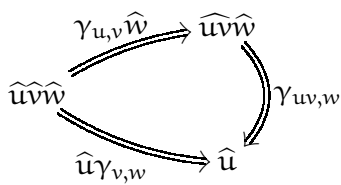
\[
\xymatrix @!C @C=4em @R=1em {
& {\rep{uv}\rep{w}}
	\ar@2 @/^3ex/ [dd] ^{\gamma_{uv,w}}
%        \ar@3 []!<-5pt,-15pt>;[dd]!<-5pt,15pt> _-{\gamma_{u,v,w}}
\\
{\rep{u}\rep{v}\rep{w}}
	\ar@2 @/^/ [ur] ^{\gamma_{u,v}\rep{w}}
	\ar@2 @/_/ [dr] _{\rep{u}\gamma_{v,w}}
\\
& {\rep{u}}
}
\]
\item if $uv$ and $vw$ are identities, and thus $u=uvw=w$:
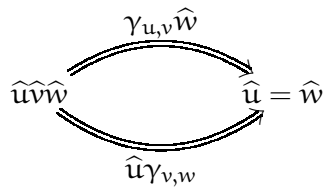
\[
\xymatrix @!C @C=5em @R=1em {
{\rep{u}\rep{v}\rep{w}}
	\ar@2 @/^4ex/ [r] ^{\gamma_{u,v}\rep{w}} _{}="src"
	\ar@2 @/_4ex/ [r] _{\rep{u}\gamma_{v,w}} ^{}="tgt"
& {\rep{u}=\rep{w}}
%		\ar@3 "src"!<-7.5pt,-10pt>;"tgt"!<-7.5pt,10pt> ^-{\gamma_{u,v,w}}
}
\]
\end{itemize}
As a conclusion, the reduced standard presentation $\N\C$ of the category $\C$ is a reduced convergent presentation of $\C$.
\end{example}

%%%%%%%%%%%%%%%%%%%%%%%%%%%
\subsection{Normalisation strategies for convergent \pdf{2}-polygraphs}
%%%%%%%%%%%%%%%%%%%%%%%%%%%

%%%
\subsubsection{The order relation on branchings}

Let $\Sigma$ be a reduced $2$-polygraph and let $u$ be a $1$-cell in $\Sigma^*$. We define the relation $\preceq$ on rewriting steps with source $u$ as follows. If $\phi$ and $\psi$ are $2$-cells of~$\Sigma$ and if $f=v\phi v'$ and $g=w\psi w'$ have source~$u$, then we write $f \preceq g$ when $v$ is smaller than~$w$, \ie, informally, when the part of $u$ on which $f$ acts is more at the left than the part on which $g$ acts. By convention, we denote branchings of $\Sigma$ in increasing order, \ie, $(f,g)$ when $f\preceq g$, which is always possible thanks to the following result.

\begin{lemma}
Let $\Sigma$ be a reduced $2$-polygraph and $u$ be a $1$-cell of $\Sigma^*$. Then the relation~$\preceq$ induces a total ordering on the rewriting steps of $\Sigma$ with source $u$.
\end{lemma}

\begin{proof}
From its definition, we already know that the relation $\preceq$ is reflexive, transitive and total. For antisymmetry, we assume that $f=v\phi v'$ and $g=w\psi w'$ are rewriting steps with source $u$, such that $f\preceq g$ and $g\preceq f$, \ie, such that $v$ and $w$ have the same size. Then, using the fact that $\Sigma_1^*$ is free, we have $v=w$ and either $s(\phi)=s(\psi)$ or $s(\phi)=s(\psi)a$ or $s(\phi)a=s(\psi)$: the latter two cases cannot occur, because $\Sigma$ is reduced and, from that same hypothesis we get, in the first case, that $\phi=\psi$, hence that $f=g$. 
\end{proof}

%%%
\subsubsection{The leftmost and rightmost normalisation strategies} 

Let $\Sigma$ be a reduced $2$-polygraph and let~$u$ be a reducible $1$-cell of $\Sigma^*$ that is not a normal form. The \emph{leftmost} and the \emph{rightmost} rewriting steps on~$u$ are denoted by $\lambda_u$ and $\rho_u$ and defined as the minimum and the maximum elements for $\preceq$ of the (finite, non-empty) set of rewriting steps of $\Sigma$ with source $u$. We note that, if $u$ and $v$ are reducible and composable $1$-cells of $\Sigma^*$, then we have:
\[
\lambda_{uv} \:=\: \lambda_{u}v
\qquad\text{and}\qquad
\rho_{uv} \:=\: u\rho_v.
\]
When $\Sigma$ terminates, the \emph{leftmost normalisation strategy of $\Sigma$} is the normalisation strategy $\sigma$ defined by Noetherian induction as follows. On normal forms, it is given by
\[
\sigma_{\rep{u}} \:=\: 1_{\rep{u}}
\]
and, on reducible $1$-cells, by 
\[
\sigma_u \:=\: \lambda_u \star_1 \sigma_{t(\lambda_u)}.
\]
One defines the \emph{rightmost normal form of $\Sigma$} in a similar way by replacing the leftmost rewriting step by the rightmost one in the case of reducible $1$-cells.

\begin{lemma}
The leftmost (resp. rightmost) normalisation strategy $\sigma$ is a left (resp. right) normalisation strategy for $\Sigma$, seen as a $(2,1)$-polygraph, with the property that, for every $1$-cell $u$, the $2$-cell $\sigma_u$ lives in $\Sigma^*\subseteq \tck{\Sigma}$.
\end{lemma}

\begin{proof}
Let us assume that $\sigma$ is the leftmost normalisation strategy, the proof in the rightmost case being symmetric. We must prove that, for every composable $1$-cells $u$ and $v$ of $\Sigma^*$, the following relation holds:
\[
\sigma_{uv} \:=\: \sigma_u v\star_1 \sigma_{\rep{u}v}.
\]
We proceed by Noetherian induction on the $1$-cell $u$. If $u$ is a normal form, then $\sigma_u=1_u$ and $\sigma_{\rep{u}v}=\sigma_{uv}$, so that the relation is satisfied. Otherwise, we have, using the definition of $\sigma$ and the properties of~$\lambda$:
\[
\sigma_{uv} 
	\:=\: \lambda_{uv}\star_1\sigma_{t(\lambda_{uv})} 
	\:=\: \lambda_u v \star_1 \sigma_{t(\lambda_u) v}.
\]
We apply the induction hypothesis to $t(\lambda_u)v$ to get:
\[
\sigma_{uv}
	\:=\: \lambda_u v \star_1 \sigma_{t(\lambda_u)}v \star_1 \sigma_{\rep{u}v}
	\:=\: \sigma_u v \star_1 \sigma_{\rep{u}v}.
\]
The fact that $\sigma_u$ is in $\Sigma^*$ is also proved by Noetherian induction on $u$, using the definition of $\sigma$ and the facts that both $1_{\rep{u}}$ and $\lambda_u$ are $2$-cells of $\Sigma^*$.
\end{proof}

\begin{remark}
A reduced and terminating $2$-polygraph can have several left or right strategies, beside the leftmost and the rightmost ones. Indeed, let us consider the reduced and terminating $2$-polygraph $\Sigma$ with one $0$-cell, three $1$-cells $a$, $b$ and $c$ and the following three $2$-cells:
\[
\xymatrix{
aac \ar@2[r]^-{\alpha} & a 
}
\qquad
\xymatrix{
bb \ar@2[r]^-{\beta} & cc
}
\qquad
\xymatrix{
acc \ar@2[r]^-{\gamma} & c.
}
\]
Let us prove that $\Sigma$ admits at least two different left normalisation strategies. For that, we examine the $1$-cell $aabb$ and all the $2$-cells of $\Sigma^*$ from $aabb$ to its normal form $ac$:
\[
\xymatrix{
aabb 
	\ar@2[rr] ^-{aa\beta}
&& aacc
	\ar@2@/^3ex/[rr] ^-{\alpha c}
	\ar@2@/_3ex/[rr] _-{a\gamma}
&& ac	.
}
\]
Thus, if $\sigma$ is a normalisation strategy, the $2$-cell $\sigma_{aabb}$ can be either $aa\beta \star_1 \alpha c$ or $aa\beta \star_1 a\gamma$. Since the $1$-cells $a$, $aa$ and $aab$ are normal forms, assuming that $\sigma$ is a left strategy still leaves us with the same choice. Hence, we can define a left normalisation strategy $\sigma$ for $\Sigma$ as the leftmost normalisation strategy on every $1$-cell of $\Sigma^*$, except for $aabb$ where it is given by
\[
\sigma_{aabb} \:=\: aa\beta \star_1 a\gamma.
\]
Thus, we have a left normalisation strategy for $\Sigma$, distinct from the leftmost normalisation strategy: indeed, the latter would send $aabb$ to $aa\beta\star_1\alpha c$.

Let us note that this phenomenon does not come from the fact that $\Sigma$ is not confluent, since we can add the $2$-cell $\delta:bcc\dfl ccb$ to $\Sigma$ to get a reduced, convergent $2$-polygraph which still has at least two different left normalisation strategies. From $\Sigma$, we can build a symmetric (for $\star_0$) $2$-polygraph that admits at least two different right normalisation strategies.

However, we can ensure that, if $\sigma$ is a left (resp. right) normalisation strategy for a reduced and terminating $2$-polygraph $\Sigma$ with the property that, for every $1$-cell $u$ of $\Sigma^*$, the $2$-cell $\sigma_u$ is in $\Sigma^*$, then this same $2$-cell admits a decomposition
\[
\sigma_u \:=\: \lambda_u \star_1 g_u 
\qquad\qquad
\text{(resp. }\:
\sigma_u \:=\: \rho_u \star_1 g_u 
\:\text{)}
\]
with $g_u$ a $2$-cell of $\Sigma^*$. Indeed, if $\sigma$ is a left strategy, we consider the decomposition $\lambda_u \:=\: v\phi w$. By definition of $\lambda_u$, the $1$-cell $vs(\phi)$ is the source of only one rewriting step of $\Sigma$, namely $v\phi$. Hence, since $\sigma_{vs(\phi)}$ is a $2$-cell of $\Sigma^*$ with source $vs(\phi)$, it admits a decomposition 
\[
\sigma_{vs(\phi)} \:=\: \lambda_{v s(\phi)} \star_1 h_u
\]
with $h_u$ a $2$-cell of $\Sigma^*$. The $2$-cell $g_u$ of $\Sigma^*$ is given by
\[
g_u \:=\: h_u w \star_1 \sigma_{\rep{v s(\phi)}w}
\]
and use the hypothesis on $\sigma$ to get:
\[
\sigma_u 
	\:=\: \sigma_{vs(\phi)}w \star_1 \sigma_{\rep{v s(\phi)}w}
	\:=\: \lambda_{v s(\phi)} w \star_1 g_u 
	\:=\: \lambda_u \star_1 g_u.
\]
The case of a right normalisation strategy is symmetric.
\end{remark}

\begin{example}
Let $\C$ be a category and let $\N\C$ be its reduced standard presentation. A generic $1$-cell of $\N\C$ is a composite $\rep{u}_1\cdots\rep{u}_n$ of non-identity $1$-cells of $\C$. In the case where no partial composition $u_iu_{i+1}\cdots u_{j}$ is an identity in $\C$, the leftmost reduction strategy $\sigma$ of $\C$ is given on $\rep{u}_1\cdots\rep{u}_n$ by:
\[
\sigma_{\rep{u}_1\cdots\rep{u}_n} 
	\:=\: \mu_{u_1,u_2} \rep{u}_3\cdots\rep{u}_n
		\:\star_1\: \mu_{u_1u_2, u_3} \rep{u}_4\cdots \rep{u}_n
		\:\star_1\: \cdots
		\:\star_1\: \mu_{u_1\cdots u_{n-1},u_n}. 
\]
If some partial composition is an identity in $\C$, we get $\sigma_{\rep{u}_1\cdots\rep{u}_n}$ by removing the corresponding $2$-cell $\mu$. For example, in the case $n=4$, $u_1u_2=1$ and $u_3u_4\neq 1$, we have:
\[
\sigma_{\rep{u}_1\rep{u}_2\rep{u}_3\rep{u}_4} 
	\:=\: \mu_{u_1,u_2} \rep{u}_3\rep{u}_4 \:\star_1\: \mu_{u_3,u_4} 
	\:=\: \mu_{u_1,u_2}\mu_{u_3,u_4}.
\]
The rightmost normalisation strategy of $\N\C$ is given in a symmetric way.
\end{example}

%%%%%%%%%%%%%%%%%%%%%%%%%%%%%%%%%%%%%%%%%%%
\subsection{The acyclic \pdf{(3,1)}-polygraph of generating confluences} 
\label{subsectionStandardisationBases}
%%%%%%%%%%%%%%%%%%%%%%%%%%%%%%%%%%%%%%%%%%%

We fix a reduced convergent $2$-polygraph $\Sigma$ equipped with its rightmost normalisation strategy~$\sigma$.

%%%%%%%%%%%%%%%%%%%%%%%%%%%%%%%%%%%%%%%%%%%
\subsubsection{Critical branchings of \pdf{2}-polygraphs}

By case analysis on the source of critical branchings of $\Sigma$, we can conclude that they must have one of the following two shapes
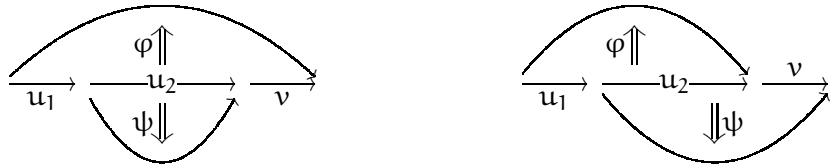
\[
\xymatrix{
\strut 
	\ar[r] _-{u_1}
	\ar@/^6ex/ [rrrr] _-{}="t1"
& \strut
	\ar[rr] |-{u_2} ^-{}="s1" _-{}="s2"
	\ar@/_6ex/ [rr] ^-{}="t2"
&& \strut
	\ar[r] _-{v}
& \strut
\ar@2 "s1"!<0pt,7.5pt>;"t1"!<0pt,-7.5pt> ^-{\phi}
\ar@2 "s2"!<0pt,-7.5pt>;"t2"!<0pt,7.5pt> _-{\psi}
}
\qquad\qquad\qquad
\xymatrix{
\strut 
	\ar[r] _-{u_1}
	\ar@/^6ex/ [rrr] _-{}="t1"
& \strut
	\ar[rr] |-{u_2} _(0.75){}="s2" ^(0.25){}="s1"
	\ar@/_6ex/ [rrr] ^-{}="t2"
&& \strut
	\ar[r] ^-{v}
& \strut
\ar@2 "s1"!<0pt,7.5pt>;"t1"!<0pt,-7.5pt> ^-{\phi}
\ar@2 "s2"!<0pt,-7.5pt>;"t2"!<0pt,7.5pt> ^-{\psi}
}
\]
where $\phi$, $\psi$ are $2$-cells of $\Sigma$. 

We note that, if $\Sigma$ is finite, then it has finitely many critical branchings: indeed, in that case there are finitely many pairs of $2$-cells of $\Sigma$ and, for $\phi$ and $\psi$ fixed, there are finitely many ways to make their sources (which are $1$-cells of a free category) overlap as in one of the two diagrams.

In fact, the $2$-polygraph $\Sigma$ being reduced, the first case cannot occur since, otherwise, the source of~$\phi$ would be reducible by $\psi$. Thus, every critical branching of $\Sigma$ must have shape $(\phi v, u_1 \psi)$. We write the branching in that order since, by definition of $\preceq$, we have $\phi v \preceq u_1 \psi$. 

We also note that the $1$-cells $u_1$, $u_2$ and $v$ are normal forms and cannot be identities. Indeed, they are normal forms since, otherwise, at least one of the sources of $\phi$ and of $\psi$ would be reducible by another $2$-cell, preventing $\Sigma$ from being reduced. If $u_2$ was an identity, then the branching would be Peiffer. Moreover, if $u_1$ (resp. $v$) was an identity, then the source of $\psi$ (resp. $\phi$) would be reducible by~$\phi$ (resp.~$\psi$).

Finally, if we write $u=u_1u_2$, the definitions of $\lambda_{uv}$ and of $\rho_{uv}$ imply that we have: 
\[
\lambda_{uv} \:=\: \phi v
\qquad\qquad\text{and}\qquad\qquad
\rho_{uv} \:=\: u_1\psi.
\] 
From all those observations, we conclude that every critical branching $b$ of $\Sigma$ must have shape 
\[
b \:=\: \big( \: \phi \rep{v}, \:\rho_{u \rep{v}} \:\big)
\]
where $u$ and $v$ are composable $1$-cells of $\Sigma^*$ and where $\phi$ is a $2$-cell of $\Sigma$ with source $u$.

%%%%%%%%%%%%%%%%%%%%%%%%%%%%%%%%%%%%%%%%%%%
\subsubsection{The basis of generating confluences}

The \emph{basis of generating confluences of $\Sigma$} is the cellular extension $\crit_2(\Sigma)$ of $\tck{\Sigma}$ made of one $3$-cell
\[
\xymatrix @R=1em {
& {\rep{u}\rep{v}}
	\ar@2 @/^/ [dr] ^-{\sigma_{\rep{u}\rep{v}}}
\\
{u\rep{v}} 
	\ar@2 @/^/ [ur] ^-{\phi\rep{v}}
	\ar@2 @/_3ex/ [rr] _-{\sigma_{u\rep{v}}} ^-{}="tgt"
&& {\rep{uv}}
\ar@3 "1,2"!<0pt,-15pt>;"tgt"!<0pt,10pt> ^-{\omega_b}
}
\]
for every critical branching $b=(\phi\rep{v},\rho_{u\rep{v}})$ of $\Sigma$. Alternatively, since, for a $2$-cell $f$, we have defined $f^*$ as the $2$-cell $f\star_1 \sigma_{t(f)}$, the $3$-cell $\omega_b$ can be pictured in the following, more compact way:
\[
\xymatrix@!C{
{u\rep{v}}
	\ar@2 @/^4ex/ [rr] ^-{(\phi\rep{v})^*} _-{}="s"
	\ar@2 @/_4ex/ [rr] _-{\rep{\phi v}^*} ^-{}="t"
&& {\rep{uv}}
\ar@3 "s"!<0pt,-10pt>;"t"!<0pt,10pt> ^-{\omega_b}
}
\]

\begin{lemma}
The rightmost normalisation strategy of $\Sigma$ extends to a right normalisation strategy of~$\crit_2(\Sigma)$.
\end{lemma}

\begin{proof}
From Corollary~\ref{lemma : decomposition leftmost strategy}, we know that it is sufficient to define a $3$-cell 
\[
\sigma^*_{\phi\rep{w}} \::\: (\phi\rep{w})^* \:\tfl\: \rep{\phi w}^* 
\]
of $\tck{\crit_2(\Sigma)}$ for every $2$-cell $\phi:v\dfl\rep{v}$ of $\Sigma$ and every $1$-cell $w$ in $\cl{\Sigma}$ such that $u=v\rep{w}$ exists. We note that, by definition, we have:
\[
(\phi\rep{w})^* \:=\: \phi\rep{w}\star_1 \sigma_{\rep{v}\rep{w}}
\qquad\text{and}\qquad
\rep{\phi w}^* \:=\: \sigma_u \:=\: \rho_u \star_1 \sigma_{t(\rho_u)}.
\]
Let us proceed by case analysis on the type of the local branching $b=(\phi\rep{w},\rho_u)$.

\begin{itemize}

\item If $b$ is aspherical, then $\rho_u=\phi\rep{w}$. In that case, we define $
\sigma^*_{\phi\rep{w}} \:=\: 1_{(\phi\rep{w})^*}
$.

\item By hypothesis, the branching $b$ cannot be a Peiffer branching. Indeed, the source of $\rho_u$ cannot be entirely contained in the normal form $\rep{w}$.

\item Otherwise, we have $\rep{w}=\rep{w}_1\rep{w}_2$ and $b_1=(\phi\rep{w}_1,\rho_{v\rep{w}_1})$ is a critical branching of $\Sigma$. In that case, we define $\sigma^*_{\phi\rep{w}}$ as the composite $3$-cell
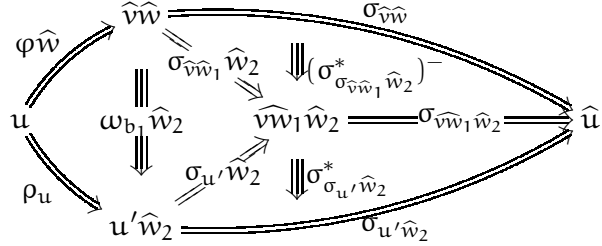
\[
\xymatrix{
& {\rep{v}\rep{w}}
	\ar@2 [dr] |-{\sigma_{\rep{v}\rep{w}_1}\rep{w}_2}
	\ar@3 []!<0pt,-20pt>;[dd]!<0pt,20pt> |-{\omega_{b_1}\rep{w}_2}
	\ar@2 @/^3ex/ [drrrr] ^-{\sigma_{\rep{v}\rep{w}}} _(0.33){}="src1"
\\
u 
	\ar@2@/^/ [ur] ^-{\phi\rep{w}}
	\ar@2@/_/ [dr] _-{\rho_u}
&& {\rep{vw}_1\rep{w}_2}
	\ar@2 [rrr] |-{\sigma_{\rep{vw}_1\rep{w}_2}} 
&&& {\rep{u}}
\\
& {u'\rep{w}_2}
	\ar@2 [ur] |-{\sigma_{u'}\rep{w}_2}
	\ar@2 @/_3ex/ [urrrr] _-{\sigma_{u'\rep{w}_2}} ^(0.328){}="tgt2"
\ar@3 "src1"!<0pt,-10pt>;"2,3"!<0pt,15pt> ^(0.9){(\sigma^*_{\sigma_{\rep{v}\rep{w}_1}\rep{w}_2})^-}
\ar@3 "2,3"!<0pt,-15pt>;"tgt2"!<0pt,10pt> ^-{\sigma^*_{\sigma_{u'}\rep{w}_2}}
}
\]
of $\tck{\crit_2(\Sigma)}$. \qedhere
\end{itemize}
\end{proof}

\begin{proposition}
The $(3,1)$-polygraph $\crit_2(\Sigma)$ is acyclic.
\end{proposition}

\begin{remark}
This result is already contained in~\cite{GuiraudMalbos09}, with a different proof and in a more general form: the generating confluences of a convergent $n$-polygraph $\Sigma$ form a homotopy basis of the $(n,n-1)$-category~$\tck{\Sigma}$.
\end{remark}

\begin{corollary}
A category with a finite convergent presentation is $\FDT_3$. 
\end{corollary}

\noindent In particular, we recover Squier's result: a monoid with a finite convergent presentation has finite derivation type,~\cite{Squier94}.

\begin{example}
Let $\C$ be a category. We consider the reduced standard presentation $\N\C$ of $\C$, equipped with the rightmost normalisation strategy. The basis of generating confluences of $\N\C$ has one $3$-cell $\mu_{u,v,w}$ for every composable non-identity $1$-cells $u$, $v$ and $w$ of $\C$. The shape of this $3$-cell depends on whether $uv$ or $vw$ is an identity or not:
\begin{itemize}
\item if neither $uv$ nor $vw$ is an identity:
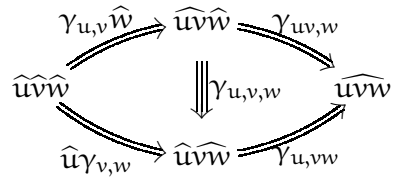
\[
\xymatrix @!C @C=3em @R=1em {
& {\rep{uv}\rep{w}}
	\ar@2 @/^/ [dr] ^{\gamma_{uv,w}}
        \ar@3 []!<0pt,-15pt>;[dd]!<0pt,15pt> ^{\gamma_{u,v,w}}
\\
{\rep{u}\rep{v}\rep{w}}
	\ar@2 @/^/ [ur] ^{\gamma_{u,v}\rep{w}}
	\ar@2 @/_/ [dr] _{\rep{u}\gamma_{v,w}}
&& {\rep{uvw}}
\\
& {\rep{u}\rep{vw}}
	\ar@2 @/_/ [ur] _{\gamma_{u,vw}}
}
\]
\item if $uv$ is an identity, but not $vw$:
\[
\xymatrix @!C @C=4em @R=1em {
& {\rep{w}}
        \ar@3 []!<-5pt,-15pt>;[dd]!<-5pt,15pt> _{\gamma_{u,v,w}}
\\
{\rep{u}\rep{v}\rep{w}}
	\ar@2 @/^/ [ur] ^{\gamma_{u,v}\rep{w}}
	\ar@2 @/_/ [dr] _{\rep{u}\gamma_{v,w}}
\\
& {\rep{u}\rep{vw}}
	\ar@2 @/_3ex/ [uu] _{\gamma_{u,vw}}
}
\]
\item if $uv$ is not an identity, but $vw$ is:
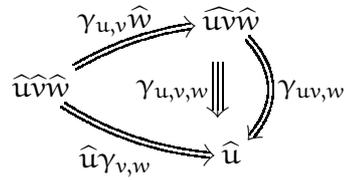
\[
\xymatrix @!C @C=4em @R=1em {
& {\rep{uv}\rep{w}}
	\ar@2 @/^3ex/ [dd] ^{\gamma_{uv,w}}
        \ar@3 []!<-5pt,-15pt>;[dd]!<-5pt,15pt> _{\gamma_{u,v,w}}
\\
{\rep{u}\rep{v}\rep{w}}
	\ar@2 @/^/ [ur] ^{\gamma_{u,v}\rep{w}}
	\ar@2 @/_/ [dr] _{\rep{u}\gamma_{v,w}}
\\
& {\rep{u}}
}
\]
\item if $uv$ and $vw$ are identities, and thus $u=uvw=w$:
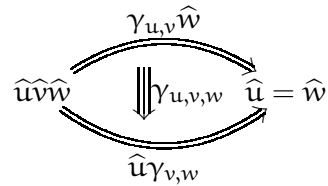
\[
\xymatrix @!C @C=5em @R=1em {
{\rep{u}\rep{v}\rep{w}}
	\ar@2 @/^4ex/ [r] ^{\gamma_{u,v}\rep{w}} _{}="src"
	\ar@2 @/_4ex/ [r] _{\rep{u}\gamma_{v,w}} ^{}="tgt"
& {\rep{u}=\rep{w}}
		\ar@3 "src"!<-7.5pt,-10pt>;"tgt"!<-7.5pt,10pt> ^{\gamma_{u,v,w}}
}
\]
\end{itemize}
The basis of generating confluences in the case of the leftmost normalisation strategy has the same $3$-cells, with source and target reversed.
\end{example}

%%%%%%%%%%%%%%%%%%%%%%%%%%%%%%%%%%%%%%%%%%%
\subsection{The acyclic \pdf{(4,1)}-polygraph of generating triple confluences} 
%%%%%%%%%%%%%%%%%%%%%%%%%%%%%%%%%%%%%%%%%%%

Let $\Sigma$ be a reduced and convergent $2$-polygraph. 

%%%%%%%%%%%%%%%%%%%%%%%%%%%%%%%%
\subsubsection{Triple branchings of \pdf{2}-polygraphs}

A \emph{triple branching of $\Sigma$} is a triple $(f,g,h)$ of $2$-cells of $\Sigma^*$ with the same source and such that $f\preceq g\preceq h$. The triple branching $(f,g,h)$ is \emph{local} when $f$, $g$ and $h$ are rewriting steps. A local triple branching $(f,g,h)$ is:
\begin{itemize}
\item \emph{aspherical} when either $(f,g)$ or $(g,h)$ is aspherical,
\item \emph{Peiffer} when either $(f,g)$ or $(g,h)$ is Peiffer,
\item \emph{overlapping}, otherwise.
\end{itemize}
Triple branchings are ordered by inclusion, similarly to branchings. A \emph{critical} triple branching is a minimal overlapping triple branching. Such a triple branching can have two different shapes, where~$\phi$,~$\psi$ and~$\chi$ are generating $2$-cells :
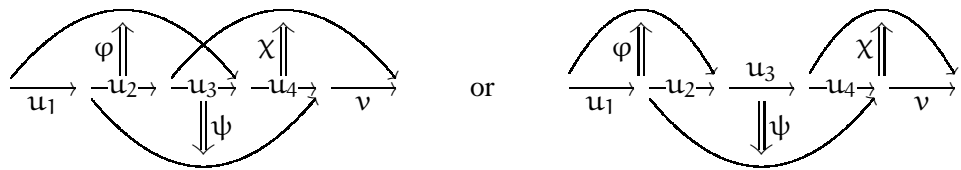
\[
\xymatrix{
\strut
	\ar [r] _-{u_1}
	\ar @/^6ex/ [rrr] _-{}="tgt1"
& \strut
	\ar [r] |-{u_2} ^-{}="src1"
	\ar @/_6ex/ [rrr] ^-{}="tgt2"
& \strut
	\ar[r] |-{u_3} _-{} ="src2"
	\ar @/^6ex/ [rrr] _-{}="tgt3"
& \strut
	\ar[r] |-{u_4} ^-{} ="src3"
& \strut
	\ar[r] _-{v}
& \strut
\ar@2 "src1"!<0pt,5pt>;"tgt1"!<0pt,-5pt> ^-{\phi}
\ar@2 "src2"!<0pt,-5pt>;"tgt2"!<0pt,5pt> ^-{\psi}
\ar@2 "src3"!<0pt,5pt>;"tgt3"!<0pt,-5pt> ^-{\chi}
}
\qquad\text{or}\qquad
\xymatrix{
\strut
	\ar [r] _-{u_1}
	\ar @/^6ex/ [rr] _-{}="tgt1"
& 
	\ar [r] |-{u_2}
	\ar @/_6ex/ [rrr] ^{}="tgt2"
& \strut
	\ar[r] ^-{u_3} _{} ="src2"
& \strut
	\ar[r] |-{u_4}
	\ar @/^6ex/ [rr] _-{}="tgt3"
& 
	\ar[r] _-{v}
& \strut
\ar@2 "1,2"!<0pt,5pt>;"tgt1"!<0pt,-5pt> ^-{\phi}
\ar@2 "src2"!<0pt,-5pt>;"tgt2"!<0pt,5pt> ^-{\psi}
\ar@2 "1,5"!<0pt,5pt>;"tgt3"!<0pt,-5pt> ^-{\chi}
}
\]
Those two shapes of critical triple branchings are sufficient for a reduced $2$-polygraph but, in a general situation, the other possible type of critical branchings (with an inclusion of one source into the other one) generates several other possibilities. Either way, if $\Sigma$ is finite, then it has a finite number of critical triple branchings.

For both possible shapes, the corresponding critical triple branching $b$ can be written 
\[
b 
	\:=\: \big( \: c\rep{v}, \: \rho_{u\rep{v}} \: \big) 
	\:=\: \big( \: f\rep{v}, \: \rho_u\rep{v}, \:\rho_{u\rep{v}} \: \big)
\]
where $c=(f,\rho_u)$ is a critical branching of~$\Sigma$ with source $u=u_1u_2u_3u_4$ and where $\rho_u=u_1\psi$. Indeed, we note that~$v$ must be a normal form for $\Sigma$ to be reduced. Moreover, in the first case, we have $f=\phi u_4$ and $\rho_{u\rep{v}}=u_1 u_2\chi$ and, in the second case, we have $f=\phi u_3 u_4$ and $\rho_{u\rep{v}}=u_1u_2u_3\chi$.

%%%%%%%%%%%%%%%%%%%%%%%%%%%%
\subsubsection{The basis of generating triple confluences}

The \emph{basis of generating triple confluences of $\Sigma$} is the cellular extension $\crit_3(\Sigma)$ of $\tck{\crit_2(\Sigma)}$ made of one $4$-cell 
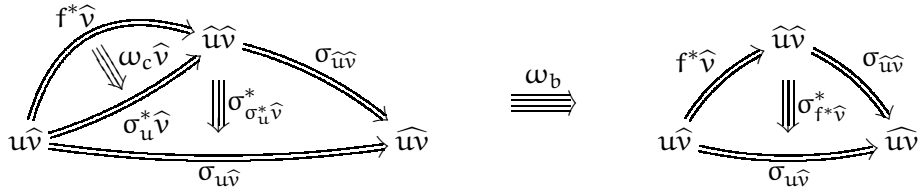
\[
\xymatrix{
&& {\rep{u}\rep{v}}
	\ar@2 @/^/ [drr] ^-{\sigma_{\rep{u}\rep{v}}}
\\
{u\rep{v}}
	\ar@2 @/^5ex/ [urr] ^-{f^*\rep{v}} _-{}="s"
	\ar@2 @/_1ex/ [urr] _-{\sigma^*_u\rep{v}} ^-{}="t"
	\ar@2 @/_/ [rrrr] _-{\sigma_{u\rep{v}}} ^-{}="tgt"
&&&& {\rep{uv}}
	\ar@3 "s"!<2.5pt,-5pt>;"t"!<-2.5pt,5pt> ^-{\omega_c\rep{v}}
	\ar@3 "1,3"!<0pt,-15pt>;"tgt"!<0pt,10pt> ^-{\sigma^*_{\sigma^*_u\rep{v}}}
}
\qquad
\raisebox{-1.9em}{\xymatrix{ \strut \ar@4 [r] ^*+\txt{$\omega_b$} & }}
\qquad
\xymatrix{
& {\rep{u}\rep{v}}
	\ar@2 @/^/ [dr] ^-{\sigma_{\rep{u}\rep{v}}}	
\\
{u\rep{v}}
	\ar@2 @/^/ [ur] ^-{f^*\rep{v}}
	\ar@2 @/_/ [rr] _-{\sigma_{u\rep{v}}} ^-{}="tgt"
&& {\rep{uv}}
	\ar@3 "1,2"!<0pt,-15pt>;"tgt"!<0pt,10pt> ^-{\sigma^*_{f^*\rep{v}}}
}
\]
for every critical triple branching $b=(f\rep{v},\rho_u\rep{v},\rho_{u\rep{v}})$ of $\Sigma$, where $c=(f,\rho_u)$ is a critical branching of $\Sigma$. Using the notations $(\,\cdot\,)^*$ and $\rep{\,\cdot\,}$ for $2$-cells and $3$-cells, the $4$-cell $\omega_b$ can also be written
\[
\xymatrix@!C{
{u\rep{v}} 
	\ar@2 @/^5ex/ [rr] ^-{(f^*\rep{v})^*} _{}="src1"
	\ar@2 @/_5ex/ [rr] _-{(\rep{fv})^*} ^{}="tgt1"
&& {\rep{uv}}
& \strut 
	\ar@4 [r] ^*+\txt{$\omega_b$}
&& {u\rep{v}}
	\ar@2 @/^5ex/ [rr] ^-{(f^*\rep{v})^*} _{}="src2"
	\ar@2 @/_5ex/ [rr] _-{(\rep{fv})^*} ^{}="tgt2"
&& {\rep{uv}}.
\ar@3 "src1"!<0pt,-10pt>;"tgt1"!<0pt,10pt> |-{(\omega_c\rep{v})^*}
\ar@3 "src2"!<0pt,-10pt>;"tgt2"!<0pt,10pt> |-{\rep{\omega_c v}^*}
}
\]
 
\begin{lemma}
The rightmost normalisation strategy of $\Sigma$ extends to a right normalisation strategy of~$\crit_3(\Sigma)$.
\end{lemma}

\begin{proof}
Let us define a $4$-cell 
\[
\sigma^*_{\omega_c\rep{w}} \::\: (\omega_c\rep{w})^* \:\qfl\: \rep{\omega_c w}^* 
\]
of $\tck{\crit_3(\Sigma)}$ for every $3$-cell $\omega_c$ of $\crit_2(\Sigma)$ and every $1$-cell $w$ in $\cl{\Sigma}$ such that $\omega_c\rep{w}$ exists. Let us denote by $v$ the source of the critical branching $c$ of $\Sigma$ and by $f$ the rewriting step of $\Sigma$ with source $v$ such that the critical branching $c$ is $(f,\rho_v)$. We proceed by case analysis on the type of the local triple branching $b=(f\rep{w},\rho_v\rep{w},\rho_{v\rep{w}})$.

\begin{itemize}

\item If $b$ is aspherical, then $\rho_{v\rep{w}}=\rho_v\rep{w}$. In that case, we define $
\sigma^*_{\omega_b\rep{w}} \:=\: 1_{(\omega_b\rep{w})^*}
$.

\item By hypothesis, the triple branching $b$ cannot be a Peiffer one.

\item Otherwise, we have $\rep{w}=\rep{w}_1\rep{w}_2$ and $b_1=(f\rep{w}_1,\rho_v\rep{w}_1,\rho_{v\rep{w}_1})$ is a critical triple branching of $\Sigma$. We define the $4$-cell $\sigma^*_{\omega_c\rep{w}}$ as the following composite in $\tck{\crit_3(\Sigma)}$:
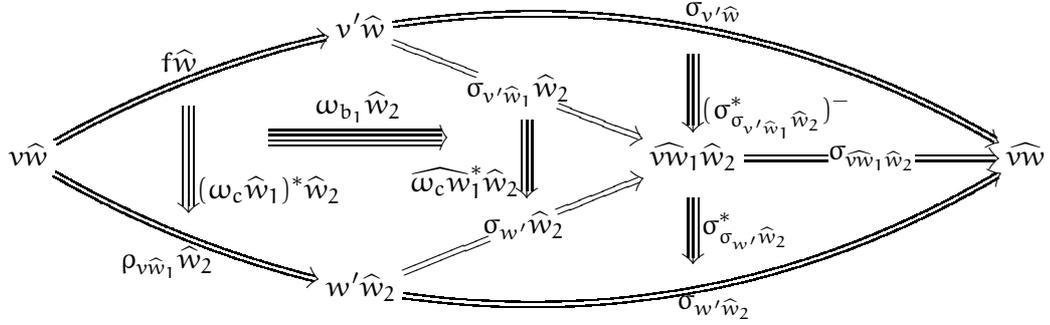
\[
\xymatrix @R=3em @!C{
&& {v'\rep{w}}
	\ar@2 [drr] |-{\sigma_{v'\rep{w}_1}\rep{w}_2} _{}="src4"
	\ar@2 @/^5ex/ [drrrr] ^{\sigma_{v'\rep{w}}} _(0.482){}="src1"
\\
{v\rep{w}} 
	\ar@2@/^/ [urr] ^{f\rep{w}} _(0.4985){}="src3"
	\ar@2@/_/ [drr] _{\rho_{v\rep{w}_1}\rep{w}_2} ^{}="tgt3"
&&&& {\rep{vw}_1\rep{w}_2}
	\ar@2 [rr] |-{\sigma_{\rep{vw}_1\rep{w}_2}} 
&& {\rep{vw}}
\\
&& {w'\rep{w}_2}
	\ar@2 [urr] |-{\sigma_{w'}\rep{w}_2} ^{}="tgt4"
	\ar@2 @/_5ex/ [urrrr] _{\sigma_{w'\rep{w}_2}} ^(0.482){}="tgt2"
\ar@3 "src1"!<0pt,-10pt>;"2,5"!<0pt,10pt> ^(0.8){(\sigma^*_{\sigma_{v'\rep{w}_1}\rep{w}_2})^-}
\ar@3 "2,5"!<0pt,-15pt>;"tgt2"!<0pt,10pt> ^-{\sigma^*_{\sigma_{w'}\rep{w}_2}}
\ar@3 "src3"!<0pt,-10pt>;"tgt3"!<0pt,10pt> ^(0.8){(\omega_c \rep{w}_1)^*\rep{w}_2} _(0.3){}="src5"
\ar@3 "src4"!<0pt,-10pt>;"tgt4"!<0pt,10pt> _(0.8){\rep{\omega_c w}_1^*\rep{w}_2} _(0.225){}="tgt5"
\ar@4 "src5"!<30pt,0pt>;"tgt5"!<-30pt,0pt> ^*+\txt{$\omega_{b_1}\rep{w}_2$}
}
\]
\end{itemize}
We apply Corollary~\ref{lemma : decomposition leftmost strategy} to extend the family of $4$-cells we have defined to a right normalisation strategy of $\crit_3(\Sigma)$.
\end{proof}

\begin{proposition}
The $(4,1)$-polygraph $\crit_3(\Sigma)$ is acyclic.
\end{proposition}

\begin{corollary}
A category with a finite convergent presentation is $\FDT_4$. 
\end{corollary}

\begin{example}
In the case of the reduced standard presentation $\N\C$ of a category $\C$, the basis of generating triple confluences (for the rightmost normalisation strategy) has one $4$-cell $\mu_{u,v,w,x}$ for every composable non-identity $1$-cells $u$, $v$, $w$ and $x$ of $\C$. The shape of the $4$-cell $\mu_{u,v,w,x}$ depends on whether $uv$, $vw$, $wx$, $uvw$ and $vwx$ are identities or not. In the case where neither of these $1$-cells is an identity, the corresponding $4$-cell is the following one:
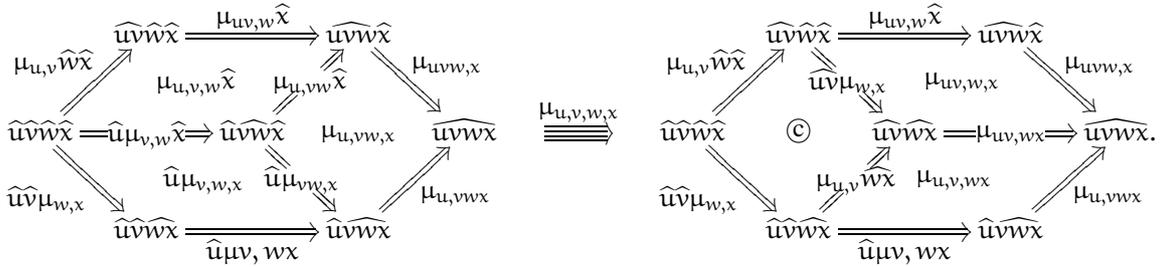
\[
\scalebox{0.9}{
\xymatrix@C=1em{
& {\rep{uv}\rep{w}\rep{x}}
	\ar@2 [rr] ^-{\mu_{uv,w}\rep{x}}
	\ar@{} [dr] |-{\mu_{u,v,w}\rep{x}}
&& {\rep{uvw}\rep{x}}
	\ar@2 [dr] ^-{\mu_{uvw,x}}
&&&&&& {\rep{uv}\rep{w}\rep{x}}
	\ar@2 [rr] ^-{\mu_{uv,w}\rep{x}}
	\ar@2 [dr] |-{\rep{uv}\mu_{w,x}}
&& {\rep{uvw}\rep{x}}
	\ar@2 [dr] ^-{\mu_{uvw,x}}
\\
{\rep{u}\rep{v}\rep{w}\rep{x}}
	\ar@2 [ur] ^-{\mu_{u,v}\rep{w}\rep{x}}
	\ar@2 [rr] |-{\rep{u}\mu_{v,w}\rep{x}}
	\ar@2 [dr] _-{\rep{u}\rep{v}\mu_{w,x}}
&& {\rep{u}\rep{vw}\rep{x}}
	\ar@2 [ur] |-{\mu_{u,vw}\rep{x}}
	\ar@2 [dr] |-{\rep{u}\mu_{vw,x}}
	\ar@{} [rr] |-{\mu_{u,vw,x}}
&& {\rep{uvwx}}
& \strut 
	\ar@4 [rr] ^*+\txt{$\mu_{u,v,w,x}$}
&& \strut
& {\rep{u}\rep{v}\rep{w}\rep{x}}
	\ar@2 [ur] ^-{\mu_{u,v}\rep{w}\rep{x}}
	\ar@2 [dr] _-{\rep{u}\rep{v}\mu_{w,x}}
	\ar@{} [rr] |-{\copyright}
&& {\rep{uv}\rep{wx}}
	\ar@2 [rr] |-{\mu_{uv,wx}}
	\ar@{} [ur] |-{\mu_{uv,w,x}}
	\ar@{} [dr] |-{\mu_{u,v,wx}}
&& {\rep{uvwx}}.
\\
& {\rep{u}\rep{v}\rep{wx}}
	\ar@2 [rr] _-{\rep{u}\mu{v,wx}}
	\ar@{} [ur] |-{\rep{u}\mu_{v,w,x}}
&& {\rep{u}\rep{vwx}}
	\ar@2 [ur] _-{\mu_{u,vwx}}
&&&&&& {\rep{u}\rep{v}\rep{wx}}
	\ar@2 [ur] |-{\mu_{u,v}\rep{wx}}
	\ar@2 [rr] _-{\rep{u}\mu{v,wx}}
&& {\rep{u}\rep{vwx}}
	\ar@2 [ur] _-{\mu_{u,vwx}}
}
}
\]
\end{example}

%%%%%%%%%%%%%%%%%%%%%%%%%%%%%%%%%%%%%%%%%%%
\subsection{The polygraphic resolution generated by a convergent \pdf{2}-polygraph} 
%%%%%%%%%%%%%%%%%%%%%%%%%%%%%%%%%%%%%%%%%%%

Let $\Sigma$ be a reduced and convergent $2$-polygraph and let us extend it into an acyclic  $(\infty,1)$-polygraph, \ie, a polygraphic resolution of the category $\cl{\Sigma}$. This $(\infty,1)$-polygraph is denoted by $\crit_{\infty}(\Sigma)$ and its generating $(n+1)$-cells, for $n\geq 2$, are (indexed by) the $n$-fold critical branchings of $\Sigma$. We proceed by induction on $n$, having already seen the base cases, for $n=2$ and $n=3$. The induction case follows the construction of the acyclic $(4,1)$-polygraph $\crit_3(\Sigma)$, so we go faster here.   

%%%%%%%%%%%%%%%%%%%%%%%%%%%%%%%%
\subsubsection{Higher branchings of \pdf{2}-polygraphs}

An \emph{$n$-fold branching of $\Sigma$} is a family $(f_1,\dots,f_n)$ of rewriting steps of $\Sigma$ with the same source and such that $f_1\preceq \cdots\preceq f_n$. We define local, aspherical, Peiffer, overlapping, minimal and critical $n$-fold branchings in a similar way to the cases $n=2$ and $n=3$. As before, we study the possible shapes of an $n$-fold critical branching~$b$ of~$\Sigma$ and we conclude that it must have shape 
\[
b \:=\: \big( \: c\rep{v}, \:\rho_{u\rep{v}} \:\big) 
\]
where $c$ is a critical $(n-1)$-fold branching of~$\Sigma$ with source $u$. Moreover, if $\Sigma$ is finite, then it has finitely many $n$-fold critical branchings.

%%%%%%%%%%%%%%%%%%%%%%%%%%%%
\subsubsection{The basis of generating \pdf{n}-fold confluences}

The \emph{basis of generating $n$-fold confluences of $\Sigma$} is the cellular extension $\crit_n(\Sigma)$ of $\tck{\crit_{n-1}(\Sigma)}$ made of one $(n+1)$-cell 
\[
\omega_b \::\: \big( \omega_c\rep{v} \big)^* \:\longrightarrow\: \rep{ \omega_c v} ^*
\]
for every critical $n$-fold branching $b=(c\rep{v},\rho_{u\rep{v}})$ of $\Sigma$. 
 
The extension of the right normalisation strategy to $\crit_n(\Sigma)$ is made in the same way as in the case $n=3$. It relies on a Corollary~\ref{lemma : decomposition leftmost strategy} and a case analysis, whose main point is to define an $(n+1)$-cell 
\[
\sigma^*_{\omega_{c} \rep{w}} \::\: 
(\omega_{c} \rep{w})^*
\:\longrightarrow\:
\rep{\omega_{c} w}^*
\]
in $\tck{\crit_n(\Sigma)}$ for every local $n$-fold branching
\[
b \:=\: \big( c\rep{w}, \rho_{v\rep{w}} \big)
\]
of $\Sigma$ such that $\rep{w}=\rep{w}_1\rep{w}_2$ and such that $b_1=(c\rep{w}_1,\rho_{v\rep{w}_1})$ is a critical $n$-fold branching of $\Sigma$. As in the case $n=3$, we define the $(n+1)$-cell $\sigma^*_{\omega_c\rep{w}}$ as the following composite, where $f$ is the first $2$-cell of the critical $n$-fold branching~$c$:
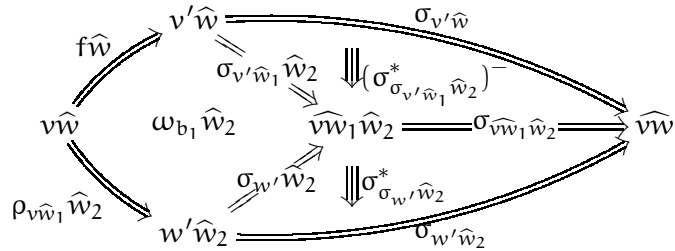
\[
\xymatrix{
& {v'\rep{w}}
	\ar@2 [dr] |-{\sigma_{v'\rep{w}_1}\rep{w}_2}
	\ar@{} [dd] |-{\omega_{b_1}\rep{w}_2}
	\ar@2 @/^3ex/ [drrrr] ^-{\sigma_{v'\rep{w}}} _(0.327){}="src1"
\\
{v\rep{w}} 
	\ar@2@/^/ [ur] ^-{f\rep{w}}
	\ar@2@/_/ [dr] _-{\rho_{v\rep{w}_1}\rep{w}_2}
&& {\rep{vw}_1\rep{w}_2}
	\ar@2 [rrr] |-{\sigma_{\rep{vw}_1\rep{w}_2}} 
&&& {\rep{vw}}
\\
& {w'\rep{w}_2}
	\ar@2 [ur] |-{\sigma_{w'}\rep{w}_2}
	\ar@2 @/_3ex/ [urrrr] _-{\sigma_{w'\rep{w}_2}} ^(0.325){}="tgt2"
\ar@3 "src1"!<0pt,-10pt>;"2,3"!<0pt,15pt> ^(0.9){(\sigma^*_{\sigma_{v'\rep{w}_1}\rep{w}_2})^-}
\ar@3 "2,3"!<0pt,-15pt>;"tgt2"!<0pt,10pt> ^-{\sigma^*_{\sigma_{w'}\rep{w}_2}}
}
\]
As a conclusion of this construction, we get that the $(n+1,1)$-polygraph $\crit_n(\Sigma)$ is acyclic. 

\begin{theorem}
\label{MainTheorem2.0}
Every convergent $2$-polygraph $\Sigma$ extends to a Tietze-equivalent, acyclic $(\infty,1)$-poly\-graph $\crit_{\infty}(\Sigma)$, whose generating $n$-cells, for every $n\geq 3$, are (indexed by) the critical $(n-1)$-fold branchings of $\Sigma$.
\end{theorem}

\noindent
As a consequence, we get:

\begin{corollary}
\label{MainTheorem2}
A category with a finite convergent presentation is $\FDT_{\infty}$. Moreover, if $\C$ is a category with a convergent presentation with no critical $n$-fold branching, for some $n\geq 2$, then $\poldim{\C} \leq n$.
\end{corollary}

\begin{example}
If $\N\C$ is the reduced standard presentation of a category $\C$, we have already built the dimensions $3$ and $4$ of the $(\infty,1)$-polygraph $\crit_{\infty}(\N\C)$, called the \emph{reduced standard polygraphic resolution of $\C$}. 

More generally, the $(\infty,1)$-polygraph $\crit_{\infty}(\N\C)$ has, for every natural number $n\geq 2$, one $n$-cell $\mu_{u_1,\dots,u_n}$ for every family $(u_1,\dots,u_n)$ of non-identity composable $1$-cells of $\C$. The shape of this cell depends on the partial compositions of the $1$-cells $u_1$, \dots, $u_n$. 

In the case where no such partial composition is an identity, we get an $n$-cell with the same shape as an $n$-simplex, representing all the possible ways to transform $\rep{u}_1\cdots\rep{u}_n$ into $\rep{u_1\cdots u_n}$, all the homotopies between these different ways, all homotopies between these homotopies, and so on. Indeed, we have seen that each $2$-cell $\mu_{u,v}:\rep{u}\rep{v}\dfl\rep{uv}$ is represented by a triangle, each $3$-cell $\mu_{u,v,w}$ is pictured as a tetrahedron and each $4$-cell $\mu_{u,v,w,x}$ as a $4$-simplex. More generally, the source and the target of this $n$-cell are $(n-1)$-composites of the following $(n-1)$-cells:
\[
d_i(u_1,\dots,u_n) \:=\: 
\begin{cases}
\rep{u}_1\mu_{u_2,\dots,u_n} 
	& \text{if } i=0 \\
\mu_{u_1,\dots,u_i u_{i+1},\dots,u_n}
	& \text{if } 1\leq i\leq n-1 \\
\mu_{u_1,\dots,u_{n-1}}\rep{u}_n
	& \text{if } i=n.
\end{cases}
\]
More precisely, the $n$-cell $\mu_{u_1,\dots,u_n}$ has the same shape as an $n$-oriental, the higher-categorical equivalent of an $n$-simplex, see~\cite{Street87}.
\end{example}

%%%%%%%%%%%%%%%%%%%%%%%%%%%%%%%%%%%%%%%%%%%%
%%%%%%%%%%%%%%%%%%%%%%%%%%%%%%%%%%%%%%%%%%%%
\section{Abelianisation of polygraphic resolutions}
\label{SectionAbelianisation}
%%%%%%%%%%%%%%%%%%%%%%%%%%%%%%%%%%%%%%%%%%%%
%%%%%%%%%%%%%%%%%%%%%%%%%%%%%%%%%%%%%%%%%%%%

Let us fix a category $\C$.

%%%%%%%%%%%%%%%%%%%%%%%%%%%%%%%%%%%%%%%%%%%
\subsection{Resolutions of finite type}
\label{FiniteHomologyType} 
%%%%%%%%%%%%%%%%%%%%%%%%%%%%%%%%%%%%%%%%%%%

%%%%%%%%%%%%%%%%%%%%%%%%%%%%%%%%%%%%%%%%%%%
\subsubsection{Modules over a category,~\cite{Mitchell72}}

A \emph{$\C$-module} is a functor from $\C$ to the category of Abelian groups $\Ab$. The $\C$-modules and natural transformations between them form an Abelian category with enough projectives, denoted by $\mod(\C)$. Equivalently, $\mod(\C)$ can be described as the category of additive functors from $\Zb\C$ to $\Ab$, where $\Zb\C$ is the free $\Zb$-category over $\C$: its $0$-cells are the ones of $\C$ and each hom-set $\Zb\C(x,y)$ is the free Abelian group generated by $\C(x,y)$. 

A $\C$-module $M$ is \emph{free} when it is a coproduct of representable functors $\Zb\C(x,-)$ and it is \emph{finitely generated} if there exists an epimorphism of $\C$-modules $F \pfl M$, with $F$ free.

Let $M$ be a $\op{\C}$-module and $N$ be a $\C$-module. The \emph{tensor product of $M$ and $N$ over $\C$} is the Abelian group $M\tens_{\C} N$ defined as the coend 
\[
M\tens_{\C}N \:=\: \int^{x\in\C_0} M(x)\:\tens_{\Zb}\:N(x).
\]
In a more explicit way, we have:
\[
M\tens_{\C}N 
	\:=\: \left( \bigoplus_{x\in \C_0} M(x)\tens_{\Zb} N(x) \right) \:\Big/\: Q 
\]
where $Q$ is the subgroup of $\underset{x\in\C_0}{\bigoplus}M(x)\tens_{\Zb} N(x)$ generated by the elements 
\[
M(u)(a)\tens b - a \tens N(u)(b),
\]
for any possible $1$-cell $u:x\fl y$ of $\C$ and any elements $a$ of $M(y)$ and $b$ of $N(x)$.

%%%%%%%%%%%%%%%%%%%%%%%%%%%%%%%%%%%%%%%%%%%
\subsubsection{Modules of finite homological type}

A $\C$-module $M$ is of \emph{homological type $\FP_n$}, for a natural number~$n$, when there exists a partial resolution of $M$ of length $n$ by projective, finitely generated $\C$-modules:
\[
\xymatrix{
P_n \ar[r] 
& P_{n-1} \ar[r] 
& \:\cdots\: \ar[r] 
& P_0 \ar[r] 
& M \ar[r]
& 0.
}
\]
A $\C$-module $M$  is of \emph{homological type $\FP_{\infty}$} when there exists a resolution of $M$ by projective, finitely generated $\C$-modules:
\[
\xymatrix{
& \:\cdots\: \ar[r]
& P_n \ar[r] 
& P_{n-1} \ar[r] 
& \:\cdots\: \ar[r] 
& P_0 \ar[r] 
& M \ar[r]
& 0.
}
\]
As a generalisation of Schanuel's lemma, we have, given two exact sequences 
\[
\xymatrix{
0\ar[r]
& P_{n+1} \ar[r] 
& P_n \ar[r] 
& \:\cdots\: \ar[r] 
& P_0 \ar[r] 
& M \ar[r]
& 0
}
\]
and 
\[
\xymatrix{
0\ar[r]
& P'_{n+1} \ar[r] 
& P'_n \ar[r] 
& \:\cdots\: \ar[r] 
& P'_0 \ar[r] 
& M \ar[r]
& 0
}
\]
of projective $\C$-modules, with $P_i$ and $P'_i$ finitely generated for every $0\leq i\leq n$, then $P_n$ is finitely generated if and only if $P'_n$ is finitely generated. This yields the following characterisation of the property~$\FP_n$:

\begin{lemma}
\label{lemme_pl_n}
Let $M$ be a $\C$-module and let $n$ be a natural number. The following assertions are equivalent:  

\begin{enumerate}[\bf i)]
\item The $\C$-module $M$ is of homological type $\FP_n$.
\item There exists a free, finitely generated partial resolution of $M$ of length $n$:
\[
\xymatrix{
F_n \ar[r] 
& F_{n-1} \ar[r] 
& \:\cdots\: \ar[r] 
& F_0 \ar[r] 
& M \ar[r]
& 0.
}
\]
\item The $\C$-module $M$ is finitely generated and, for every $0\leq k<n$ and every projective, finitely generated partial resolution of $M$ of length $k$
\[
\xymatrix{
P_k \ar[r] ^-{d_k} 
& P_{k-1} \ar[r] 
& \,\cdots\, \ar[r] 
& P_0 \ar[r] 
& M \ar[r]
& 0,
}
\]
the $\C$-module $\ker d_k$ is finitely generated.  
\end{enumerate}
\end{lemma}

\begin{lemma}\label{lemmaLanFPn}
Let $\D$ be a category, let $F: \C \fl \D$ be a functor and let $\Lan_F : \Mod(\C) \fl \Mod(\D)$ be the additive left Kan extension along $F$. If $M$ is a $\C$-module of homological type $\FP_n$ then $\Lan_F(M)$ is a $\D$-module of homological type $\FP_n$.
\end{lemma}

\begin{proof}
Let us assume that $M$ is a $\C$-module of type $\FP_n$. Then there exists a projective, finitely generated partial resolution $P_\ast \fl M$ of length $n$. If $x$ is a $0$-cell in $\D$, then we have: 
\[
\Lan_F(M)(x) \:=\: \Zb\D(F,x) \tens_{\C} M.
\]
Since each $\C$-module $P_i$ is finitely generated and projective, then so is the $\D$-module $\Lan_F(P_i)$. Moreover, the functor $\Lan_F$ is right-exact: it follows that $\Lan_F(P_\ast) \fl \Lan_F(M)$ is a projective, finitely generated partial resolution of length $n$. This proves that $\Lan_F(M)$ is of type $\FP_n$. 
\end{proof}

%%%%%%%%%%%%%%%%%%%%%%%%%%%%%%%%%%%%%%%%%%%
\subsection{Categories of finite homological type}

%%%%%%%%%%%%%%%%%%%%%%%%%%%%%%%%%%%%%%%%%%%
\subsubsection{Natural systems of Abelian groups}
\label{NaturalSystems}

The \emph{category of factorisations of $\C$} is the category, denoted by $F\C$, whose $0$-cells are the $1$-cells of $\C$ and whose $1$-cells from $w$ to $w'$ are pairs $(u,v)$ of $1$-cells of $\C$ such that the following diagram commutes in $\C$:  
\[
\xymatrix @R=2em @C=2em{   
& \cdot
	\ar [r] ^-{w} _-{}="src"
& \cdot
	\ar [dr] ^-{v}
\\
\cdot
	\ar [ur] ^-{u}
	\ar [rrr] _-{w'} ^-{}="tgt"
&&& \cdot 
		\ar@{} "src";"tgt" |-{\sm\copyright}
}
\] 
In such a situation, the triple $(u,w,v)$ is called a \emph{factorisation of $w'$}. Composition in $F\C$ is defined by pasting, \ie, if $(u,v):w\fl w'$ and $(u',v'):w'\fl w''$ are $1$-cells of $F\C$, then the composite $(u,v)(u',v')$ is $(u'u,vv')$:
\[
\xymatrix @R=2em @C=2em{   
&& \cdot
	\ar [r] ^-{w} _-{}="src1"
& \cdot
	\ar [dr] ^-{v}
\\
& \cdot
	\ar [ur] ^-{u}
	\ar [rrr] |-{w'} ^-{}="tgt1" _-{}="src2"
&&& \cdot 
	\ar [dr] ^-{v'}
\\
\cdot
	\ar [ur] ^-{u'}
	\ar [rrrrr] _-{w''} ^-{}="tgt2"
&&&&& \cdot
		\ar@{} "src1";"tgt1" |-{\sm\copyright}
		\ar@{} "src2";"tgt2" |-{\sm\copyright}
}
\]
The identity of $w$ is $(1_{s(w)},1_{t(w)})$:
\[
\xymatrix @R=2em @C=2em{   
& \cdot
	\ar [r] ^-{w} _-{}="src"
& \cdot
	\ar [dr] ^-{1_{t(w)}}
\\
\cdot
	\ar [ur] ^-{1_{s(w)}}
	\ar [rrr] _-{w'} ^-{}="tgt"
&&& \cdot 
		\ar@{} "src";"tgt" |-{\sm\copyright}
}
\] 
A \emph{natural system (of Abelian groups) on $\C$} is an $F\C$-module $D$, \ie, a functor $D:F\C\fl\Ab$. As in~\cite{BauesWirsching85}, we denote by $D_w$ the Abelian group which is the image of $w$ by $D$. If there is no confusion, we denote by $uav$ the image of $a\in D_w$ through the morphism of groups $D(u,v) : D_w \fl D_{w'}$. The category of natural systems on $\C$ is denoted by $\Nat(\C)$.

%%%%%%%%%%%%%%%%%%%%%%%%%%%%%%%%%%%%%%%%%%%
\subsubsection{Free natural systems}

Given a family $X$ of $1$-cells of $\C$, we denote by $\fnat{\C}{X}$ the \emph{free natural
  system on $\C$ generated by $X$}, which is defined by 
\[
\fnat{\C}{X}
	\:=\: \bigoplus_{x\in X} F\C(x,-).
\]
In particular, if $\Sigma$ is an $(n,1)$-polygraph such that $\cl{\Sigma}\simeq\C$, we consider:
\begin{itemize}
\item The free natural system $\fnat{\C}{\Sigma_0}$ generated by the $1$-cells $1_x$, for $x\in\Sigma_0$: if $w$ is a $1$-cell in $\C$, then $\fnat{\C}{\Sigma_0}_w$ is the free Abelian group generated by the pairs $(u,v)$ of $1$-cells of $\C$ such that $uv=w$. 
\item For every $1\leq k\leq p$, the free natural system $\fnat{\C}{\Sigma_k}$ generated by a copy of the $1$-cell $\cl{\phi}$ for each $k$-cell $\phi$ of~$\Sigma_k$: if $w$ is a $1$-cell in $\C$, then $\fnat{\C}{\Sigma_k}_w$ is the free Abelian group generated by the triples $(u,\phi,v)$, thereafter denoted by $u[\phi]v$, made of a $k$-cell $\phi$ of $\Sigma_k$ and $1$-cells $u$, $v$ of $\C$ such that $u\cl{\phi}v=w$. 
\end{itemize}

%%%%%%%%%%%%%%%%%%%%%%%%%%%%%%%%%%%%%%%%%%%
\subsubsection{Categories of finite homological type}

The property for a category $\C$ to be of homological type $\FP_n$ is defined according to a category of modules over one of the categories in the following diagram  
\[
\xymatrix@R=1em{ 
&& {\op{\C}} 
	\ar@{>->} @/^/ ^-{q_1} [dr]
\\ 
F\C
	\ar@{->>} [r] ^-{\pi} 
& {\op{\C} \times \C} 
	\ar@{->>} @/^/ [ur] ^-{p_1}
	\ar@{->>} @/_/ [dr] _-{p_2}
&& {\tck{\C}}
\\
&& {\C} 
	\ar@{>->} @/_/ [ur] _-{q_2}
}
\]
where $\tck{\C}$ is the groupoid generated by $\C$, $\pi$ is the projection $u\mapsto(s(u),t(u))$, $p_1$ and $p_2$ are the projections of the cartesian product, $q_1$ and $q_2$ are the injections $\op{u}\mapsto u^-$ and $u\mapsto u$. 

Let us denote by~$\Zb$ the constant natural system on $\C$ given, for any $1$-cell $u$ of $\C$, by
\[
\Zb_u \:=\: \Zb 
\qquad\text{and}\qquad 
\Zb(u,1) \:=\: \Zb(1,u) \:=\: \id_{\Zb}.
\]
Let $\Zb\C$ denote the $\op{\C}\times\C$-module $\Zb\C$ whose component at $(x,y)$ is the free $\Zb$-module $\Zb\C(x,y)$ generated by $\C(x,y)$. 

A category $\C$ is of \emph{homological type}
\begin{enumerate}[\bf i)]
\item \emph{$\FP_n$} when the constant natural system $\Zb$ is of type $\FP_n$,
\item \emph{bi-$\FP_n$} when the $\op{\C}\times\C$-module $\Zb\C$ is of type $\FP_n$,  
\item \emph{left-$\FP_n$} when the constant $\C$-module $\Zb$ is of type $\FP_n$,  
\item \emph{right-$\FP_n$} when the constant $\op{\C}$-module $\Zb$ is of type $\FP_n$,  
\item \emph{top-$\FP_n$} when the constant $\tck{\C}$-module $\Zb$ is of type $\FP_n$. 
\end{enumerate}

\begin{proposition}
\label{Proposition:implicationsPL}
For categories, we have the following implications:
\[
\xymatrix@R=1em{ 
&& {\text{right-$\FP_n$}} 
	\ar@2 @/^/ [dr]
\\ 
{\FP_n}
	\ar@2 [r] 
& {\text{bi-$\FP_n$}} 
	\ar@2 @/^/ [ur] 
	\ar@2 @/_/ [dr] 
&& {\text{top-$\FP_n$.}}
\\
&& {\text{left-$\FP_n$}} 
	\ar@2 @/_/ [ur]
}
\]
\end{proposition}

\begin{proof}
We have the following left Kan extensions:
\[
{\rm Lan}_{\pi}(\Zb) \:=\: \Zb\C 
\qquad\qquad
{\rm Lan}_{p_i}(\Zb) \:=\: \Zb
\qquad\qquad
{\rm Lan}_{q_i}(\Zb) \:=\: \Zb. 
\]
Hence the implications are consequences of Lemma~\ref{lemmaLanFPn}. 
\end{proof}

%%%%%%%%%%%%%%%%%%%%%%ù
\subsubsection{Converse implications}

The converse of the implications bi-$\FP_n$ $\dfl$ left/right-$\FP_n$ and left/right-$\FP_n$ $\dfl$ top-$\FP_n$ of Proposition~\ref{Proposition:implicationsPL} does not hold in general. Indeed, Cohen constructed a right-$\FP_{\infty}$ monoid which is not left-$\FP_{1}$: thus, the properties top-$\FP_n$, left-$\FP_n$ and right-$\FP_n$ are not equivalent in general,~\cite{Cohen92}. Moreover, monoids with a finite convergent presentation are of types left-$\FP_{\infty}$ and right-$\FP_{\infty}$,~\cite{Squier87,Anick86,Kobayashi91}, but there exists a finitely presented monoid, of types left-$FP_{\infty}$ and right-$FP_{\infty}$, which does not satisfy the homological finiteness condition~$\mathrm{FHT}$, introduced by Pride and Wang,~\cite{Kobayashi01}; since the properties $\mathrm{FHT}$ and bi-$\FP_3$ are equivalent,~\cite{Kobayashi03}, it follows that the properties left-$\FP_n$ and right-$\FP_n$ do not imply the property bi-$\FP_n$ in general. We conjecture that the converse of the first implication is not true either, but this is still an open problem. However, in the special case of groupoids, all the implications are equivalences.

\begin{proposition}
For groupoids, the properties $\FP_n$, bi-$\FP_n$, left-$\FP_n$, right-$\FP_n$ and top-$\FP_n$ are equivalent.
\end{proposition}

\begin{proof}
If $\G$ is a groupoid, then the groupoid $\tck{\G}$ generated by $\G$ is $\G$ itself. Hence, the properties left-$\FP_n$ and top-$\FP_n$ are equivalent. As a consequence, it is sufficient to prove that left-$\FP_n$ implies $\FP_n$ to conclude the proof. 

Given a $\G$-module $M$, we denote by $\tilde{M}$ the natural system of $\G$ defined, on a $1$-cell $w:x\fl y$ of~$\G$, by $\widetilde{M}_w = M(y)$ and, on a factorisation $w'=uwv$ with $w:x\fl y$ and $w':x'\fl y'$ in $\G$, by $\widetilde{M}(u,v)=M(v):M(y)\fl M(y')$. Let us assume that $M$ is a  projective $\G$-module, \ie, there exists a family $X$ of $1$-cells of $\G$ and a surjective morphism of groups
\[
\xymatrix @!C @C=2.5em {
{\bigoplus_{x\in X} \Zb \G(x,y)}
	\ar@{->>} []!<5pt,0pt>;[r]!<-5pt,0pt> ^-{\pi_y}
& M(y)
}
\]
that is natural in $y$. Let us denote by $L$ the natural system on $\G$ defined, on a $1$-cell $u$ with target $y$, by
\[
L_u \:=\: \bigoplus_{x\in X} \Zb F\G(1_x, 1_y).
\]
For fixed $0$-cells $x$ and $y$ of $\G$, the set $F\G(1_x,1_y)$ is the one of pairs $(u,v)$ of $1$-cells of $\G$ such that $u1_x v=1_y$. Because $\G$ is a groupoid, this is exactly the set of pairs $(v^-,v)$ where $v:x\fl y$ is a $1$-cell of~$\G$. Hence, we have a bijection between $\G(x,y)$ and $F\G(1_x,1_y)$ that is natural in $y$. This gives, for any $1$-cell~$u$ with target $y$ in $\G$, a surjective morphism of groups
\[
\xymatrix @!C @C=4em {
L_u \:=\: \bigoplus_{x\in X} \Zb \G(x,y) 
	\ar@{->>} []!<5pt,0pt>;[r]!<-5pt,0pt> ^-{\pi_{t(u)}}
& M(y) \:=\: \tilde{M}(u). 
}
\]
Moreover, one checks that this is natural in $u$, yielding an epimorphism $\pi_{t(-)}:L\pfl \tilde{M}$ of natural systems on $\G$, proving that $M$ is projective. Moreover, by construction, if $M$ is finitely generated, then we can choose the family $X$ to be finite, so that $\tilde{M}$ is finitely generated too. 

As a conclusion, from a partial resolution $P_* \pfl \Zb$ of length $n$ by projective and finitely generated $\G$-modules, we can build a partial resolution $\tilde{P}_* \pfl \Zb$ of length $n$ by projective and finitely generated natural systems on $\G$. Thus, the property left-$\FP_n$ implies the property $\FP_n$ for groupoids.
\end{proof}

%%%%%%%%%%%%%%%%%%%%%%%%%%%%%%%%%%%%%%%%%%%
\subsubsection{Finite homological type and homology}

The cohomology of categories with values in natural systems was defined in~\cite{Wells79} and~\cite{BauesWirsching85}. One can also define the homology of a category $\C$ with values in a \emph{contravariant natural system} $D$ on $\C$, that is an $\op{(F\C)}$-module, and relate its finiteness properties to the fact that $\C$ is $\FP_n$. Note that this is independent of the rest of this article.

We consider the nerve $N_*(\C)$ of $\C$, with boundary maps $d_i:N_n(\C)\fl N_{n-1}(\C)$, for $0\leq i\leq n$. For $s=(u_1,\dots,u_n)$ in $N_n(\C)$, we denote by $\cl{s}$ the composite $1$-cell $u_1\cdots u_n$ of $\C$. For every natural number $n$, the $n$-th chain group $C_n(\C,D)$ is defined as the Abelian group   
\[
C_n(\C,D) \:=\: \bigoplus_{s\in N_n(\C)} D_{\cl{s}}.
\]
We denote by $\iota_s$ the embedding of $D_{\cl{s}}$ into~$C_n(\C,D)$. The boundary map $d : C_n(\C,D) \fl C_{n-1}(\C,D)$ is defined, on the component $D_{\cl{s}}$ of $C_n(\C,D)$, by:
\[
d\iota_s 
	\:=\: \iota_{d_0(s)} u_{1*} 
		\:+\: \sum_{i=1}^{n-1} (-1)^i \iota_{d_i(s)} 
		\:+\: (-1)^n \iota_{d_n(s)} u_n^* \;, 
\] 
with $s=(u_1,\dots,u_n)$ and where $u_{1*}$ and $u_n^*$ respectively denote $D(u_1,1)$ and $D(1,u_n)$. The homology of $\C$ with coefficients in $D$ is defined as the homology of the complex $(C_*(\C,D),d_*)$:
\[
\ho_*(\C,D) \:=\: \ho_*(C_*(\C,D),d_*).
\]
We denote by $\Tor_*^{F\C}(D,-)$ the left derived functor from the functor $D\tens_{F\C}-$. One proves that there is an isomorphism which is natural in $D$:
\[
\ho_*(\C,D) \:\simeq\: \Tor_*^{F\C}(D,\Zb).
\]
As a consequence, using Lemma~\ref{lemme_pl_n}, we get:

\begin{proposition}
If a category $\C$ is of homological type $\FP_n$, for a natural number $n$, then the Abelian group $\ho_k(\C,\Zb)$ is finitely generated for every $0\leq k\leq n$. 
\end{proposition}

%%%%%%%%%%%%%%%%%%%%%%%%%%%%%%%%%%%%%%%%%%%
\subsection{The Reidemeister-Fox-Squier complex}

Let $\Sigma$ be an $(n,1)$-polygraph. The mapping of every $1$-cell $x$ of $\Sigma$ to the element $[x]$ of $\fnat{\cl{\Sigma}}{\Sigma_1}_{\cl{x}}$ is extended to associate to every $1$-cell $u$ of $\Sigma_1^*$ the element $[u]$ of $\fnat{\cl{\Sigma}}{\Sigma_1}_{\cl{u}}$, defined by induction on the size of $u$ as follows:
\[
[1_x] \:=\: 0 
\qquad\text{and}\qquad
[uv] \:=\: [u]\cl{v} + \cl{u}[v].
\]
This is well-defined since the given relations are compatible with the associativity and identity relations of~$\Sigma_1^*$. Note that the mapping $[\cdot]$ is a special case of the notion of derivation of the category $\Sigma_1^*$ into the natural system $\fnat{\cl{\Sigma}}{\Sigma_1}$ on $\cl{\Sigma}$, seen as a natural system on $\Sigma_1^*$ by composition with the canonical projection $\Sigma_1^*\pfl\cl{\Sigma}$, see~\cite{BauesWirsching85}.
 
Then, for $1<k\leq n$, the mapping of every $k$-cell $\alpha$ of $\Sigma$ to the element $[\alpha]$ of $\fnat{\cl{\Sigma}}{\Sigma_k}_{\cl{\alpha}}$ is extended to associate to every $k$-cell $f$ of $\tck{\Sigma}_k$ the element $[f]$ of $\fnat{\cl{\Sigma}}{\Sigma_k}_{\cl{f}}$, defined by induction on the size of $f$ as follows:
\[
[1_u] \:=\: 0
\qquad\qquad [f^-]\:=\:-[f]
\qquad\qquad [f\star_i g] \:= 
	\begin{cases}
	[f]\cl{g} + \cl{g}[f] &\text{if } i=0 \\
	[f] + [g] &\text{otherwise.}
	\end{cases}
\]
Here, to check that this is well-defined, we also have to show that this definition is compatible with exchange relations, for every $0\leq i<j\leq k$: 
\[
[(f\star_i g) \star_j (h \star_i  k)] \:=\: 
[(f\star_j h) \star_i (g \star_j k)] \:=
\begin{cases}
[f]\cl{g} + \cl{f}[g] + [h]\cl{k} + \cl{h}[k] 
	&\text{if } i=0 
\\
[f] + [g] + [h] + [k] 
	&\text{otherwise.}
\end{cases}
\]

%%%%%%%%%%%%%%%%%%%%%%%%%%%
\subsubsection{The Reidemeister-Fox-Squier complex}
\label{ReidemesterFoxSquierComplex}

Let $\Sigma$ be an $(n,1)$-polygraph. For $1\leq k\leq n$, the \emph{$k$-th Reidemeister-Fox-Squier boundary map of $\Sigma$} is the morphism of natural systems 
\[
\delta_k \::\: \fnat{\cl{\Sigma}}{\Sigma_k} \:\longrightarrow\: \fnat{\cl{\Sigma}}{\Sigma_{k-1}}
\]
defined, on the generator $[\alpha]$ corresponding to a $k$-cell $\alpha$ of $\Sigma$, by: 
\[
\delta_k[\alpha] \:=
\begin{cases}
(\cl{\alpha},1) - (1,\cl{\alpha}) 
	&\text{if } k=1
\\
[s(\alpha)] - [t(\alpha)]
	&\text{otherwise.}
\end{cases}
\]
The \emph{augmentation map of $\Sigma$} is the morphism of natural systems $\epsilon : \fnat{\cl{\Sigma}}{\Sigma_0} \fl \Zb$ defined, for every pair $(u,v)$ of composable $1$-cells of $\cl{\Sigma}$, by:
\[ 
\epsilon(u,v) \:=\: 1.
\]
By induction on the size of cells of $\tck{\Sigma}$, one proves that, for every $k$-cell $f$ in $\tck{\Sigma}$, with $k\geq 1$, the following holds:
\[
\delta_k[f] \:=
\begin{cases}
(\cl{f},1) - (1,\cl{f}) 
	&\text{if } k=1
\\
[s(f)] - [t(f)]
	&\text{otherwise.}
\end{cases}
\]
As a consequence, we have $\epsilon\delta_1=0$ and $\delta_k\delta_{k+1}=0$, for every $1\leq k<n$. Thus, we get the following chain complex of natural systems on $\cl{\Sigma}$, called the \emph{Reidemeister-Fox-Squier complex of $\Sigma$} and denoted by~$\fnat{\cl{\Sigma}}{\Sigma}$:
\[
\xymatrix{
\fnat{\cl{\Sigma}}{\Sigma_n}
	\ar [r] ^-{\delta_n}
& \fnat{\cl{\Sigma}}{\Sigma_{n-1}}
	\ar [r] ^-{\delta_{n-1}}
& \,\cdots\,
	\ar [r] ^-{\delta_2}
& \fnat{\cl{\Sigma}}{\Sigma_1}
	\ar [r] ^-{\delta_1}
& \fnat{\cl{\Sigma}}{\Sigma_0}
	\ar [r] ^-{\epsilon}
& \Zb
	\ar [r]
& 0.
}
\]
The terminology is due to the fact that this complex is inspired by constructions by Reidemeister, Fox and Squier, see~\cite{Reidemeister49, Fox53, Squier87}. In particular, the \emph{Fox Jacobian} used by Squier is the boundary map $\delta_2$, sending every $2$-cell $f:u\dfl v$, \ie, every relation $u=v$ of the presentation $\Sigma_2$, to the element $[u]-[v]$ of $\fnat{\cl{\Sigma}}{\Sigma_1}_{\cl{u}}$.

%%%%%%%%%%%%%%%%%%%%%%%%%%%%%%%%%%%%%%%%%%%
\subsection{Abelianisation of polygraphic resolutions}
%%%%%%%%%%%%%%%%%%%%%%%%%%%%%%%%%%%%%%%%%%%

Let us fix a partial polygraphic resolution $\Sigma$ of length $n\geq 1$ of the category $\C$.

%%%%%%%%%%%%%%%%%%%%%%%%%%%%%%%%%%%%%%%%%%%
\subsubsection{Contracting homotopies}

Since $\Sigma$ is acyclic, it admits a left normalisation strategy $\sigma$. We denote by $\sigma_k$, for $-1\leq k\leq p$, the following families of morphisms of groups, indexed by a $1$-cell~$w$ of~$\C$: 
\[
\defmap{ (\sigma_{-1})_w }
	{ \Zb }{ \fnat{\C}{\Sigma_0}_w }
	{ 1 }{ (1,w) }
\qquad\qquad
\defmap{ (\sigma_0)_w }
	{ \fnat{\C}{\Sigma_0}_w }{ \fnat{\C}{\Sigma_1}_w }
	{ (u,v) }{ [\rep{u}]v }
\]
\[
\defmap{ (\sigma_k)_w }
	{ \fnat{\c}{\Sigma_k}_w }{ \fnat{\c}{\Sigma_{k+1}}_w }
	{ u[x]v }{ [\sigma_{\rep{u}x}] v }
\]

\begin{lemma}
For every $k\in\ens{1,\dots,n-1}$, every $k$-cell $f$ of $\tck{\Sigma}$ and every $1$-cells $u$ and $v$ of $\C$ such that $u\cl{f}v$ exists, we have:
\[
\sigma_k (u[f]v) \:=\: [\sigma_{\rep{u} f}]v.
\]
\end{lemma}

\begin{proof}
We proceed by induction on the size of $f$. If $f=1_w$, for some $(k-1)$-cell $w$ of $\tck{\Sigma}$, then we have, on the one hand,
\[
\sigma_k (u[1_w]v) \:=\: \sigma_k(0) \:=\: 0
\]
and, on the other hand,
\[
[\sigma_{1_{\rep{u}w}}] v \:=\: [1_{1_{\rep{u}w}}] v \:=\: 0.
\]
If $f$ has size $1$, then the result holds by definition of $\sigma_k$. Let us assume that $f=gh$, where $g$ and $h$ are non-identity $k$-cells of $\tck{\Sigma}$. Then we use the induction hypothesis on $g$ and $h$ to get, on the one hand:
\[
\sigma_k (u[gh]v) 
	\:=\: \sigma_k (u[g]\cl{h} v) + \sigma_k(u\cl{g}[h]v)
	\:=\: [\sigma_{\rep{u}g}] \cl{h}v + [\sigma_{\rep{u\cl{g}} h}] v.
\]
On the other hand, since $\sigma$ is a left normalisation strategy, we have:
\[
[\sigma_{\rep{u}gh}] v
	\:=\: \big[ \sigma_{\rep{u}g} s_1(h) \star_1 \sigma_{\rep{u\cl{g}} h} \big] v
	\:=\: [\sigma_{\rep{u}g}] \cl{h}v + [\sigma_{\rep{u\cl{g}} h}] v.
\]
Finally, let us assume that $f=g\star_i h$, where $g$ and $h$ are non-identity $k$-cells of $\tck{\Sigma}$ and $i\geq 1$. Then we get:
\[
\sigma_k (u[g\star_i h]v) 
	\:=\: \sigma_k (u[g] v) + \sigma_k(u[h]v)
	\:=\: [\sigma_{\rep{u} g}] v + [\sigma_{\rep{u} h}] v.
\]
And we also have:
\[
[\sigma_{\rep{u}(g \star_i h)}] v
	\:=\: [\sigma_{\rep{u}g \star_i \rep{u} h}] v
	\:=\: [\sigma_{\rep{u}g} \star_i \sigma_{\rep{u} h}] v
	\:=\: [\sigma_{\rep{u} g}] v + [\sigma_{\rep{u}h}] v.
\qedhere
\]
\end{proof}

\begin{theorem}
\label{Theorem:AbelianResolution}
If $\Sigma$ is a (partial) polygraphic resolution (of length $n$) of a category $\C$, then the Reidemeister-Fox-Squier complex $\fnat{\C}{\Sigma}$ is a free (partial) resolution (of length $n$) of the constant natural system $\Zb$ on $\C$.
\end{theorem}

\begin{proof}
Let us prove that $\sigma_*$ is a contracting homotopy. Each $(\sigma_{-1})_w$ is a section of $\epsilon_w$, hence $\epsilon$ is an epimorphism of natural systems. By linearity of the boundary and section maps, it is sufficient to check the relation 
\[
\delta_1\sigma_0 (u,v) \:=\: (u,v) - (1,uv) \:=\: (u,v) -
\sigma_{-1}\epsilon(u,v), 
\]
on a generator $(u,v)$ of $\fnat{\C}{\Sigma_0}_w$ to get the exactness of $\fnat{\C}{\Sigma}$ at $\fnat{\C}{\Sigma_0}$. Then, on a generator $u[x]v$ of $\fnat{\C}{\Sigma_1}_w$, we compute
\[
\delta_2\sigma_1 (u[x]v) 
	\:=\: \delta_2([\sigma_{\rep{u}x}]) v 
	\:=\: [\rep{u}x] v - [\rep{ux}] v 
	\:=\: [\rep{u}]\cl{x} v + u [x] v - [\rep{ux}] v.
\]
and
\[
\sigma_0\delta_1(u[x]v) 
	\:=\: \sigma_0(u\cl{x},v) - \sigma_0(u,\cl{x}v) 
	\:=\: [\rep{ux}] v - [\rep{u}] \cl{x} v.
\]
Hence, using the linearity of the boundary and section maps, we get $\delta_2 \sigma_1 + \sigma_0\delta_1=1_{\fnat{\C}{\Sigma_1}}$, proving exactness at $\fnat{\C}{\Sigma_1}$. Finally, for $k\in\ens{2,\dots,p-1}$ and a generator $u[\alpha]v$ of $\fnat{\C}{\Sigma_k}_w$, we have:
\begin{align*}
\delta_{k+1}\sigma_k (u[\alpha]v) 
	\:&=\: \delta_{k+1} [\sigma_{\rep{u}\alpha}] v \\
	\:&=\: [\rep{u}\alpha] v - [\sigma_{\rep{u}s(\alpha)} \star_{k-1} \sigma_{\rep{u}t(\alpha)} ^- ] v \\
	\:&=\: u[\alpha]v - [\sigma_{\rep{u}s(\alpha)}]v + [\sigma_{\rep{u}t(\alpha)}]v \\
	\:&=\: u[\alpha]v - \sigma_{k-1}(u[s\alpha]v) + \sigma_{k-1}(u [t\alpha]v) \\
	\:&=\: u[\alpha]v - \sigma_{k-1}\delta_k(u[\alpha]v).
\end{align*}
Thus, by linearity of the boundary and section maps, we get $\delta_{k+1} \sigma_k + \sigma_{k-1}\delta_k=1_{\fnat{\C}{\Sigma_k}}$, proving exactness at $\fnat{\C}{\Sigma_k}$ and concluding the proof.
\end{proof}

\noindent
By construction, if $\Sigma$ is a finite $(n,1)$-polygraph, then each $\fnat{\cl{\Sigma}}{\Sigma_k}$ is finitely generated. In particular, every category with a finite number of $0$-cells is of homological type $\FP_0$ and every category that is finitely generated (resp. presented) is of homological type $\FP_1$ (resp. $\FP_2$). More generally, we have the following result, generalising the fact that a finite derivation type monoid is of homological type~$\FP_3$, see~\cite{CremannsOtto94,Pride95}:

\begin{corollary}
\label{Theorem:FDTimpliesFP}
For categories, the property $\FDT_n$ implies the property $\FP_n$, for every $0\leq n\leq\infty$.
\end{corollary}

\noindent
Finally, as a consequence of Theorem~\ref{MainTheorem2.0}, we get:

\begin{corollary}
If a category admits a finite and convergent presentation, then it is of homological type~$\FP_{\infty}$.
\end{corollary}

\begin{example}
Let $\C$ be a category. The Reidemeister-Fox-Squier complex corresponding to the reduced standard polygraphic resolution $\crit_{\infty}(\N\C)$ of $\C$ is (isomorphic to) the following one
\[
\xymatrix{
& \,\cdots\,
	\ar [r] ^-{\delta_{n+1}}
& \fnat{\C}{\C_{n}}
	\ar [r] ^-{\delta_{n}}
& \fnat{\C}{\C_{n-1}}
	\ar [r] ^-{\delta_{n-1}}		
& \,\cdots\,
% 	\ar [r] ^-{\delta_3}
% & \fnat{\C}{\C_2}
	\ar [r] ^-{\delta_2}
& \fnat{\C}{\C_1}
	\ar [r] ^-{\delta_1}
& \fnat{\C}{\C_0}
	\ar [r] ^-{\epsilon}
& \Zb
	\ar [r]
& 0
}
\]
where $\C_n$ is the set of composable non-identity $1$-cells $u_1$, \dots, $u_n$ of $\C$. The differential map is given, on a generator $[u_1,\dots,u_n]$, by:
\begin{align*}
\delta[u_1,\dots,u_n] 
	\:&=\: \sum_{i=0}^n (-1)^{n-i} [d_i(u_1,\dots,u_n)] \\
	\:&=\: (-1)^n u_1[u_2,\dots, u_n] 
		+ \sum_{i=1}^{n-1} (-1)^{n-i} [u_1,\dots,u_i u_{i+1},\dots,u_n] 
		+ [u_1,\dots,u_{n-1}] u_n \,. 
\end{align*}
\end{example}

%%%%%%%%%%%%%%%%%%%%%%%%%
\subsection{Homological syzygies and cohomological dimension}
%%%%%%%%%%%%%%%%%%%%%%%%%

%%%%%%%%%%%%%%%%%%%%%%
\subsubsection{Homological syzygies}

For every $k$ in $\ens{1,\dots,p+1}$, the kernel of $\delta_k$ is denoted by $h_k(\Sigma)$ and called \emph{the natural system of homological $k$-syzygies of $\Sigma$}. The kernel of $\epsilon$ is denoted by $h_0(\cl{\Sigma})$ and called \emph{the augmentation ideal of $\cl{\Sigma}$}. 

When $\Sigma_0$ is finite, then the natural system $h_0(\cl{\Sigma})$ is finitely generated if, and only if, the category~$\cl{\Sigma}$ has homological type $\FP_1$. If $\Sigma$ is a generating $1$-polygraph for a category $\C$, one checks that~$h_0(\C)$ is generated by the set $\ens{(x,1)-(1,x)\;|\;x\in \Sigma_1}$. Thus a category has homological type $\FP_1$ if, and only if, it is finitely generated.

From Theorem~\ref{Theorem:AbelianResolution}, we get a characterisation of the homological properties $\FP_n$ in terms of polygraphic resolutions:

\begin{proposition}
Let $\C$ be a category and $\Sigma$ be a partial finite polygraphic resolution of $\C$ of length~$n$. If the natural system $\h_n(\Sigma)$ of homological $n$-syzygies of~$\Sigma$ is finitely generated, then $\C$ is of homological type $\FP_{n+1}$. 
\end{proposition}

\noindent
Theorem~\ref{Theorem:AbelianResolution} also gives a description of homological $n$-syzygies in terms of critical $n$-fold branchings of a convergent presentation:

\begin{proposition}
\label{Proposition:SyzygyGenerators}
Let $\C$ be a category with a convergent presentation $\Sigma$. Then, for every $n\geq 2$, the natural system $\h_n(\Sigma)$ of homological $n$-syzygies of $\Sigma$ is generated by the elements
\[
\delta_{n+1}[\omega_{b}] 
	\:=\: [(\omega_c\rep{v})^*] \:-\: [\rep{\omega_c v}^*] 
\]
where $b=(c\rep{v},\rho_{u\rep{v}})$ ranges over the critical $n$-fold branchings of $\Sigma$.
\end{proposition}

%%%%%%%%%%%%%%%%%%%%%%
\subsubsection{Cohomological dimension}

Finally, Theorem~\ref{Theorem:AbelianResolution} gives the following bounds for the cohomological dimension of a category. We recall that the \emph{cohomological dimension} of a category $\C$, is defined, when it exists as the lowest $0\leq n\leq \infty$ such that the constant natural system $\Zb$ on $\C$ admits a projective resolution 
\[
\xymatrix{
0 
	\ar [r] 
& P_n 
	\ar [r]
& \:\cdots\:
	\ar [r]
& P_1
	\ar [r] 
& P_0
	\ar [r]
& \Zb
	\ar [r]
& 0.
}
\]
In that case, the cohomological dimension of $\C$ is denoted by $\cohdim{\C}$. In particular, when $\C$ is free, then we have $\cohdim{\F} \leq 1$, see~\cite{BauesWirsching85}. 

\begin{proposition}
The cohomological dimension of a category $\C$ admits the following upper bounds:
\begin{enumerate}[{\bf i)}]
\item The inequality $\cohdim{\C} \leq \poldim{\C}$ holds.
\item If $\C$ admits an aspherical partial polygraphic resolution of length~$n$, then $\cohdim{\C} \leq n$. 
\item If $\C$ admits a convergent presentation with no critical $n$-fold branching, then $\cohdim{\C} \leq n$. 
\end{enumerate}
\end{proposition}

%%%%%%%%%%%%%%%%%%%%%%%%%%%%%%%%%%%%%%%%%%%
\subsection{Homological syzygies and identities among relations}
\label{subsectionIAR}
%%%%%%%%%%%%%%%%%%%%%%%%%%%%%%%%%%%%%%%%%%%

In~\cite{GuiraudMalbos10smf}, the authors have introduced the natural system of identities among relations of an $n$-polygraph $\Sigma$. If $\Sigma$ is a convergent $2$-polygraph, this natural system on $\cl{\Sigma}$ is generated by the critical branchings of $\Sigma$. In Proposition~\ref{Proposition:SyzygyGenerators}, we have seen that this is also the case of the natural system of homological $2$-syzygies of $\Sigma$. In this section, we prove that, more generally, the natural systems of homological $2$-syzygies and of identities among relations of any $2$-polygraph are isomorphic.

%%%%%%%%%%%%%%%%%%%%%%%%%%%%%%%%%%%%%%%%%%%
\subsubsection{Natural systems on $n$-categories}

We recall from~\cite{GuiraudMalbos09}, that a \emph{context} of an $n$-category $\Cr$ is an $n$-cell $C$ of the $n$-category $\Cr[x]$, where $x$ is an $(n-1)$-sphere of $C$, and such that $C$ contains exactly one occurrence of the $n$-cell $x$ of $\Cr[x]$. Such a context admits a (generally non-unique) decomposition
\[
C \:=\: 
	f_n \star_{n-1} ( f_{n-1} \star_{n-2} ( \cdots 
	\star_1 f_1 \; x \; g_1 \star_1 
	\cdots ) \star_{n-2} g_{n-1} ) \star_{n-1} g_n, 
\]
where, for every $k$ in $\ens{1,\dots,n}$, $f_k$ and $g_k$ are $k$-cells of $\Cr$. The context $C$ is a \emph{whisker of $\Cr$} if the $n$-cells $f_n$ and $g_n$ are identities. If $f$ is an $n$-cell of $\Cr$ with boundary $x$, one denotes by $C[f]$ the $n$-cell of $\Cr$ obtained by replacing $x$ with $f$ in $C$. 

If $\Gamma$ is a cellular extension of $\Cr$, then every non-identity $(n+1)$-cell $f$ of $\Cr[\Gamma]$ has a decomposition 
\[
f \:=\: C_1[\phi_1] \star_n \cdots \star_n C_k[\phi_k],
\]
with $k\geq 1$ and, for every $i$ in $\ens{1,\dots,k}$, $\phi_i$ in $\Gamma$ and $C_i$ a whisker of $\Cr[\Gamma]$.

The \emph{category of contexts of $\Cr$} is denoted by $\Ct{\Cr}$, its objects are the $n$-cells of $\Cr$ and its morphisms from $f$ to $g$ are the contexts $C$ of $\Cr$ such that $C[f]=g$ holds. When $\C$ is a category, the category $\Ct{\C}$ of contexts of $\C$ is isomorphic to the category $F\C$ of factorisations of $\C$. A \emph{natural system on $\Cr$} is a $\Ct{\Cr}$-module, \ie, a functor $D$ from $\Ct{\Cr}$ to the category $\Ab$ of Abelian groups; we denote by~$D_u$ and by~$D_C$ the images of an $n$-cell~$u$ and of a context~$C$ of $\Cr$ through the functor $D$. 

%%%%%%%%%%%%%%%%%%%%%%%%%%%%%%%%%%%%%%%%%%%
\subsubsection{Identities among relations}

Let $\Sigma$ be an $n$-polygraph, seen as an $(n,n-1)$-polygraph. An $n$-cell $f$ of $\tck{\Sigma}$ is \emph{closed} when its source and target coincide. The natural system $\Pi(\Sigma)$ on $\cl{\Sigma}$ of \emph{identities among relations of $\Sigma$} is defined as follows:
\begin{itemize}
\item If $u$ is an $(n-1)$-cell of $\cl{\Sigma}$, the Abelian group $\Pi(\Sigma)_u$ is generated by one element $\iar{f}$, for each $n$-cell $f:v\dfl v$ of $\tck{\Sigma}$ such that $\cl{v}=u$, submitted to the following relations:
\begin{itemize}
\item if $f:v\fl v$ and $g:v\fl v$ are $n$-cells of $\tck{\Sigma}$, with $\cl{v}=u$, then 
\[
\iar{f\star_{n-1} g} \:=\: \iar{f}+\iar{g} \;;
\]
\item if $f:v\fl w$ and $g:w\fl v$ are $n$-cells of $\tck{\Sigma}$, with $\cl{v}=\cl{w}=u$, then
\[
\iar{f\star_{n-1} g} \:=\: \iar{g\star_{n-1} f} \;.
\]
\end{itemize}
\item If $g=C[f]$ is a factorisation in $\cl{\Sigma}$, then the morphism $\Pi(\Sigma)_{C} : \Pi(\Sigma)_f \fl \Pi(\Sigma)_g$ of groups is defined by
\[
\Pi(\Sigma)_{C} (\iar{f}) \:=\: \iar{\rep{C}[f]},
\]
where $\rep{C}$ is any representative context for $C$ in $\Sigma^*$. We recall from~\cite{GuiraudMalbos10smf} that the value of $\Pi(\Sigma)$ does not depend on the choice of $\rep{C}$, proving that $\Pi(\Sigma)$ is a natural system on $\cl{\Sigma}$ and allowing one to denote this element of $\Pi(\Sigma)_{g}$ by $C\iar{f}$. 
\end{itemize}
As consequences of the defining relations of each group $\Pi(\Sigma)_u$, we get the following equalities:
\[
\iar{1_u} \:=\: 0
\qquad\qquad
\iar{f^-} \:=\: -\iar{f}
\qquad\qquad
\iar{ g\star_{n-1} f \star_{n-1} g^- } \:=\: \iar{f}
\]
for every $n-1$-cell $u$ and every $n$-cells $f:u\fl u$ and $g:v\fl u$ of $\tck{\Sigma}$. 

\begin{lemma}
Let $\Sigma$ be a $2$-polygraph. For every closed $2$-cell $f$ of $\tck{\Sigma}$, we have $[f]=0$ in~$\fnat{\C}{\Sigma_2}$ if, and only if, $\iar{f}=0$ holds in $\Pi(\Sigma)$.
\end{lemma}

\begin{proof}
To prove that $\iar{f}=0$ implies $[f]=0$, we check that the relations defining $\Pi(\Sigma)$ are also satisfied in $\fnat{\C}{\Sigma_2}$. The first relation is given by the definition of the map $[\cdot]$. The second relation is checked as follows:
\[
[f\star_1 g] \:=\: [f]+[g] \:=\: [g]+[f] \:=\: [g\star_1 f].
\]
Conversely, let us consider a $2$-cell $f:w\dfl w$ in $\tck{\Sigma}$ such that $[f]=0$ holds. We decompose $f$ into:
\[
f \:=\: u_1 \phi_1^{\epsilon_1} v_1 \star_1 \cdots \star_1 u_p \phi_p^{\epsilon_p} v_p
\]
where $\phi_i$ is a $2$-cell of $\Sigma$, $u_i$ and $v_i$ are $1$-cells of $\tck{\Sigma}$, $\epsilon_i$ is an element of $\ens{-,+}$. Then we get:
\[
0 \:=\: [f] \:=\: \sum_{i=1}^p \epsilon_i \cl{u}_i [\phi_i] \cl{v}_i
\]
Since $\fnat{\C}{\Sigma_2}$ is freely generated, as an $F\C$-module, by the elements $[\phi]$ of $\fnat{\C}{\Sigma_2}_{\cl{\phi}}$, for $\phi$ a $2$-cell of $\Sigma$, this implies the existence of a self-inverse permutation $\tau$ of $\ens{1,\dots,p}$ such that the following relations are satisfied: 
\[
\phi_i  \:=\: \phi_{\tau(i)}
\qquad 
\cl{u}_i \:=\: \cl{u}_{\tau(i)}
\qquad
\cl{v}_i \:=\: \cl{v}_{\tau(i)}
\qquad
\epsilon_i \:=\: -\epsilon_{\tau(i)}. 
\]
Let us denote, for every $1\leq i\leq p$, the source and target of $\phi_i^{\epsilon_i}$ by $w_i$ and $w'_i$ respectively. They satisfy $\cl{w}_i=\cl{w}'_i$. We also fix a section $\rep{\;\cdot\;}$ and a left strategy $\sigma$ for the $2$-polygraph $\Sigma$. In particular, the section satisfies $\rep{u}=\rep{v}$ for every $1$-cells $u$ and $v$ such that $\cl{u}=\cl{v}$.

For every $1\leq i\leq p$, we denote by $f_i$ the following $2$-cell of $\tck{\Sigma}$:
\[
f_i \:=\: \sigma^-_{u_i w_i v_i} \star_1 u_i\phi_i^{\epsilon_i}v_i \star_1 \sigma_{u_i w'_i v_i}.
\]
Using the facts that $w$ is equal to both $u_1w_1v_1$ and $u_p w'_p v_p$ and that $u_iw'_iv_i$ is equal to $u_{i+1} w_{i+1} v_{i+1}$ for every $1\leq i<p$, we can write the $2$-cell $f$ of the $(2,1)$-category $\tck{\Sigma}$ as the following composite:
\[
f \:=\: \sigma_w \star_1 f_1 \star_1 f_2 \star_1 \cdots \star_1 f_p \star_1 \sigma^-_w.
\]
As a consequence, we get:
\[
\iar{f} 
	\:=\: \iar{\sigma^-_w\star_1 f\star_1 \sigma_w} 
	\:=\: \sum_{i=1}^p \iar{f_i}.
\]
In order to conclude this proof, it is sufficient to check that, for every $1\leq i\leq p$, we have the equality $\iar{f_{\tau(i)}} = -\iar{f_i}$. Let us fix an $i$ in $\ens{1,\dots,p}$ and let us compute $\iar{f_i}$. Since $\sigma$ is a left normalisation strategy, we have
\[
\sigma_{u_i w_i v_i} \:=\: \sigma_{u_i} w_i v_i \star_1 \sigma_{\rep{u}_i w_i} v_i \star_1 \sigma_{\rep{u_i w}_i v_i}
\]
and, using the fact that $\rep{u_i w}_i=\rep{u_i w}'_i$, 
\[
\sigma_{u_i w'_i v_i} \:=\: \sigma_{u_i} w'_i v_i \star_1 \sigma_{\rep{u}_i w'_i} v_i \star_1 \sigma_{\rep{u_i w}_i v_i}.
\]
This gives:
\[
\iar{f_i} 
	\:=\: \iar{ 
		\sigma^-_{\rep{u_i w_i}v_i} \star_1 \sigma^-_{\rep{u}_i w_i} v_i \star_1 \sigma^-_{u_i} w_i v_i
		\star_1 u_i\phi_i^{\epsilon_i}v_i 
		\star_1 \sigma_{u_i} w'_i v_i \star_1 \sigma_{\rep{u}_i w'_i} v_i \star_1 \sigma_{\rep{u_i w_i}v_i}
		}.
\]
We can remove the final $2$-cell $\sigma_{\rep{u_i w_i}v_i}$ and its inverse at the beginning of  the composition, to get:
\[
\iar{f_i} 
	\:=\: \iar{ 
		\sigma^-_{\rep{u}_i w_i} v_i \star_1 \sigma^-_{u_i} w_i v_i
		\star_1 u_i\phi_i^{\epsilon_i}v_i 
		\star_1 \sigma_{u_i} w'_i v_i \star_1 \sigma_{\rep{u}_i w'_i} v_i
		}.
\]
Then, we use exchange relations to get:
\[
\iar{f_i}
	\:=\: \iar{ 
		\sigma^-_{\rep{u}_i w_i}
		\star_1 \rep{u}_i\phi_i^{\epsilon_i} 
		\star_1 \sigma_{\rep{u}_i w'_i}
	} \cl{v}_i.
\]
Now, let us compute $\iar{f_{\tau(i)}}$. We already know that $\phi_{\tau(i)}=\phi_i$ and $\epsilon_{\tau(i)}=-\epsilon_i$. As a consequence, we get $w_{\tau(i)}=w'_i$ and $w'_{\tau(i)}=w_i$. Moreover, we have $\rep{u}_{\tau(i)}=\rep{u}_{i}$, so that we have:
\begin{align*}
\iar{f_{\tau(i)}} 
	\:=\: &\iar{ 
		\sigma^-_{\rep{u_i w_i}v_{\tau(i)}}
			\star_1 \sigma^-_{\rep{u}_i w'_i} v_{\tau(i)} 
			\star_1 \sigma^-_{u_i} w'_i v_{\tau(i)} \\
		&\star_1 u_{\tau(i)} \phi_i^{-\epsilon_i} v_{\tau(i)} 
		\star_1 \sigma_{u_i} w_i v_{\tau(i)} 
			\star_1 \sigma_{\rep{u}_i w_i} v_{\tau(i)} 
			\star_1 \sigma_{\rep{u_i w_i} v_{\tau(i)}}
		}.
\end{align*}
We remove the $2$-cell $\sigma_{\rep{u_i w_i} v_{\tau(i)}}$ and its inverse and, then, we use exchange relations and $\cl{v}_{\tau(i)}=\cl{v}_i$, in order to get:
\[
\iar{f_{\tau(i)}} 
	\:=\: \iar{
		\sigma^-_{\rep{u}_i w'_i}  
		\star_1 \rep{u}_i\phi_i^{-\epsilon_i} 
		\star_1 \sigma_{\rep{u}_i w_i} 
	} \cl{v}_i
	\:=\: -\iar{f_i}.
\] 
This implies $\iar{f}=0$, thus concluding the proof.
\end{proof}

\begin{lemma}
Let $\Sigma$ be a $2$-polygraph. For every element $a$ in $h_2(\Sigma)$, there exists a closed $2$-cell $f$ in~$\tck{\Sigma}$ such that $a=[f]$ holds.
\end{lemma}

\begin{proof}
Let $w$ be the $1$-cell of $\cl{\Sigma}$ such that $a$ belongs to $\fnat{\cl{\Sigma}}{\Sigma_2}_w$ and let $\Sigma_3$ be a homotopy basis of the $(2,1)$-category~$\tck{\Sigma}$. Since $\delta_2(a)=0$, Theorem~\ref{Theorem:AbelianResolution} implies the existence of an element $b$ in $\fnat{\cl{\Sigma}}{\Sigma_3}_w$ such that $a=\delta_3(b)$ holds. By definition of $\fnat{\cl{\Sigma}}{\Sigma_3}_w$, we can write
\[
b \:=\: \sum_{i=1}^p \epsilon_i u_i [\alpha_i] v_i 
\]
with, for every $1\leq i\leq p$, $\alpha_i$ in $\Sigma_3$, $u_i$ and $v_i$ in $\cl{\Sigma}$ and $\epsilon_i$ in $\ens{-,+}$ such that $u_i\cl{\alpha}_iv_i=w$ holds. We fix a section $\rep{\;\cdot\;}$ of $\Sigma$ and we choose $2$-cells
\[
g_i \::\: \rep{w} \:\dfl\: \rep{u}_i s_1(\alpha_i^{\epsilon_i}) \rep{v}_i
\qquad\text{and}\qquad
h_i \::\: \rep{u}_i t_1(\alpha_i^{\epsilon_i}) \rep{v}_i \:\dfl\: \rep{w}.
\]
Let $A$ be the following $3$-cell of $\tck{\Sigma}_3$:
\[
A 
	\:=\: \big( g_1\star_1 \rep{u}_1\alpha_1^{\epsilon_1} \rep{v}_1 \star_1 h_1 \big)
			\star_1 \cdots
			\star_1 \big( g_k\star_1 \rep{u}_k\alpha_k^{\epsilon_k} \rep{v}_k \star_1 h_k \big). 		
\]
By definition of $[\cdot]$ on $3$-cells, we have 
\[
[A] 
	\:=\: \sum_{i=1}^p \big[ g_i\star_1 \rep{u}_i\alpha_i^{\epsilon_i}\rep{v}_i \star_1 h_i \big]
	\:=\: \sum_{i=1}^p \big( [1_{g_i}] + \epsilon_i u_i[\alpha_i] v_i + [1_{h_i}] \big)
	\:=\: b.
\]
Finally, we get:
\[
a \:=\: \delta_3[A] \:=\: [s(A)]-[t(A)] \:=\: [s(A)\star_1 t(A)^-]. 
\]
Hence $f=s(A)\star_1 t(A)^-$ is a closed $2$-cell of $\tck{\Sigma}$ that satisfies $a=[f]$.
\end{proof}

\begin{theorem}
\label{Theorem:IsomorphismPiH2}
For every $2$-polygraph $\Sigma$, the natural systems of homological $2$-syzygies and of identities among relations of $\Sigma$ are isomorphic.
\end{theorem}

\begin{proof}
We define a morphism of natural systems $\Phi : \Pi(\Sigma) \fl \h_2(\Sigma)$ by
\[
\Phi \big( \iar{f} \big) \:=\: [f] \;.
\]
This definition is correct, since the defining relations of $\Pi(\Sigma)$ also hold in $\fnat{\C}{\Sigma_2}$, hence in $h_2(\Sigma)$. Let us check that $\Phi$ is a morphism of natural systems. Indeed, we have
\[
\Phi(u\iar{f}v) \:=\: \Phi(\iar{\rep{u}f\rep{v}}) \:=\: [\rep{u}f\rep{v}] \:=\: u[f]v \:=\: u\Phi(\iar{f})v
\]
for every $2$-cell $f:w\dfl w$ in $\tck{\Sigma}$ and $1$-cells $u$, $v$ in $\cl{\Sigma}$ such that $\rep{u}f\rep{v}$ is defined. 

Now, let us define a morphism of natural systems $\Psi : \h_2(\Sigma) \fl \Pi(\Sigma)$. Let $a$ be an element of $\h_2(\Sigma)_w$. Then there exists a closed $2$-cell $f:u\dfl u$ such that $a=[f]$ and $w=\cl{u}$. We define 
\[
\Psi(a) \:=\: \iar{f}.
\]
This definition does not depend on the choice of $f$. Indeed, let us assume that $g:v\dfl v$ is a closed $2$-cell such that $a=[g]$ holds. It follows that $\cl{v}=w=\cl{u}$. Hence, we can choose a $2$-cell $h:u\dfl v$ in $\tck{\Sigma}$. Then we have:
\[
a \:=\: [f] \:=\: [g] \:=\: [h\star_1 g\star_1 h^-].
\]
As a consequence, we get: 
\[
[f\star_1 h^-\star_1 g^-\star_1 h] \:=\: [f] - [h \star_1 g \star_1 h^-] \:=\: 0.
\]
Thus
\[
0 \:=\: \iar{f\star_1 h^- \star_1 g^- \star_1 h} \:=\: \iar{f} - \iar{h\star_1 g\star_1 h^-} \:=\: \iar{f} -\iar{g}.
\]
The relations $\Psi\Phi=1_{\Pi(\Sigma)}$ and $\Phi\Psi=1_{\h_2(\Sigma)}$ are direct consequences of the definitions of $\Phi$ and $\Psi$.
\end{proof}

%%%%%%%%%%%%%%%%%%%%%%%%%%%%%%%%%%%%%%%%%%%
\subsection{Abelian finite derivation type}
\label{subsectionAbelianTDF}
%%%%%%%%%%%%%%%%%%%%%%%%%%%%%%%%%%%%%%%%%%%

An $(n,n-1)$-category $\Cr$ is \emph{Abelian} if, for every $(n-1)$-cell $u$ of $\Cr$, the group $\Aut^{\Cr}_u$ of closed $n$-cells of $\Cr$ with source $u$ is Abelian. The Abelian $(n,n-1)$-category generated by $\Cr$ is the $(n,n-1)$-category, denoted by $\ab{\Cr}$ and defined as the quotient of $\Cr$ by the cellular extension that contains one $n$-sphere
\[
f\star_{n-1} g \:\longrightarrow\: g\star_{n-1} f
\]
for every closed $n$-cells $f$ and $g$ of $\Cr$ with the same source.

One says that an $n$-polygraph $\Sigma$ is of \emph{Abelian finite derivation type}, $\FDTAB$ for short, when the Abelian $(n,n-1)$-category $\abtck{\Sigma}$ admits a finite homotopy basis. In this section, we prove that an $n$-polygraph is $\FDTAB$ if, and only if, the natural system~$\Pi(\Sigma)$ of identities among relations of $\Sigma$ is finitely generated.

In~\cite{GuiraudMalbos10smf}, it is proved that, given an $n$-polygraph $\Sigma$, there exists an isomorphism of natural systems on the free $(n-1)$-category $\Sigma^*_{n-1}$:
\[
\Pi(\Sigma)_{\cl{u}} \:\simeq\: \Aut^{\abtck{\Sigma}}_u.
\]
In fact, this property characterises the natural system $\Pi(\Sigma)$ on the $(n-1)$-category $\cl{\Sigma}$ up to isomorphism. In~\cite{GuiraudMalbos10smf}, we also proved the following result.

\begin{lemma}[\cite{GuiraudMalbos10smf}]
\label{LemmaHomotopiesBoucles}
Let $\Cr$ be an $(n,n-1)$-category and let $\Br$ be a family of closed $n$-cells of $\Cr$. The following assertions are equivalent: 
\begin{enumerate}[\bf i)]
\item The cellular extension $\tilde{\Br} = \ens{ \tilde{\beta}:\beta\fl 1_{s\beta}, \:\beta\in\Br }$ of $\C$ is a homotopy basis.
\item Every closed $n$-cell $f$ in $\Cr$ has a decomposition
\begin{equation} \label{HomotopyBaseTDFLemmaEq}
f \:=\: \left( g_1 \star_{n-1} C_1\left[ \beta_1^{\epsilon_1} \right] \star_{n-1} g_1^- \right) 
	\:\star_{n-1}\: \cdots \:\star_{n-1}\: 
	\left( g_p \star_{n-1} C_p \left[ \beta_p^{\epsilon_p} \right] \star_{n-1} g_p^- \right)
\end{equation}
with, for every $1\leq i\leq p$, $\beta_i$ in $\Br$, $\epsilon_i$ in $\ens{-,+}$, $C_i$ a whisker of $\Cr$ and $g_i$ an $n$-cell of~$\Cr$.
\end{enumerate}
\end{lemma}

\begin{proposition}
\label{Proposition:IartfIffFdtab}
An $n$-polygraph $\Sigma$ is $\FDTAB$ if, and only if, the natural system $\Pi(\Sigma)$ on $\cl{\Sigma}$ is finitely generated.
\end{proposition}

\begin{proof}
Let us assume that the $n$-polygraph $\Sigma$ is $\FDTAB$. Then the Abelian $(n,n-1)$-category $\abtck{\Sigma}$ has a finite homotopy basis $\Br$. Let $\partial\Br$ be the set of closed $n$-cells of $\abtck{\Sigma}$ defined by: 
\[
\partial\Br \:=\: \ens{ \; \partial\beta = s(\beta)\star_{n-1} t(\beta)^- \;, \: \beta\in \Br \;}.
\]
By Lemma~\ref{LemmaHomotopiesBoucles}, any closed $n$-cell $f$ in $\abtck{\Sigma}$ can be written 
\[
f \:=\: \big( g_1\star_{n-1} C_1[\partial\beta_1^{\epsilon_1}] \star_{n-1} g_1^- \big) 
	\star_{n-1}\dots\star_{n-1}
	\big( g_p\star_{n-1} C_p[\partial\beta_p^{\epsilon_p}] \star_{n-1} g_p^- \big),
\]
with, for every $1\leq i\leq p$, $\beta_i$ in $\Br$, $\epsilon_i$ in $\ens{-,+}$, $C_i$ a whisker of $\tck{\Sigma}$ and $g_i$ an $n$-cell of $\tck{\Sigma}$. As a consequence, for any identity among relations $\iar{f} $ in $\Pi(\Sigma)$, we have: 
\[
\iar{f} 
	\:=\: \sum_{i=1}^k \epsilon_i \iar{g_i \star_{n-1} C_i[\partial\beta_i] \star_{n-1} g_i^-}
	\:=\: \sum_{i=1}^k \epsilon_i C_i\iar{\partial\beta_i}.
\]
Thus, the elements of $\iar{\partial\Br}$ form a generating set for the natural system $\Pi(\Sigma)$.  Hence $\Pi(\Sigma)$ is finitely generated. 

Conversely, suppose that the natural system $\Pi(\Sigma)$ is finitely generated. This means that there exists a finite set~$\Br$ of closed $n$-cells of $\abtck{\Sigma}$ such that the following property is satisfied: for every $(n-1)$-cell $u$ of $\cl{\Sigma}$ and every closed $n$-cell $f$ with source $v$ of $\abtck{\Sigma}$ such that $\rep{v}=u$, one can write
\[
\iar{f} \:=\: \sum_{i=1}^p \epsilon_i C_i \iar{\beta_i}
\]
with, for every $1\leq i\leq p$, $\beta_i$ in $\Br$, $C_i$ a whisker of $\cl{\Sigma}$ and $\epsilon_i$ an integer; moreover, each whisker $C_i$ is such that, for every representative $\rep{C}_i$ of $C_i$ in $\abtck{\Sigma}$, $\rep{C}_i[\beta_i]$ is a closed $n$-cell of $\abtck{\Sigma}$ whose source $v_i$ satisfies $\cl{v}_i=v$. We fix, for every~$i$, an $n$-cell $g_i:\rep{v}\dfl v_i$ in $\tck{\Sigma}$. Then, the properties of $\Pi(\Sigma)$ imply: 
\begin{align*}
\iar{f} 
	\:&=\: \sum_{i=1}^p \iar{g_i \star_{n-1} \rep{C}_i[\beta_i^{\epsilon_i}] \star_{n-1} g_i^-} \\
	\:&=\: \iar{ \big( g_1 \star_{n-1} \rep{C}_1[\beta_1^{\epsilon_1}] \star_{n-1} g_1^- \big)
		\star_{n-1} \cdots \star_{n-1}
		\big( g_p \star_{n-1} \rep{C}_p[\beta_p^{\epsilon_p}] \star_{n-1} g_p^- \big) }.
\end{align*}
We use the isomorphism between $\Pi(\Sigma)_u$ and $\Aut^{\abtck{\Sigma}}_v$ and Lemma~\ref{LemmaHomotopiesBoucles} to deduce that the cellular extension $\tilde{\Br}=\ens{\tilde{\beta}:\beta\fl 1_{s\beta}, \beta\in \Br}$ of $\abtck{\Sigma}$ is a homotopy basis. Thus $\Sigma$ is $\FDTAB$.
\end{proof}

\noindent
In~\cite{GuiraudMalbos10smf}, the authors have proved that the property to be finitely generated for $\Pi(\Sigma)$ is Tietze-invariant for finite polygraphs: if $\Sigma$ and $\Upsilon$ are two Tietze-equivalent finite $n$-polygraphs, then the natural system~$\Pi(\Sigma)$ is finitely generated if, and only if, the natural system $\Pi(\Upsilon)$ is finitely generated.

By Proposition~\ref{Proposition:IartfIffFdtab}, we deduce that the property $\FDTAB$ is Tietze-invariant for finite polygraphs, so that one can say that an $n$-category is $\FDTAB$ when it admits a presentation by a finite $(n+1)$-polygraph which is $\FDTAB$. In this way, the following result relates the homological property $\FP_3$ and the homotopical property $\FDTAB$.

\begin{theorem}
\label{Theorem:TDFAB<=>FP_3}
Let $\C$ be a category with a finite presentation $\Sigma$. The following conditions are equivalent:
\begin{enumerate}[{\bf i)}]
\item the category $\C$ is of homological type $\FP_3$,
\item the natural system $\h_2(\Sigma)$ on $\C$ is finitely generated,
\item the natural system $\Pi(\Sigma)$ on $\C$ is finitely generated,
\item the category $\C$ is $\FDTAB$.
\end{enumerate}
\end{theorem}

\begin{proof}
The equivalence between {\bf i)} and {\bf ii)} comes from the definition of the property $\FP_3$. The equivalence between {\bf ii)} and {\bf iii)} is a consequence of Theorem~\ref{Theorem:IsomorphismPiH2}. The last equivalence is given by Proposition~\ref{Proposition:IartfIffFdtab}.
\end{proof}

\noindent
In Corollary~\ref{Theorem:FDTimpliesFP}, we have seen that $\FDT_3$ implies $\FP_3$. We expect that the reverse implication is false in general, which is equivalent to proving that $\FDTAB$ does not imply $\FDT_3$, since $\FP_3$ is equivalent to $\FDTAB$ for finitely presented categories.

%%%%%%%%%%%%%%%%%%%%%%%%%%%%%%%%%%%%%%%%%%
%%%%%%%%%%%%%%%%%%%%%%%%%%%%%%%%%%%%%%%%%%
\section{Examples}
\label{SectionExamples}
%%%%%%%%%%%%%%%%%%%%%%%%%%%%%%%%%%%%%%%%%%
%%%%%%%%%%%%%%%%%%%%%%%%%%%%%%%%%%%%%%%%%%

%%%%%%%%%%%%%%%%%%%%%%%%%%%%%%%%%%%%%%%%%%%
\subsection{A concrete example of reduced standard polygraphic resolution}
%%%%%%%%%%%%%%%%%%%%%%%%%%%%%%%%%%%%%%%%%%%

Let us denote by $A$ the monoid with one non-unit element, $a$, and with product given by $aa=a$. The standard presentation of $A$, seen as a category, is the reduced and convergent $2$-polygraph, denoted by $\As$, with one $0$-cell (for $1$), one $1$-cell (for $a$) and one $2$-cell $aa\dfl a$. Here we use diagrammatic notations, where $a$ is denoted by a vertical string $\:\twocell{1}\:$ and the $2$-cell $aa\dfl a$ is pictured as~\twocell{mu}. The $2$-polygraph $\As$ has one critical branching:
\[
\big(\: \twocell{mu}\:\twocell{1} \:, \: \twocell{1}\:\twocell{mu} \:\big) \;.
\]
The corresponding generating confluence (for the rightmost normalisation strategy) is the $3$-cell:
\[
\xymatrix{
{\twocell{(mu *0 1) *1 mu}}  
	\ar@3 [r] ^-{\twocell{alpha}}
& {\twocell{(1 *0 mu) *1 mu}}
}
\]
By extending $\As$ with that $3$-cell, one gets a finite, acyclic $(3,1)$-polygraph, still denoted by $\As$ and which is a partial polygraphic resolution of $A$ of length~$3$. We conclude that the monoid $A$ has the property $\FDT_3$ and, thus, is of homological type $\FP_3$. In particular, the natural system $\h_2(\As)$ of homological $2$-syzygies of $\As$ is generated by the following element:
\[
\delta_3 [\twocell{alpha} ] 
	\:=\: \left[ {\twocell{(mu *0 1) *1 mu}} \right] - \left[ {\twocell{(1 *0 mu) *1 mu}} \right]
	\:=\: \big[\twocell{mu}\:\twocell{1} \big] + \big[\twocell{mu}\big] 
			- \big[\twocell{1}\:\twocell{mu}\big] - \big[\twocell{mu}\big]
	\:=\: \big[\twocell{mu}\big] a - a \big[\twocell{mu}\big].
\]
The $2$-polygraph $\As$ has exactly one critical triple branching: 
\[
b \:=\: \big(\: 
	\twocell{mu}\:\twocell{1}\:\twocell{1} \:,\:
	\twocell{1}\:\twocell{mu}\:\twocell{1} \:,\:
	\twocell{1}\:\twocell{1}\:\twocell{mu} 
\:\big) \;.
\]
This triple critical branching $b$ has shape $(f\rep{a},g\rep{a},\rho_{u\rep{a}})$ with 
\[
f\:=\: \twocell{mu}\:\twocell{1}
\qquad\text{and}\qquad
g\:=\: \twocell{1}\:\twocell{mu} \;.
\]
We compute the $3$-cells $\omega_{f,g}$, $\sigma^*_{f\rep{a}}$ and $\sigma^*_{g\rep{a}}$, using their definitions and the properties of the rightmost normalisation strategy $\sigma$, to get:
\[
\omega_{f,g}\rep{a} \:=\: \twocell{alpha}\:\twocell{1}
\qquad\qquad
\sigma^*_{f\rep{a}} \:=\: \twocell{(mu *0 2) *1 alpha} \star_2 \twocell{(2 *0 mu) *1 alpha}
\qquad\qquad
\sigma^*_{g\rep{a}} \:=\: \twocell{(1 *0 mu *0 1) *1 alpha} \star_2 \twocell{(1 *0 alpha) *1 mu} \;.
\] 
We fill the diagram defining $\omega_b=\twocell{aleph}$, obtaining:
\[
\xymatrix@!C@C=1.5em{
& {\twocell{1}\:\twocell{1}\:\twocell{1}}
	\ar@2 [rr] ^-{\twocell{mu}\:\twocell{1}}
	\ar@{} [dr] |-{\twocell{alpha}\:\twocell{1}}
&& {\twocell{1}\:\twocell{1}}
	\ar@2 [dr] ^-{\twocell{mu}}
&&&&& {\twocell{1}\:\twocell{1}\:\twocell{1}}
	\ar@2 [rr] ^-{\twocell{mu}\:\twocell{1}}
	\ar@2 [dr] |-{\twocell{1}\:\twocell{mu}}
&& {\twocell{1}\:\twocell{1}}
	\ar@2 [dr] ^-{\twocell{mu}}
\\
{\twocell{1}\:\twocell{1}\:\twocell{1}\:\twocell{1}}
	\ar@2 [ur] ^-{\twocell{mu}\:\twocell{1}\:\twocell{1}}
	\ar@2 [rr] |-{\twocell{1}\:\twocell{mu}\:\twocell{1}}
	\ar@2 [dr] _-{\twocell{1}\:\twocell{1}\:\twocell{mu}}
&& {\twocell{1}\:\twocell{1}\:\twocell{1}}
	\ar@2 [ur] |-{\twocell{mu}\:\twocell{1}}
	\ar@2 [dr] |-{\twocell{1}\:\twocell{mu}}
	\ar@{} [rr] |-{\twocell{alpha}}
&& {\twocell{1}}
& \strut 
	\ar@4 [r] ^*+\txt{$\twocell{aleph}$}
& \strut
& {\twocell{1}\:\twocell{1}\:\twocell{1}\:\twocell{1}}
	\ar@2 [ur] ^-{\twocell{mu}\:\twocell{1}\:\twocell{1}}
	\ar@2 [dr] _-{\twocell{1}\:\twocell{1}\:\twocell{mu}}
	\ar@{} [rr] |-{\copyright}
&& {\twocell{1}\:\twocell{1}}
	\ar@2 [rr] |-{\twocell{mu}}
	\ar@{} [ur] |-{\twocell{alpha}}
	\ar@{} [dr] |-{\twocell{alpha}}
&& {\twocell{1}}
\\
& {\twocell{1}\:\twocell{1}\:\twocell{1}}
	\ar@2 [rr] _-{\twocell{1}\:\twocell{mu}}
	\ar@{} [ur] |-{\twocell{1}\:\twocell{alpha}}
&& {\twocell{1}\:\twocell{1}}
	\ar@2 [ur] _-{\twocell{mu}}
&&&&& {\twocell{1}\:\twocell{1}\:\twocell{1}}
	\ar@2 [ur] |-{\twocell{mu}\:\twocell{1}}
	\ar@2 [rr] _-{\twocell{1}\:\twocell{mu}}
&& {\twocell{1}\:\twocell{1}}
	\ar@2 [ur] _-{\twocell{mu}}
}
\]
Contracting one dimension, we see that the $4$-cell \twocell{aleph} is, in fact, Mac Lane's pentagon, or Stasheff's polytope $K_4$:
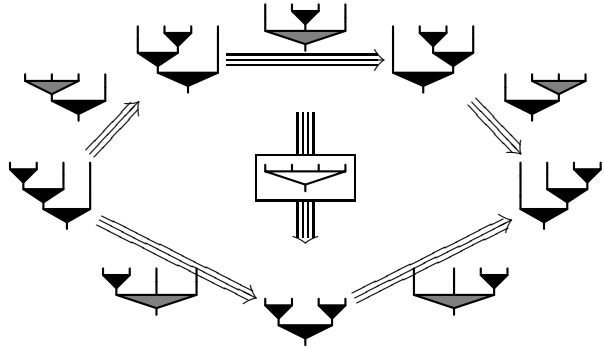
\[
\xymatrix@R=0em@C=1em{
& {\twocell{(1 *0 mu *0 1) *1 (mu *0 1) *1 mu }}
	\ar@3 [rr] ^-{\twocell{(1 *0 mu *0 1) *1 alpha}} 
& {\strut}
	\ar@4 []!<0pt,-20pt>;[dddd]!<0pt,30pt> |-{\boxed{\twocell{aleph}}}
& {\twocell{(1 *0 mu *0 1) *1 (1 *0 mu) *1 mu}}
	\ar@3 [ddr] ^-{\twocell{(1 *0 alpha) *1 mu}}
\\ 
\strut
\\
{\twocell{(mu *0 2) *1 (mu *0 1) *1 mu}}
	\ar@3 [uur] ^-{\twocell{(alpha *0 1) *1 mu}}
	\ar@3 [ddrr] _-{\twocell{(mu *0 2) *1 alpha}}
&&&& {\twocell{(2 *0 mu) *1 (1 *0 mu) *1 mu}}
\\
&& \strut
\\
&& {\twocell{(mu *0 mu) *1 mu }}
	\ar@3 [uurr] _-{\twocell{(2 *0 mu) *1 alpha}}
}
\]
We get a finite, acyclic $(4,1)$-polygraph $\crit_3(\As)$ which is a partial polygraphic resolution of $A$ of length~$4$, proving that $A$ has the property $\FDT_4$ and, as a consequence, that it is of homological type $\FP_4$. In particular, the natural system $\h_3(\As)$ of homological $3$-syzygies of $\As$ is generated by the following element:
\begin{align*}
\delta_4[\twocell{aleph}] \:
	&=\: \left[\twocell{(alpha *0 1) *1 mu}\right]
		+ \left[\twocell{(1 *0 mu *0 1) *1 alpha}\right]
		+ \left[\twocell{(1 *0 alpha) *1 mu}\right]
		- \left[\twocell{(mu *0 2) *1 alpha}\right]
		- \left[\twocell{(2 *0 mu) *1 alpha}\right] \\
	&=\: \left[\twocell{alpha}\:\twocell{1}\:\right]
		+ \left[\twocell{alpha}\right]
		+ \left[\:\twocell{1}\:\twocell{alpha}\right] 
		- \left[\twocell{alpha}\right] 
		- \left[\twocell{alpha}\right] \\
	&=\: a\left[\twocell{alpha}\right] - \left[\twocell{alpha}\right] + \left[\twocell{alpha}\right] a . 
\end{align*}
Iterating the process, we get a resolution of $A$ by an acyclic $(\infty,1)$-polygraph $\crit_{\infty}(\As)$. For every natural number, $\crit_{\infty}(\As)$ has exactly one $n$-cell, whose shape is Stasheff's polytope $K_n$. For example, in dimension~$5$, the generating $5$-cell $\omega_b$ is associated to the following critical quadruple branching:
\[
b \:=\: \big(\:
	\twocell{mu}\:\twocell{1}\:\twocell{1}\:\twocell{1} \:,\:
	\twocell{1}\:\twocell{mu}\:\twocell{1}\:\twocell{1} \:,\:
	\twocell{1}\:\twocell{1}\:\twocell{mu}\:\twocell{1} \:,\:
	\twocell{1}\:\twocell{1}\:\twocell{1}\:\twocell{mu} 
\:\big) \;.
\]
To compute the source and target of the corresponding $4$-cell $\omega_b$, we use the inductive construction of the rightmost strategy $\sigma$. Alternatively, one can also start from the $2$-dimensional source and target of~$\omega_b$, which are obtained as the $2$-cells associated to the source $aaaaa$ of the critical quadruple branching~$b$ by the leftmost and the rightmost strategies, respectively:
\[
s_2(\omega_b) \:=\: \twocell{(mu *0 3) *1 (mu *0 2) *1 (mu *0 1) *1 mu}
\qquad\qquad\text{and}\qquad\qquad
t_2(\omega_b) \:=\: \twocell{(3 *0 mu) *1 (2 *0 mu) *1 (1 *0 mu) *1 mu} .
\]
Then one computes all the possible $3$-cells from $s_2(\omega_b)$ to $t_2(\omega_b)$ and one fills all the $3$-dimensional spheres with $4$-cells built from the generating $4$-cell~\twocell{aleph}. Either way, we obtain the following composite $4$-cell as the source of $\omega_b$:
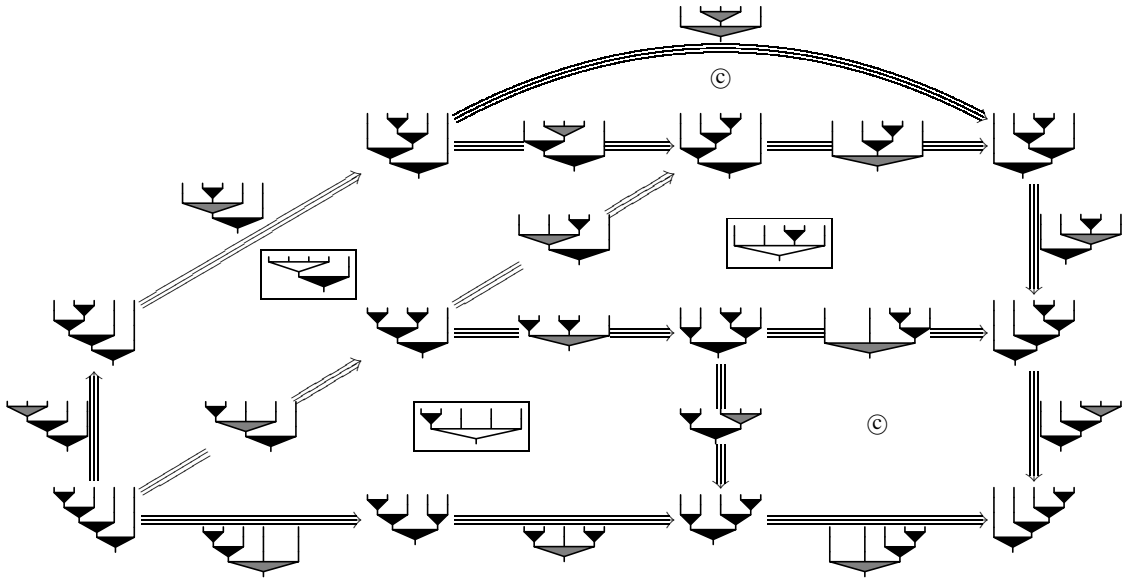
\[
\scalebox{0.75}{
\xymatrix@C=10em@R=5em{
& {\twocell{(1 *0 mu *0 2) *1 (1 *0 mu *0 1) *1 (mu *0 1) *1 mu}}
	\ar@3 @/^10ex/ [rr] ^-{\twocell{(1 *0 alpha *0 1) *1 alpha}} 
		_-{}="src0"
	\ar@3 [r] |-{\twocell{(1 *0 alpha *0 1) *1 (mu *0 1) *1 mu}}
& {\twocell{(2 *0 mu *0 1) *1 (1 *0 mu *0 1) *1 (mu *0 1) *1 mu}}
	\ar@3 [r] |-{\twocell{(2 *0 mu *0 1) *1 (1 *0 mu *0 1) *1 alpha}}
%	\ar@4 []!<30pt,-25pt>;[d]!<30pt,25pt> |-{\boxed{\twocell{(2 *0 mu *0 1) *1 aleph}}}
	\ar@{} [d] ^-{\boxed{\twocell{(2 *0 mu *0 1) *1 aleph}}}
	\ar@{} "src0"; |-{\copyright}
& {\twocell{(2 *0 mu *0 1) *1 (1 *0 mu *0 1) *1 (1 *0 mu) *1 mu}}
	\ar@3 [d] ^-{\twocell{(2 *0 mu *0 1) *1 (1 *0 alpha) *1 mu}}
\\
{\twocell{(1 *0 mu *0 2) *1 (mu *0 2) *1 (mu *0 1) *1 mu}}
	\ar@3 [ur] ^-{\twocell{(1 *0 mu *0 2) *1 (alpha *0 1) *1 mu}}
		_-{}="src1"
& {\twocell{(mu *0 mu *0 1) *1 (mu *0 1) *1 mu}}
	\ar@3 [ur] |-{\twocell{(2 *0 mu *0 1) *1 (alpha *0 1) *1 mu}}
	\ar@3 [r] |-{\twocell{(mu *0 mu *0 1) *1 alpha}}
%	\ar@4 "src1"!<40pt,40pt>;[]!<-40pt,-40pt> |-{\boxed{\twocell{(aleph *0 1) *1 mu}}}
	\ar@{} "src1"; |-{\boxed{\twocell{(aleph *0 1) *1 mu}}}
%	\ar@4 []!<30pt,-25pt>;[d]!<30pt,25pt> |-{\boxed{\twocell{(mu *0 3) *1 aleph}}}
	\ar@{} [d] ^-{\boxed{\twocell{(mu *0 3) *1 aleph}}}
& {\twocell{(2 *0 mu *0 1) *1 (mu *0 mu) *1 mu}}
	\ar@3 [r] |-{\twocell{(2 *0 mu *0 1) *1 (2 *0 mu) *1 alpha}}
	\ar@3 [d] |-{\twocell{(mu *0 alpha) *1 mu}}
	\ar@{} [dr] |-{\copyright}
& {\twocell{(2 *0 mu *0 1) *1 (2 *0 mu) *1 (1 *0 mu) *1 mu}}
	\ar@3 [d] ^-{\twocell{(2 *0 alpha) *1 (1 *0 mu) *1 mu}}
\\
{\twocell{(mu *0 3) *1 (mu *0 2) *1 (mu *0 1) *1 mu}}
	\ar@3 [u] ^-{\twocell{(alpha *0 2) *1 (mu *0 1) *1 mu}}
	\ar@3 [ur] |-{\twocell{(mu *0 3) *1 (alpha *0 1) *1 mu}}
	\ar@3 [r] _-{\twocell{(mu *0 3) *1 (mu *0 2) *1 alpha}}
& {\twocell{(mu *0 3) *1 (mu *0 mu) *1 mu}}
	\ar@3 [r] _-{\twocell{(mu *0 1 *0 mu) *1 alpha}}
& {\twocell{(3 *0 mu) *1 (mu *0 mu) *1 mu}}
	\ar@3 [r] _-{\twocell{(3 *0 mu) *1 (2 *0 mu) *1 alpha}}
& {\twocell{(3 *0 mu) *1 (2 *0 mu) *1 (1 *0 mu) *1 mu}}
}
}
\]
And the following composite $4$-cell is the target of $\omega_b$:
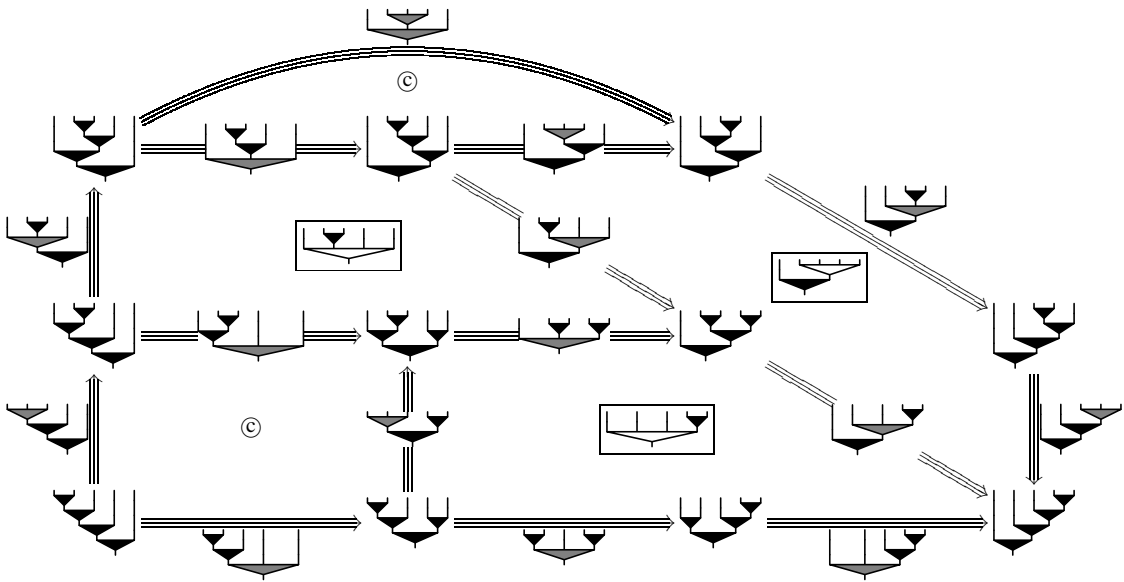
\[
\scalebox{0.75}{
\xymatrix@C=10em@R=5em{
{\twocell{(1 *0 mu *0 2) *1 (1 *0 mu *0 1) *1 (mu *0 1) *1 mu}}
	\ar@3 @/^10ex/ [rr] ^-{\twocell{(1 *0 alpha *0 1) *1 alpha}} 
		_-{}="src0"
	\ar@3 [r] |-{\twocell{(1 *0 mu *0 2) *1 (1 *0 mu *0 1) *1 alpha}}
& {\twocell{(1 *0 mu *0 2) *1 (1 *0 mu *0 1) *1 (1 *0 mu) *1 mu}}
	\ar@3 [r] |-{\twocell{(1 *0 alpha *0 1) *1 (1 *0 mu) *1 mu}}
	\ar@3 [dr] |-{\twocell{(1 *0 mu *0 2) *1 (1 *0 alpha) *1 mu}}
	\ar@{} "src0"; |-{\copyright}
	\ar@{} [d] _-{\boxed{\twocell{(1 *0 mu *0 2) *1 aleph}}}
& {\twocell{(2 *0 mu *0 1) *1 (1 *0 mu *0 1) *1 (1 *0 mu) *1 mu}}
	\ar@3 [dr] ^-{\twocell{(2 *0 mu *0 1) *1 (1 *0 alpha) *1 mu}}
		_-{}="src1"
\\
{\twocell{(1 *0 mu *0 2) *1 (mu *0 2) *1 (mu *0 1) *1 mu}}
	\ar@3 [u] ^-{\twocell{(1 *0 mu *0 2) *1 (alpha *0 1) *1 mu}}
	\ar@3 [r] |-{\twocell{(1 *0 mu *0 2) *1 (mu *0 2) *1 alpha}}
& {\twocell{(1 *0 mu *0 2) *1 (mu *0 mu) *1 mu}}
	\ar@3 [r] |-{\twocell{(1 *0 mu *0 mu) *1 alpha}}
& {\twocell{(1 *0 mu *0 mu) *1 (1 *0 mu) *1 mu}}
	\ar@3 [dr] |-{\twocell{(3 *0 mu) *1 (1 *0 alpha) *1 mu}}
	\ar@{} "src1"; |-{\boxed{\twocell{(1 *0 aleph) *1 mu}}}
	\ar@{} [d] _-{\boxed{\twocell{(3 *0 mu) *1 aleph}}}
& {\twocell{(2 *0 mu *0 1) *1 (2 *0 mu) *1 (1 *0 mu) *1 mu}}
	\ar@3 [d] ^-{\twocell{(2 *0 alpha) *1 (1 *0 mu) *1 mu}}
\\
{\twocell{(mu *0 3) *1 (mu *0 2) *1 (mu *0 1) *1 mu}}
	\ar@3 [u] ^-{\twocell{(alpha *0 2) *1 (mu *0 1) *1 mu}}
	\ar@{} [ur] |-{\copyright}
	\ar@3 [r] _-{\twocell{(mu *0 3) *1 (mu *0 2) *1 alpha}}
& {\twocell{(mu *0 3) *1 (mu *0 mu) *1 mu}}
	\ar@3 [r] _-{\twocell{(mu *0 1 *0 mu) *1 alpha}}
	\ar@3 [u] |-{\twocell{(alpha *0 mu) *1 mu}}
& {\twocell{(3 *0 mu) *1 (mu *0 mu) *1 mu}}
	\ar@3 [r] _-{\twocell{(3 *0 mu) *1 (2 *0 mu) *1 alpha}}
& {\twocell{(3 *0 mu) *1 (2 *0 mu) *1 (1 *0 mu) *1 mu}}
}
}
\]
The corresponding generator of the natural system $h_4(\As)$ of homological $4$-syzygies of $\As$ is:
\begin{align*}
\delta_5[\omega_b] 
	\:&=\: \left[ \twocell{(aleph *0 1) *1 mu} \right]
		+ \left[ \twocell{(2 *0 mu *0 1) *1 aleph} \right]
		+ \left[ \twocell{(mu *0 3) *1 aleph} \right]
		- \left[ \twocell{(1 *0 aleph) *1 mu} \right]
		- \left[ \twocell{(1 *0 mu *0 2) *1 aleph} \right]
		- \left[ \twocell{(3 *0 mu) *1 aleph} \right] \\
	\:&=\: \left[ \twocell{aleph} \right] a - a \left[ \twocell{aleph} \right].
\end{align*}

%%%%%%%%%%%%%%%%%%%%%%%%%%%%%%%%
\subsection{The category \pdf{\Epi}}
%%%%%%%%%%%%%%%%%%%%%%%%%%%%%%%%

We denote by $\Epi$ the subcategory of the simplicial category whose $0$-cells are the natural numbers and whose morphisms from $n$ to $p$ are the monotone surjections from $\ens{0,\dots,n}$ to $\ens{0,\dots,p}$. This category is studied in~\cite{LivernetRichter10}, where it is denoted by $\Delta^{\text{epi}}$. 

The category $\Epi$ admits a presentation by the (infinite) $2$-polygraph $\Sigma$ with the natural numbers as $0$-cells, with $1$-cells
\[
\xymatrix{
n+1
	\ar [rr] ^-{s_i}
&& n
&&
0\leq i\leq n,
}
\]
where $s_i$ represents the map 
\[
s_i (j) \:=\: 
\begin{cases}
j &\text{if } 0\leq j\leq i, \\
j-1 &\text{if } i+1 \leq j\leq n+1,
\end{cases}
\]
and with $2$-cells 
\[
\xymatrix@W=2em@R=1em{
& n+1
	\ar@/^/ [dr] ^-{s_j} 
	\ar@2 []!<0pt,-15pt>;[dd]!<0pt,15pt> ^-{s_{i,j}}
\\
n+2 
	\ar@/^/ [ur] ^-{s_i}
	\ar@/_/ [dr] _-{s_{j+1}}
&& n
&& 0\leq i\leq j\leq n + 1.
\\
& n+1 
	\ar@/_/ [ur] _-{s_i}
}
\]
Let us prove that this $2$-polygraph is convergent. For termination, given a $1$-cell $u=s_{i_1}\dots s_{i_k}$ of $\Sigma^*$, we define the natural number $\nu(u)$ as the number of pairs $(i_p,i_q)$ such that $i_p\leq i_q$, with $1\leq p<q\leq k$. In particular, we have $\nu(s_i s_j) = 1$ and $\nu(s_{j+1} s_i)=0$ when $i\leq j$, \ie, when $s_is_j$ is the source and $s_{j+1}s_i$ is the target of a $2$-cell of $\Sigma$. Moreover, we have $\nu(w u w') > \nu(w v w')$ when $\nu(u)>\nu(v)$ holds. Thus, for every non-identity $2$-cell $f:u\dfl v$ in $\Sigma^*$, the strict inequality $\nu(u)>\nu(v)$ is satisfied, yielding termination of $\Sigma$.

The $2$-polygraph $\Sigma$ has one critical branching $(s_{i,j} s_k, s_i s_{j,k})$ for every possible $0\leq i\leq j\leq k\leq n+2$ and it is confluent, so that we get a partial resolution of $\Epi$ by the acyclic $(3,1)$-polygraph $\crit_2(\Sigma)$ obtained by extending $\Sigma$ with all the $3$-cells filling the confluence diagrams associated to the critical branchings:
\[
\xymatrix@C=5em@R=3em{
& s_{j+1} s_i s_k 
	\ar@2 [r] ^-{s_{j+1} s_{i,k}} _{}="src"
& s_{j+1} s_{k+1} s_i
	\ar@2 @/^/ [dr] ^-{s_{j+1,k+1} s_i}
\\
s_i s_j s_k 
	\ar@2 @/^/ [ur] ^-{s_{i,j} s_k} 
	\ar@2 @/_/ [dr] _-{s_i s_{j,k}}
&&& s_{k+2} s_{j+1} s_i.
\\
& s_i s_{k+1} s_j
	\ar@2 [r] _-{s_{i,k+1} s_j} ^{}="tgt"
& s_{k+2} s_i s_j
	\ar@2 @/_/ [ur] _-{s_{k+2} s_{i,j}}
\ar@3 "src"!<0pt,-25pt>;"tgt"!<0pt,25pt> ^-{s_{i,j,k}}
}
\]
To simplify notations, we draw each $1$-cell $s_i$ as a vertical string $\:\twocell{1}_i\:$, each $2$-cell $s_{i,j}$ as $\twocell{tau}_{i,j}$, so that each $3$-cell $s_{i,j,k}$ has the same shape as the Yang-Baxter relation, or permutohedron of order~$3$:
\[
\xymatrix{
{\twocell{(tau *0 1) *1 (1 *0 tau) *1 (tau *0 1)}}_{i,j,k}
	\ar@3 [rr] ^-{\twocell{yb}_{i,j,k}}
&& {\twocell{(1 *0 tau) *1 (tau *0 1) *1 (1 *0 tau)}_{i,j,k}}
}
\]
Using those notations, the natural system $h_2(\Sigma)$ of homological $2$-syzygies of $\Sigma$ is generated by the elements
\[
\delta_3\left[\twocell{yb}_{i,j,k}\right] 
	\:=\: \left[\twocell{(tau *0 1) *1 (1 *0 tau) *1 (tau *0 1)}_{i,j,k} \right] 
		- \left[\twocell{(1 *0 tau) *1 (tau *0 1) *1 (1 *0 tau)}_{i,j,k} \right] \\
	\:=\: 
\left\{
\begin{array}{c l}
& \big([\twocell{tau}_{i,j}] s_k - s_{k+2} [\twocell{tau}_{i,j}]\big) \vspace{1ex} \\
		+\: & \big(s_{j+1} [\twocell{tau}_{i,k}] - [\twocell{tau}_{i,k+1}] s_j\big)  \vspace{1ex} \\
		+\: & \big([\twocell{tau}_{j+1,k+1}] s_i - s_i [\twocell{tau}_{j,k}]\big) .
\end{array}
\right.
\]
The $2$-polygraph $\Sigma$ has one critical triple branching
\[
\big( \: s_{i,j} s_k s_l, \: s_i s_{j,k} s_l, \: s_i s_j s_{k,l} \: \big)
\]
for every possible $0\leq i\leq j\leq k\leq l\leq n+3$. This yields a partial resolution of $\Epi$ by the acyclic  $(4,1)$-polygraph $\crit_3(\Sigma)$ with one $4$-cell $s_{i,j,k,l}$ for every possible $0\leq i\leq j\leq k\leq l \leq n+3$. In string diagrams, omitting the subscripts, each critical triple branching is written
\[
\big(\: 
	\twocell{tau}\:\twocell{1}\:\twocell{1} \:,\: 
	\twocell{1}\:\twocell{tau}\:\twocell{1} \:,\:
	\twocell{1}\:\twocell{1}\:\twocell{tau}
\:\big).
\]
With the same conventions, the corresponding $4$-cell has the shape of the permutohedron of order~$4$:
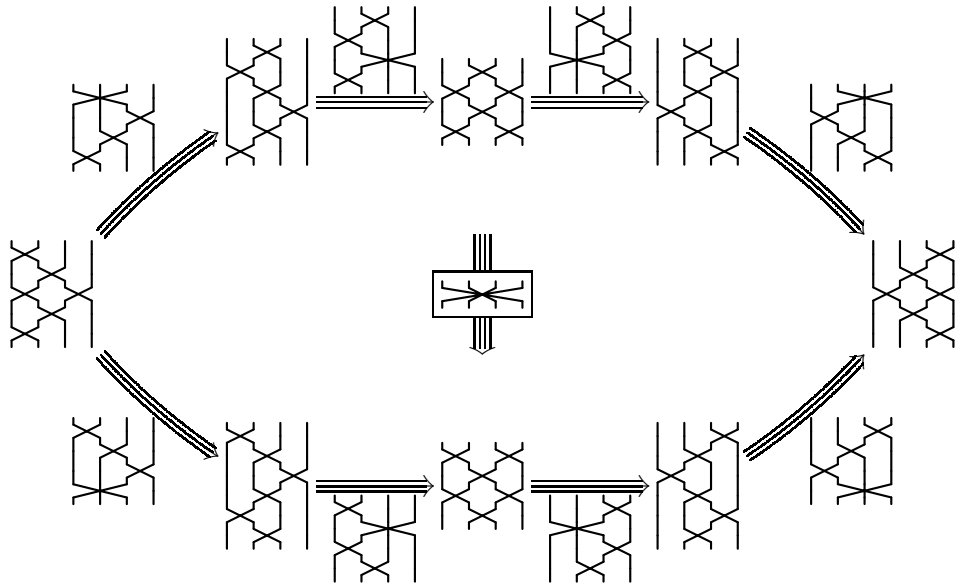
\[
\xymatrix @C=4em@R=2em{
& {\twocell{(1 *0 tau *0 1) *1 (tau *0 2) *1 (1 *0 tau *0 1) *1 (2 *0 tau) *1 (1 *0 tau *0 1) *1 (tau *0 2)}}
	\ar@3 [r] ^-{\twocell{(1 *0 tau *0 1) *1 (tau *0 2) *1 (1 *0 yb) *1 (tau *0 2)}}
& {\twocell{(1 *0 tau *0 1) *1 (tau *0 tau) *1 (1 *0 tau *0 1) *1 (tau *0 tau)}}
	\ar@3 [r] ^-{\twocell{(1 *0 tau *0 1) *1 (2 *0 tau) *1 (yb *0 1) *1 (2 *0 tau)}}
	\ar@4 []!<0pt,-50pt>;[dd]!<0pt,50pt> |-{\boxed{\twocell{yb2}}}
& {\twocell{(1 *0 tau *0 1) *1 (2 *0 tau) *1 (1 *0 tau *0 1) *1 (tau *0 2) *1 (1 *0 tau *0 1) *1 (2 *0 tau)}}
	\ar@3 @/^/ [dr] ^-{\twocell{(1 *0 yb) *1 (tau *0 2) *1 (1 *0 tau *0 1) *1 (2 *0 tau)}}
\\
{\twocell{(tau *0 2) *1 (1 *0 tau *0 1) *1 (tau *0 tau) *1 (1 *0 tau *0 1) *1 (tau *0 2)}}
	\ar@3 @/^/ [ur] ^-{\twocell{(yb *0 1) *1 (2 *0 tau) *1 (1 *0 tau *0 1) *1 (tau *0 2)}}
	\ar@3 @/_/ [dr] _-{\twocell{(tau *0 2) *1 (1 *0 tau *0 1) *1 (2 *0 tau) *1 (yb *0 1)}}
&&&& {\twocell{(2 *0 tau) *1 (1 *0 tau *0 1) *1 (tau *0 tau) *1 (1 *0 tau *0 1) *1 (2 *0 tau)}}
\\
& {\twocell{(tau *0 2) *1 (1 *0 tau *0 1) *1 (2 *0 tau) *1 (1 *0 tau *0 1) *1 (tau *0 2) *1 (1 *0 tau *0 1)}}
	\ar@3 [r] _-{\twocell{(tau *0 2) *1 (1 *0 yb) *1 (tau *0 2) *1 (1 *0 tau *0 1)}}
& {\twocell{(tau *0 tau) *1 (1 *0 tau *0 1) *1 (tau *0 tau) *1 (1 *0 tau *0 1)}}
	\ar@3 [r] _-{\twocell{(2 *0 tau) *1 (yb *0 1) *1 (2 *0 tau) *1 (1 *0 tau *0 1)}}
& {\twocell{(2 *0 tau) *1 (1 *0 tau *0 1) *1 (tau *0 2) *1 (1 *0 tau *0 1) *1 (2 *0 tau) *1 (1 *0 tau *0 1)}}
	\ar@3 @/_/ [ur] _-{\twocell{(2 *0 tau) *1 (1 *0 tau *0 1) *1 (tau *0 2) *1 (1 *0 yb)}}
}
\]
As usual, the elements $\delta_4 \left[\twocell{yb2}_{i,j,k,l} \right]$, for $0\leq i\leq j\leq k\leq k\leq n+3$, form a generating set for the natural system $\h_3(\Sigma)$ of homological $3$-syzygies of $\Sigma$. 

More generally, this construction extends to a polygraphic resolution of $\Epi$ where, for every natural number~$n$, the generating $n$-cell has the shape of a permutohedron of order~$n$.

%%%%%%%%%%%%%%%%%%%%%%%%%%%%%%%%%%%%%%%%%%%%%%%%%%%%%%%%
% Bibliographie
%%%%%%%%%%%%%%%%%%%%%%%%%%%%%%%%%%%%%%%%%%%%%%%%%%%%%%%%
\begin{small}
\renewcommand{\refname}{\textbf{References}}
\bibliographystyle{amsplain}
\bibliography{bibliographie}
\end{small}

\end{document}

%% file: macroPY.tex
\UseTips
\SelectTips{eu}{11}

\newdir{ >}{{}*!/-10pt/@{>}}
\newdir{ -}{{}*!/-10pt/@{}}
\newdir{> }{{}*!/+10pt/@{>}}

\makeatletter

\xyletcsnamecsname@{dir4{}}{dir{}}
\xydefcsname@{dir4{-}}{\line@ \quadruple@\xydashh@}
\xydefcsname@{dir4{.}}{\point@ \quadruple@\xydashh@}
\xydefcsname@{dir4{~}}{\squiggle@ \quadruple@\xybsqlh@}
\xydefcsname@{dir4{>}}{\Tttip@}
\xydefcsname@{dir4{<}}{\reverseDirection@\Tttip@}

\xydef@\quadruple@#1{%
	\edef\Drop@@{%
		\dimen@=#1\relax
		\dimen@=.5\dimen@
		\A@=-\sinDirection\dimen@
		\B@=\cosDirection\dimen@
		\setboxz@h{%
			\setbox2=\hbox{\kern3\A@\raise3\B@\copy\z@}%
			\dp2=\z@ \ht2=\z@ \wd2=\z@ \box2
			\setbox2=\hbox{\kern\A@\raise\B@\copy\z@}%
			\dp2=\z@ \ht2=\z@ \wd2=\z@ \box2
			\setbox2=\hbox{\kern-\A@\raise-\B@\copy\z@}%
			\dp2=\z@ \ht2=\z@ \wd2=\z@ \box2
			\setbox2=\hbox{\kern-3\A@\raise-3\B@ \noexpand\boxz@}%
			\dp2=\z@ \ht2=\z@ \wd2=\z@ \box2
		}%
		\ht\z@=\z@ \dp\z@=\z@ \wd\z@=\z@ \noexpand\styledboxz@
	}%
}

\xydef@\Tttip@{\kern2pt \vrule height2pt depth2pt width\z@
	\Tttip@@ \kern2pt \egroup
	\U@c=0pt \D@c=0pt \L@c=0pt \R@c=0pt \Edge@c={\circleEdge}%
	\def\Leftness@{.5}\def\Upness@{.5}%
	\def\Drop@@{\styledboxz@}\def\Connect@@{\straight@{\dottedSpread@\jot}}}
	
\xydef@\Tttip@@{%
	\dimen@=.25\dimen@
 	\B@=\cosDirection\dimen@
	\setboxz@h\bgroup\reverseDirection@\line@ \wdz@=\z@ \ht\z@=\z@ \dp\z@=\z@
	{\vDirection@(1,-1)\xydashl@ \xyatipfont\char\DirectionChar}%
	{\vDirection@(1,+1)\xydashl@ \xybtipfont\char\DirectionChar}%
}

\xydef@\ar@form{
	\ifx \space@\next \expandafter\DN@\space{\xyFN@\ar@form}%
	\else\ifx ^\next \DN@ ^{\xyFN@\ar@style}\edef\arvariant@@{\string^}%
	\else\ifx _\next \DN@ _{\xyFN@\ar@style}\edef\arvariant@@{\string_}%
	\else\ifx 0\next \DN@ 0{\xyFN@\ar@style}\def\arvariant@@{0}%
	\else\ifx 1\next \DN@ 1{\xyFN@\ar@style}\def\arvariant@@{1}%
	\else\ifx 2\next \DN@ 2{\xyFN@\ar@style}\def\arvariant@@{2}%
	\else\ifx 3\next \DN@ 3{\xyFN@\ar@style}\def\arvariant@@{3}%
	\else\ifx 4\next \DN@ 4{\xyFN@\ar@style}\def\arvariant@@{4}%
	\else\ifx \bgroup\next \let\next@=\ar@style
	\else\ifx [\next \DN@[##1]{\ar@modifiers{[##1]}}%]
	\else\ifx *\next \DN@ *{\ar@modifiers}%
	\else\addLT@\ifx\next \let\next@=\ar@slide
	\else\ifx /\next \let\next@=\ar@curveslash
	\else\ifx (\next \let\next@=\ar@curveinout %)
	\else\addRQ@\ifx\next \addRQ@\DN@{\ar@curve@}%
	\else\addLQ@\ifx\next \addLQ@\DN@{\xyFN@\ar@curve}%
	\else\addDASH@\ifx\next \addDASH@\DN@{\defarstem@-\xyFN@\ar@}%
	\else\addEQ@\ifx\next \addEQ@\DN@{\def\arvariant@@{2}\defarstem@-\xyFN@\ar@}%
	\else\addDOT@\ifx\next \addDOT@\DN@{\defarstem@.\xyFN@\ar@}%
	\else\ifx :\next \DN@:{\def\arvariant@@{2}\defarstem@.\xyFN@\ar@}%
	\else\ifx ~\next \DN@~{\defarstem@~\xyFN@\ar@}%
	\else\ifx !\next \DN@!{\dasharstem@\xyFN@\ar@}%
	\else\ifx ?\next \DN@?{\ar@upsidedown\xyFN@\ar@}%
	\else \let\next@=\ar@error
	\fi\fi\fi\fi\fi\fi\fi\fi\fi\fi\fi\fi\fi\fi\fi\fi\fi\fi\fi\fi\fi\fi\fi \next@}

\makeatother

\newcommand{\fl}{\to}

\newcommand{\dfl}{\Rightarrow}
\newcommand{\tfl}{\Rrightarrow}

\newcommand{\ens}[1]{\left\{#1\right\}}
\DeclareMathOperator{\id}{Id}

\newcommand{\tens}{\otimes}
\newcommand{\cl}[1]{\overline{#1}}

\newcommand{\iar}[1]{\left\lfloor#1\right\rfloor}

\newcommand{\ie}{\emph{i.e.}}

\renewcommand{\phi}{\varphi}
\renewcommand{\epsilon}{\varepsilon}

\newcommand{\Nb}{\mathbb{N}}

\newcommand{\Zb}{\mathbb{Z}}

\newcommand{\dr}{\partial}

\newcommand{\Br}{\EuScript{B}}
\newcommand{\Cr}{\EuScript{C}}
\newcommand{\Dr}{\EuScript{D}}
\newcommand{\Er}{\EuScript{E}}

\def\c{\mathbf{C}}

\newcommand{\Ab}{\mathbf{Ab}}

\newcommand{\Mod}[1]{\mathbf{Mod}(#1)}

\newcommand{\Cat}[1]{\mathbf{Cat}_{#1}}

\newcommand{\Ct}[1]{\mathbf{C}{#1}}

\newcommand{\ho}{\mathrm{H}}
\renewcommand{\ker}{\mathrm{Ker}\;}

\newcommand{\Tor}{\mathrm{Tor}}

\definecolor{orange}{rgb}{1,0.55,0}